%% file: trackingMPC15_arxiv.tex
\documentclass{ifacconf}

\usepackage{graphicx}      
\usepackage{natbib}        
\usepackage{tikz}
\usepackage{pgfplots}
\usetikzlibrary{shapes,arrows}
 \usepgfplotslibrary{fillbetween}
\usetikzlibrary{arrows,calc,positioning,bending}
\usepackage{amssymb}
\usepackage{amsmath}
\usepackage{algorithm}
\usepackage[noend]{algpseudocode}
\usepackage{mathrsfs}
\usepackage{optidef}
\DeclareMathAlphabet{\mathpzc}{OT1}{pzc}{m}{it}

\newtheorem{task}{Task}
\newtheorem{defi}{Definition}
\newtheorem{ass}{Assumption}
\newtheorem{propo}{Proposition}
\newtheorem{lemm}{Lemma}
\newtheorem{example}{Example}
\newtheorem{proof}{Proof}
\newcommand*{\Rn}[1]{\ensuremath{\mathbb{R}^{n_{#1}}}}
\newcommand*{\R}{\ensuremath{\mathbb{R}}}
\newcommand*{\N}{\ensuremath{\mathbb{N}}}
\DeclareMathOperator*{\argmin}{arg\,min} 

\makeatletter
\def\BState{\State\hskip-\ALG@thistlm}
\makeatother
\newcommand{\algorithmicbreak}{\textbf{break}}
\newcommand{\Break}{\State \algorithmicbreak}
\begin{document}
\begin{frontmatter}

\title{Constrained Gaussian Process Learning for Model Predictive Control \thanksref{footnoteinfo}} 

\thanks[footnoteinfo]{The work of this paper is supported by the Federal Ministry of Education and Research within the Forschungscampus \textit{STIMULATE} under grant number ‘13GW0095A’.}

\author[First]{Janine Matschek} 
\author[Second]{Andreas Himmel} 
\author[Second,Third]{Kai Sundmacher}
\author[First]{Rolf Findeisen}

\address[First]{{Laboratory for Systems Theory and Automatic Control, Otto-von-Guericke University Magdeburg, Germany  \\ (e-mail: \{janine.matschek, rolf.findeisen\}@ovgu.de).}}
\address[Second]{Process Systems Engineering, Otto-von-Guericke University Magdeburg, Germany}
\address[Third]{Max Planck Institute for Dynamics of Complex Technical Systems, Magdeburg,  Germany}

\begin{abstract}                
Many control tasks can be formulated as a tracking problem of a known or unknown reference signal. Examples are movement compensation in collaborative robotics, the synchronisation of oscillations for power systems or reference tracking of recipes in chemical process operation. Tracking performance as well as guaranteeing stability of the closed loop strongly depends on two factors: Firstly, it depends on whether the future desired tracking reference signal is known and, secondly, whether the system can track the reference at all. 
This paper shows how to use machine learning, i.e.\ Gaussian processes, to learn a reference from (noisy) data, while  guaranteeing trackability of the modified desired reference predictions in the framework of model predictive control. Guarantees are provided by adjusting the hyperparameters via a constrained optimization.
Two specific scenarios, i.e. asymptotically constant and  periodical references, are discussed.
\end{abstract}


\end{frontmatter}
\section{Introduction}

Model predictive control (MPC) is a popular optimization based control strategy which can handle a broad class of dynamical systems including nonlinear, constrained, multi-input-multi-output systems. 
MPC can be used for different control task, including setpoint stabilisation, tracking of time dependent references, path following or economical operation of a system, see for example \citep{matschek2019nonlinear}. To guarantee repeated feasibility of the optimal control problem as well as to achieve closed loop stability several concepts exist, which differ depending on the control tasks. In tracking MPC the controller can e.g. be designed in error coordinates which leads to time-varying error dynamics and consequently time-dependent terminal ingredients to prove stability \citep{faulwasser2011model}. To do so, the reference needs be known and trackable for the system, i.e. it must be compliant with the state constraints and an admissible reference input should exist to follow the reference given the system dynamics. Alternatively, one can use artificial references  \citep{limon2008mpc, limon2012mpc,ferramosca2009mpc} to ensure feasibility under changing references. Hereby, the system state is steered to follow the artificial reference while the distance of the artificial reference to the actual reference is minimized. While this does not lead to exact tracking of the non-trackable reference, it results in an additional degree of freedom for the controller to modify the reference. Reference modification to achieve good performance and stability can also be achieved by reference governors, which act as pre-filters for the reference signal, see e.g. \citep{garone2017reference} and references therein. Based on the current system state and reference value, a reference governor modifies the reference whenever a constraint violation in the controlled system might occur.

We aim at obtaining prediction models which approximate the reference evolution based on past observations with machine learning that can be used in predictive control for tracking. At the same time, we want to guarantee reachability and trackability of the learned reference by including constraints in the machine learning procedure. 
Applications in which the reference is given only in terms of data/observations  and should be modelled or predicted to be available to the controller are diverse. Examples are dynamically operated chemical plants where references are obtained via real time optimization or autonomous cars which learn from and adapt to human driver provided references. Gaussian processes were used in \cite{maiworm2018two} as a feedforward controller for quantum dot microscopy. In \citep{Klenske2016} Gaussian processes are used to provide external signals to a model predictive controller  correcting the orientation of an astronomic telescope, while in \citep{matschek2020} breathing motions in minimally invasive surgery are modelled via Gaussian processes and are provided to a predictive motion compensation  controller.  
Here, we show how Gaussian processes can be trained to model an external reference signal based on (noisy) measurement data and extrapolate its evolution into the future while guaranteeing trackability of the reference for the controlled system, i.e. satisfying state and input constraints.
Even if the underlying, unknown reference which we want to predict is trackable, uncertainties as e.g.\ measurement noise might lead to a loss of that property, such that usage of unfiltered reference measurements would lead to infeasibility of the control problem. Using GPs with constraints in the learning can restore trackability in such cases. Moreover, adding  constraints can improve the approximation quality of the GP, as unrealistic evolutions are excluded. In case that the original reference is not trackable, the constrained GP allows to find a tradeoff between close approximation and constraint satisfaction. Thus the GP serves both for reference prediction and adaptation of it whenever necessary, see also Figure~\ref{fig:setup}.

\tikzset{
block/.style = {draw, fill=white, rectangle, minimum height=1.5cm, minimum width=1.3cm, align=center},
tmp/.style  = {coordinate}, 
sum/.style= {draw, fill=white, circle, minimum size=0.5cm},
input/.style = {coordinate},
output/.style= {coordinate},
pinstyle/.style = {pin edge={to-,thin,black},
>=latex
}
}

\begin{figure}[t]
\centering
\begin{tikzpicture}[>=latex']

 \node [input] (input){};
    \node [block, right =0.8cm of input] (GP)
    		{GP \\ Reference\\ Generator};
    \node [block, right =2.4cm of GP] (MPC) 
    		{ Controller};
   
    \draw [->] (input) --  (GP);
    \draw [->] (GP) --  (MPC);

\node[align=center] at (3.7,-0.4) {learned \\ reference};
\node[align=center] at (-0.3,-0.4) {provided \\ data ($\times$)};

 \draw[fill=red!20!white, draw=none] plot coordinates{(-1.5, 1.2) (0.5, 1.2) (0.5,1.6) (-1.5,1.6)};

\draw [->] plot coordinates{(-1.5,0.5) (0.5,0.5) };
\draw [->] plot coordinates{(-1.5,0.5) (-1.5,1.7) };
\draw [] plot coordinates{(-0.5,0.4) (-0.5,0.6) };
\node at (-0.5,0.25) {$\tau$};
 \draw [densely dashed,black!50] plot [ smooth, tension=0.8] coordinates{(-1.5,0.7)  (-1,1.3) (-0.5,0.9) (0,1) (0.4,0.8) };

 \draw [mark=x,mark size=2pt, draw=none] plot  coordinates{(-1.4,0.95) (-1.15,1.05) (-1,1.4) (-0.8,1.15)  (-0.5,0.95)   };


 \draw[fill=red!20!white, draw=none] plot coordinates{(2.8, 1.2) (4.8, 1.2) (4.8,1.6) (2.8,1.6)};

\draw [->] plot coordinates{(2.8,0.5) (4.8,0.5) };
\draw [->] plot coordinates{(2.8,0.5) (2.8,1.7) };
\draw [] plot coordinates{(3.8,0.4) (3.8,0.6) };
\node at (3.8,0.25) {$\tau$};

 \draw [black] plot [ smooth, tension=0.8] coordinates{(2.8,0.7)  (3.3,1.2) (3.8,0.9) (4.3,1)  (4.7,0.8) };

\end{tikzpicture}
\caption{Based on data, denoted by $\times$, the unknown reference (dashed) should be modelled by the GP. It learns a reference (solid) while both prediction into the future ($t>\tau$) as well as reference shaping to satisfy the constraints (red) is achieved.}
\label{fig:setup}
\end{figure}
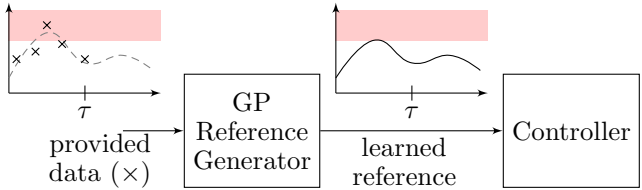

The main contributions of this paper is a guideline how to setup and train GPs to be used as reference predictors with guaranteed trackability of the learned reference. These learning algorithms utilize special structures of the underlying reference which should be modelled, as well as constrained hyperparameter estimation. In contrast to \citep{daVeiga2012gaussian} we do not use truncated multinormal distributions, but constrain the predicted mean of the GP to lie inside the reachable tube and the state constraints during hyperparameter estimation.

The remainder of this paper is structured as follows: Section~2 describes the problem setup of learning based reference prediction via Gaussian processes. Section~3 proposes algorithms for Gaussian processes training to guarantee trackability of the reference. 
Section~4 summarises the achievements and provides directions for future work.

\section{problem setup}

Consider the nonlinear time-discrete system
\begin{equation}
x(k+1)=f(x(k),u(k)), \quad x(0)=x_0,
\label{eq:system}
\end{equation}
where $x \in \Rn{x}$ is the state, $u  \in \Rn{u}$ is the input, and $x_0 \in \Rn{x}$ is the initial condition of the system.

In tracking MPC, the goal is to design a controller such that the system state follows a  reference $x_\text{r}$ while satisfying state constraints $\mathcal X$ and input constraints $\mathcal U$. 
One possibility to guarantee stability of tracking MPC is the use of time varying terminal equality or inequality constraints which depend on the reference \citep{faulwasser2011model, rawlings2017model, matschek2019nonlinear}. Determination of these terminal ingredient consequently requires the knowledge of the reference $x_r$, which is however not always a priori known.  
For example, dynamic operation of chemical plants might lead to sudden changes in the reference based on economical considerations. 
Other examples are autonomous vehicles following a human driver (e.g. adapting to its velocity) while not knowing its future decisions. 
In such cases, machine learning can be used to obtain a  model of the reference, such that the reference can be predicted and thus is known to the controller (including predictions of future values). 
Besides the required knowledge of the reference, the reference must fulfil the following property:
\begin{defi}(Trackability).\\
A reference $x_\text{r}: \N_0 \to \Rn{x} $ is said to be trackable for system~\eqref{eq:system} if it fulfils the state constraints $x_\text{r}(k)\in  \mathcal X$ and can be followed given the system dynamics $\exists  u_\text{r}(k) \in \mathcal{U} $ such that $ x_\text{r}(k+1)=f(x_\text{r}(k),u_\text{r}(k))$ for all $ k \in \N_0$. 
\label{def:reach}
\end{defi}
Though restrictive, trackability of the reference according to Definition~\ref{def:reach} enables desirable properties such as recursive feasibility of an MPC with terminal equality constraints (once being on the reference there exists an admissible input to stay on it).

This definition allows us to formulate the task to use machine learning (in our case Gaussian processes) to design and learn a suitable reference from data:

\begin{task}(Reference generator).\\
Given system~\eqref{eq:system} and data/  measurements $\mathcal D:=\prod\limits_{i=0}^{n} \R^+_0 \times \Rn{x} $ describing the desired reference $r:\N_0\to \Rn{x}$. Design a reference generator $g: \N_0 \times \mathcal D \to \mathcal X $, which provides for $D\in \mathcal D$ a reference $x_\text{r}: \N_0 \to \mathcal X, (k) \mapsto x_\text{r}(k):=g(k,D)$,  which fulfils:
\begin{enumerate}
\item Trackability: The reference $x_\text{r}$ is trackable.
\item Reference Prediction: The reference $x_\text{r}$ spans at least over a receding prediction horizon $N$, i.e.  at $k, x_\text{r}(i)$ is known $\forall i \in \{k,k+1,\ldots,k+N\}$. 
\item Data fitting: The reference model finds a trade off between model complexity and data consistency, i.e. $x_\text{r}(k) \approx r(k)$ for $D\in \mathcal D$.
\end{enumerate}
\label{task:ref}
\end{task}

To address this Task we propose to use a machine learning technique called Gaussian processes. 
We show how to guarantee trackability of the learned reference by introducing additional constraints in the learning phase (constrained hyperparameter optimization). For this purpose a general introduction to GPs and an elaboration on the use of them as reference generators is provided first.

\subsection{Gaussian processes}


Gaussian Processes are stochastic modelling approaches which can be used for classification and regression problems. In control they have gained an increasing attention for the modelling of both static and dynamic systems, see e.g.~\citep{Rasmussen2006, Ostafew2016,  kocijan2004gaussian, berkenkamp2015safe}.
Reasons for this popularity are the limited amount of design decisions, their capability of dealing with noisy data, and the confidence interval that is provided by the GP which allows to investigate the quality of the obtained model.

We will use GPs to obtain the desired reference generator.  The Gaussian process is uniquely defined by a mean function $m: \mathbb{R} \to \mathbb{R} $ and a symmetric, positive semi-definite covariance function $\kappa: \mathbb{R}\times \mathbb{R} \to \mathbb{R}_0^+$ and is denoted by 
\begin{equation*}
y(t) \sim \mathcal{GP} (m(t),\kappa(t,t')).
\end{equation*} 
Here, $t, t' \in \mathbb{R}$ are the regressors or inputs to the GP and the values of the process $y(t)$ at each specific time $t$ possess a normal distribution.
Via the covariance function (also called kernel) a GP relates similarities between the input variables to the similarity between the output variables. 
These mean and covariance functions involve  hyperparameters $\theta \in \Rn{\theta}$, where $n_\theta$ depends on the selected functions $m$ and $\kappa$. The values of the hyperparameters  can be learned based on a hyperparameter training set $D_\theta:=\big\{ \left(t_{\theta,i},y_{\theta,i}\right) \in \R_0^+ \times \Rn{x}\,|\, i=1,2,\ldots,n_{D_\theta}\big\}\in \mathcal D$.
To do so, often a point estimate of the hyperparameters is calculated via the maximization of the marginal logarithmic likelihood.

Our overall goal is to predict or infer the distribution of the output at (possibly unseen) test points $t_*$.
This prediction is based on several design decisions, the hyperparameters, and 
 a training data set  $D_\text{t}:=\big\{ \left(t_{\text{t},i},y_{\text{t},i}\right) \in \R_0^+ \times \Rn{x}\,|\, i=1,2,\ldots,n_{D_\text{t}}\big\}\in \mathcal D$. 
 For ease of notation, we define
  \begin{subequations}
 \begin{align}
     \boldsymbol{t}&:=[t_{\text{t},1},\ldots, t_{\text{t},n_{D_\text{t}}} ],\\ \boldsymbol{y}&:=[y_{\text{t},1},\ldots, y_{\text{t},n_{D_\text{t}}} ],\\  \boldsymbol{m}(\boldsymbol{t})&:=[m(t_{\text{t},1}), \ldots,m(t_{\text{t},n_{D_\text{t}}}) ]
 \end{align}\label{eq:bold_data}
    \end{subequations}
The joint distribution of the training data output $\boldsymbol{y}$ and the test data output $y_*$ at $t_*$ can be expressed as
\begin{equation*}
\begin{pmatrix}
\boldsymbol{y}\\y_*
\end{pmatrix} \sim \mathcal{GP}\left(   \begin{pmatrix}
\boldsymbol{m}(\boldsymbol{t})\\m(t_*)
\end{pmatrix}  , \begin{pmatrix}
K(\boldsymbol{t},\boldsymbol{t})+\sigma^2_n I & K(\boldsymbol{t},t_*)\\ K(t_*,\boldsymbol{t}) & \kappa(t_*,t_*)
\end{pmatrix}\right).
\end{equation*}
Here, $\sigma^2_n$ represents the variance of the measurement noise. The entries of the covariance matrix $K$ are calculated using the covariance function $ \kappa$. Specifically, $K(\boldsymbol{t},\boldsymbol{t})$ is of dimension $n_{D_\text{t}}\times n_{D_\text{t}}$ and specifies the covariance between all of the training data points, while $ K(\boldsymbol{t},t_*)$ and  $ K(t_*,\boldsymbol{t})$ (with dimensions $n_{D_\text{t}}\times 1$ and $1 \times n_{D_\text{t}}$, respectively) define the cross correlation between test and training data points. The scalar $ \kappa(t_*,t_*)$ is the auto covariance of the test data. 
Given this joint probability distribution, the conditional posterior distribution can be calculated via the posterior mean function $m^+:\mathbb{R} \to \mathbb{R}$ defined by
\begin{align}
m^+(t_*):=& m(t_*)\nonumber\\
&+K(t_*,\boldsymbol{t})(K(\boldsymbol{t},\boldsymbol{t})+\sigma^2_n I)^{-1}\big(\boldsymbol{y}-\boldsymbol{m}(\boldsymbol{t})\big)
\label{eq:mean}
\end{align}
 and the posterior covariance $\kappa^+: \mathbb{R}\times \mathbb{R} \to \mathbb{R}_0^+$ defined by
 \begin{align*}
 \kappa^+(t_*,t_*):=&\kappa(t_*,t_*)\\
 &-K(t_*,\boldsymbol{t})\big(K(\boldsymbol{t},\boldsymbol{t})+\sigma_n^2 I\big)^{-1}K(\boldsymbol{t},t_*).
 \end{align*}
 
These posterior moments clearly depend on the involved data set $D=D_\theta \cup D_\text{t}$ and the hyperparameter $\theta$. Therefore, whenever necessary we will explicitly denote this dependency by $m^+(t_*| D,\theta)$ and  $\kappa^+(t_*,t_*| D, \theta)$, but refrain from using it elsewhere for sake of brevity of notation.
The aforementioned design decisions involved in GP modelling include the selection of the prior  mean and covariance functions $m$ and $\kappa$.   If knowledge of the underlying system is available, it can be included this way leading to improved inter- and extrapolation quality. 

\subsection{GPs as reference predictors}

In our setup we will use the posterior mean of the GP as the reference $x_\text{r}(k):=m^+(t_*)$, where $t_*=T_\text{s}k \in \mathfrak{T}:=\{t \in \R_0^+ |\,  t=T_\text{s}k,k\in \N_0 \}$ with sampling time $T_\text{s}$. 
If $n_x>1$, either $n_x$ independent GPs can be trained (as e.g. done in \citep{matschek2020} or correlations between the outputs can be modelled, see e.g. \citep{Rasmussen2006} (Chapter 9.1) or \citep{salzmann2010implicitly}, and references therein.


By using Gaussian processes as reference predictors Task~\ref{task:ref} (2) and (3) are naturally fulfilled, as GPs form a (static) prediction model trained via a hyperparameter optimization that avoids overfitting.
The remaining task is to satisfy the trackability property. 
The reference $x_\text{r}$ and therefore the posterior mean $m^+(t_*)$  with $t_*\in \mathfrak{T}$ must be consistent with state constraints. Additionally, it must be followable for system \eqref{eq:system} such that there exists an input $u(k) \in \mathcal U,\, \forall k\in \N_0$ for system \eqref{eq:system} to stay on the reference when starting on it. 
To this end we use the definition of a reachable set from \cite{blanchini2008set}:

\begin{defi}
(Reachability set). Given the set of initial conditions $\mathcal P \subset \Rn{x}$, the reachability set $\mathcal R_T (\mathcal P)\subset \Rn{x}$ from $\mathcal P$ in time $T < +\infty$ is the set of all states $x$ for which there exists $x(0) \in \mathcal P$ and $u(\cdot) \in \mathcal U$ such that $x(T) = x$.
\end{defi}

We can use the one step ahead reachable set $\mathcal R_1 (\mathcal P)$ from an initial condition $\mathcal P:=\{x_\text{r}(k)\}$ as a sufficient condition to verify that $x_\text{r}(k+1)$ is followable, i.e. if  $x_\text{r}(k+1)\in \mathcal R_1 \big(x_\text{r}(k)\big)$  and $x_\text{r}(k) \in \mathcal X$ for all $k\in\N_0$ then the reference is reachable according to Definition~\ref{def:reach}.
We denote the one step ahead reachable tube as  $\mathcal T_{k+1}:=\mathcal R_1 (x_\text{r}(k))$. An illustration of a reachable tube is shown in Figure~\ref{fig:tube}.
In essence, we have to design the GP such that the posterior mean $m^+(t_*=T_\text{s} k)$ is constrained by the intersection of the state constraints $\mathcal X$ and the reachable tube $\mathcal T_k$. Then, point three of Task~\ref{task:ref} is fulfilled. We show how to do so in the following paragraphs.

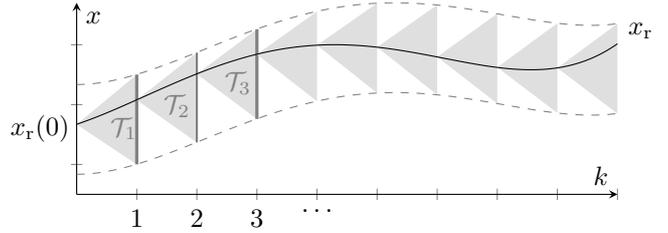
\begin{figure}
\begin{center}
\input{pretty_illu_tube.tex}
\caption{Illustration of a reachable tube.}
\label{fig:tube}
\end{center}
\end{figure}

\section{Constrained GP Learning}

We propose to include constraints in the learning phase of the GP to satisfy $m^+(t_*) \in (\mathcal T_k \cap \mathcal X )$ with $t_*=T_\text{s}k$ by this ensuring trackability. 
To learn the hyperparameters $\theta$ of the mean and covariance function we rely on maximising the logarithmic marginal likelihood. This optimization results in a point estimate of the most likely hyperparameters given the prior belief (uniformly distributed hyperparameter prior) and the hyperparameter training data set $D_\theta\in\mathcal D$. Analogously to \eqref{eq:bold_data} we define $\boldsymbol{t}_\theta, \boldsymbol{y}_\theta$ and $\boldsymbol{m}(\boldsymbol{t}_\theta)$ from $D_\theta$. 
The optimization problem can be written as
\begin{argmini!}
{\theta}{l(\theta)}
{\label{eq:hyp_opti}}{\hat{\theta}:=}
\addConstraint{m^+(t_*|D_\theta,\theta)}{\in (\mathcal T_k \cap \mathcal X )}
\addConstraint{t_*=T_\mathrm{s}k,\;\forall k }{\in\left\{0,\ldots,\bar{k}\right\},}
\end{argmini!}
where the cost function $l(\theta)$ is the negative logarithmic marginal likelihood
\begin{equation*}
l( \theta):= \ln (|K(\boldsymbol{t}_\theta,\boldsymbol{t}_\theta)|)+\boldsymbol{y}_\theta^\top K(\boldsymbol{t}_\theta,\boldsymbol{t}_\theta)^{-1}\boldsymbol{y}_\theta+n_{D_\theta}\ln (2\pi).
\end{equation*}
We denote the optimal solution of \eqref{eq:hyp_opti} with $\hat\theta$.
Even though the number of decision variables $\theta$ is rather small, the optimization problem is complex as it has a large number of constraints and due to the nonlinearity and nonconvexity of the cost and constraint. Especially the inverse of $K(\boldsymbol{t}_\theta,\boldsymbol{t}_\theta)$ introduces a significant computational complexity when considering a high number of data points $n_{\mathcal{D}_\theta}$. To solve the  optimization problem several numerical optimization methods exist, see e.g. \cite{kocijan2016modelling}. 
In the remainder of this paper, we rely on the following assumption:
\begin{ass}
The optimization problem \eqref{eq:hyp_opti} is feasible.
\label{ass:feasible}
\end{ass}

\begin{propo}
Given Assumption~\ref{ass:feasible} the resulting parame\-trisation $\hat \theta $ of the GP obtained via problem~\eqref{eq:hyp_opti} guarantees trackability of the reference for  system~\eqref{eq:system} for all $k \in \{0,1,\ldots, \bar k\}$. 
\label{prop:reach}
\end{propo}

Please note that constraint satisfaction of the predicted mean is guaranteed in a deterministic way (despite the stochastic nature of the GP). 
However, Proposition~\ref{prop:reach} only guarantees trackability up to step  $\bar k$. To be able to guarantee trackability for all times, as demanded in Task~\ref{task:ref}, we need to further investigate the underlying reference structure. We will consider asymptotic as well as periodic references as special cases.

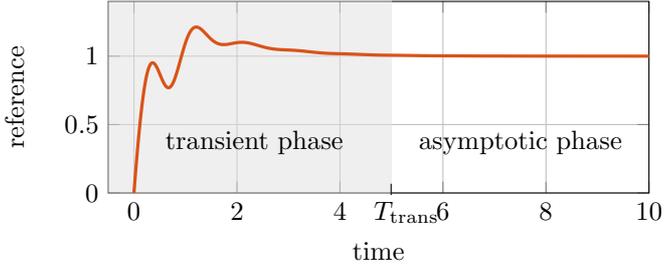
\begin{figure}
\begin{center}
\input{pretty_ref_illu_const.tex}
\caption{Illustration of an asymptotically constant reference. During the transient phase (grey area) the reference varies arbitrary whereas after $T_\text{trans}=T_\text{s}k_\text{trans} $ it converges to a constant.}
\label{fig:ref_illu_const}
\end{center}
\end{figure}

\subsection{Asymptotically Constant References}
A special case of time dependent references are those which change for a finite time and are  constant (or converge to a constant) afterwards, cf. Figure~\ref{fig:ref_illu_const}. Examples for such references are the transition between two setpoints in chemical plants or the parking of a car.
As the transient phase is of finite time $T_\text{trans}=T_\text{s} k_\text{trans} < \infty$, a finite number of constraints allows for  trackability during the transient via $\bar k \geq k_\text{trans}$.
In the following, an iterative algorithm is derived which ensures trackability for all times $k\geq \bar k$.

For the specific type of reference we use a prior mean $m$ and covariance function $\kappa$ with following properties:
\begin{ass}
The prior mean function $m$ is constant.
\label{ass:mean_const}
\end{ass}
\begin{ass}
The covariance function $\kappa$ is stationary and strictly monotonously decreasing $\kappa(t_1,t_2)<\kappa(t_3,t_4)$ for all $|t_2-t_1|>|t_4-t_3|$. 
\label{ass:cov_const}
\end{ass}
\begin{ass}
The absolute value of the time derivative of the covariance function $\dot \kappa: \R \times \R \to \R, \dot \kappa:=\dfrac{\partial \kappa(\cdot,t')}{\partial t} $ is strictly monotonously decreasing $|\dot \kappa(t_1,t_2)|<|\dot \kappa(t_3,t_4)|$ at least for all $|t_2-t_1|>|t_4-t_3|>\zeta(\theta)$, where $\zeta:\Rn{\theta}\to \R$   depends on the hyperparameters. 
\label{ass:derivative_kappa}
\end{ass}
For example, the popular squared exponential covariance function $\kappa(t,t') = \theta_1^2 \text{exp}(-\frac{1}{2\theta_2^{2}} (t-t')^2)$ fulfils Assumption~\ref{ass:cov_const} and \ref{ass:derivative_kappa} for $\zeta(\theta) =\theta_2$.

Assuming fixed  hyperparameters $\theta$ and a fixed training data set $D_\theta \in \mathcal D$, 
the posterior mean \eqref{eq:mean} can be reformulated as a weighted sum 
\begin{equation}
m^+(t_*)=m(t_*)+\sum\limits_{i=1}^{n_{D_\theta}} c_i  \kappa(t_i,t_*).
\label{eq:sum}
\end{equation}
Here, $t_i \in \boldsymbol{t}_\theta $ from $D_\theta$ and $c_i$ are constant coefficients which depend on the fixed hyperparameters and the training data. 
A bound on $m^+$ can be obtained via the triangular inequality such that 
\begin{equation}
    | m^+(t_*)-m(t_*)|\leq \bar m(t_*) := \sum\limits_{i=1}^{n_{D_\theta}} |c_i|  \kappa(t_i,t_*).
    \label{eq:bound_m}
\end{equation}
Similar to \eqref{eq:sum} and \eqref{eq:bound_m} and with Assumption~\ref{ass:mean_const}, the  derivative of $m^+$ can be expressed via
\begin{equation*}
    \dot m^+(t_*) = \sum\limits_{i=1}^{n_{D_\theta}} c_i \dot{ \kappa}(t_i,t_*),   
\end{equation*}
with a corresponding bound
\begin{equation}
    |\dot m^+(t_*)| \leq \bar{\dot{m}}(t_*):=\sum\limits_{i=1}^{n_{D_\theta}} |c_i| |\dot{ \kappa}(t_i,t_*)|.
      \label{eq:bound_mdot}
\end{equation}
Furthermore, we rely on the assumption that the growth rate of the tube can be characterised by a constant lower bound $\underline{ \tau}$ and upper bound $\overline{\tau}$:
\begin{ass}
$\underline{ \tau}$ and $\overline{\tau}$ are constant and form an inner approximation of the one-step reachable tube such that $[x_\text{r}(k)-\underline{ \tau}T_\text{s}, x_\text{r}(k)+\overline{ \tau}T_\text{s}] \subseteq \mathcal{T}_{k+1}$ 
for all $x_\text{r}(k)  \in \tilde{\mathcal{X}}\subseteq \mathcal{X}$. 
\label{ass:tube_growth}
\end{ass}

Let  $\tilde{\mathcal{X}}:=[m(\bar t)-\bar m(\bar t) ,m(\bar t)+\bar m (\bar t)]$. If $\tilde{\mathcal{X}} \subseteq \mathcal X$, and if $\underline \tau <\bar{\dot{m}}(\bar t)<\overline \tau$ for time $\bar t=\bar k T_\text{s}$, then (relying also on \eqref{eq:hyp_opti}) $m(t) \in \mathcal X \cap \mathcal T_k(\mathcal P)$ for all $t \in \mathfrak{T}$, i.e. trackability is achieved as also outlined in Lemma~\ref{lemma:constant}. If these requirements are not fulfilled for $\bar k$, the hyperparameter optimization \eqref{eq:hyp_opti} must be performed again with updated $\bar k$. Iteratively updating the number of constraints by increasing  $\bar k$ will lead to constraint satisfaction for a longer time span, decreased bounds for the mean and its derivative (as $|t_i-\bar t|$ is increased) as well as less conservative bounds $\underline \tau$, $\overline \tau$.
In Algorithm \ref{algo:hyp}, the whole procedure is summarised.

\begin{algorithm}
\caption{GP Learning for asymptotically constant references}
\begin{algorithmic}[1]
\Procedure{GP Training}{}
\State \textbf{Init} $\bar k, \mathcal D_\theta$
\While {true} 
\State obtain $\hat \theta$ via \eqref{eq:hyp_opti}  using $\left(\bar k,\mathcal D_\theta\right) $
\If {$|t_i-T_\text{s}\bar k|>\zeta\left(\hat \theta\right)$}
\State compute $ m^+(\bar t),\bar m(\bar t),\bar{\dot{ m}}(\bar t)$ with $\bar t=T_\text{s}\bar k$  \phantom{ssssssssssi} via \eqref{eq:mean},\eqref{eq:bound_m},\eqref{eq:bound_mdot}  using $\left(\hat \theta,\mathcal D_\theta \right)$
\If {$m^+(\bar t) \pm \bar m(\bar t) \in \mathcal X $}
\State choose $\underline \tau, \overline \tau$ in accordance to Ass.~\ref{ass:tube_growth} with \phantom{sssssssssssssii} $\tilde{\mathcal {X}}=[m^+(\bar t) - \bar m(\bar t), m^+(\bar t) + \bar m(\bar t)]$
\If{$ \underline \tau \leq \bar{\dot{ m}}(\bar t) \leq \overline \tau $} 
\Break
\EndIf
\EndIf
\EndIf
\State $\bar k \gets \bar k +1$
\EndWhile
\State \Return $\bar k, \hat \theta$
\EndProcedure
\end{algorithmic}
\label{algo:hyp}
\end{algorithm}

We assume that Algorithm~\ref{algo:hyp} terminates in finite time.
This allows to conclude the following result:

\begin{lemm}
Given Assumptions~\ref{ass:feasible} to \ref{ass:tube_growth} the posterior mean \eqref{eq:mean} of a GP trained with Algorithm~\ref{algo:hyp} is trackable in the sense of Definition~\ref{def:reach} for system \eqref{eq:system}.
\label{lemma:constant}
\end{lemm}
\begin{proof}
If Algorithm~\ref{algo:hyp} converges, we obtain the optimised hyperparameters $\hat \theta$ and the time instant $\bar k$. Problem \eqref{eq:hyp_opti} (under Assumption~\ref{ass:feasible}) guarantees $m^+(t)\in (\mathcal T_k \cap \mathcal X )$ for all $t\in \mathfrak{T}_\leq:=\{t \in \mathfrak T|t\leq \bar t=\bar k T_\text{s}\}$. 
From Assumption~\ref{ass:cov_const} follows $\bar m(t) < \bar m(\bar t)$ 
if $|t_i- t|>|t_i- \bar t|$ for all $t_i \in \boldsymbol{t}_\theta$.
Line 7 in Algorithm~\ref{algo:hyp} guarantees $[m(\bar t)-\bar m(\bar t), m(\bar t)+\bar m(\bar t)] \in \mathcal X$ and thus $m^+(t) \in \mathcal X$ for all times $t \in \mathfrak{T}_> :=\mathfrak{T}\setminus \mathfrak{T}_\leq $. 

Line 5 in Algorithm~\ref{algo:hyp} ensures monotonicity of $|\dot \kappa|$ (see Assumption~\ref{ass:derivative_kappa}).
Additionally, $ \underline \tau \leq \bar{\dot{ m}}(\bar t) \leq \overline \tau $ (line 9).
Due to the structure of \eqref{eq:bound_mdot} monotonicity of $|\dot \kappa|$ implies monotonicity of $\bar{\dot{ m}}$ such that 
$\bar{\dot{m}}(t) < \bar{\dot{m}}(\bar t)$ for all $t$ which fulfil $|t_i- t|>|t_i-\bar t|$, with $t_i\in \boldsymbol{t}_\theta$. 
Consequently, 
$ \underline \tau \leq \bar{\dot{ m}}(t) \leq \overline \tau $ for all $t\in \mathfrak{T}_>$. 
Including  Assumption~\ref{ass:tube_growth} (line 8) results in 
$m^+(t+T_\text{s}) \in \left[ m^+(t)-\underline{ \tau}T_\text{s}    ,m^+(t)+\overline{ \tau}T_\text{s}\right]  \subseteq \mathcal{T}_{k+1}\,\forall t \in \mathfrak{T}_>$.  
In other words, $m^+(t)\in \mathcal T_k $ for all times  $t \in \mathfrak{T}_>$. 
All in all, $m^+(t)\in (\mathcal T_k \cap \mathcal X )$ for all times  $t \in \mathfrak{T}$. \hspace{2.7cm} \qed 

\end{proof}

We provide an illustrative example to show the power of the approach: 

\begin{figure}
\begin{center}
\input{pretty_convergence_as_const.tex}
\caption{Optimal hyperparameters $\hat \theta$, the bound for the mean value $\bar m$, and the bound on its derivative $\bar{\dot{m}}$ compared to the inner approximation of the tube growth $\overline \tau, \underline \tau$ at each iteration of Algorithm~\ref{algo:hyp} for Example~\ref{ex:const}. The last 11 iterations are zoomed in the middle plot.}
\label{fig:conv_const}
\end{center}
\end{figure}
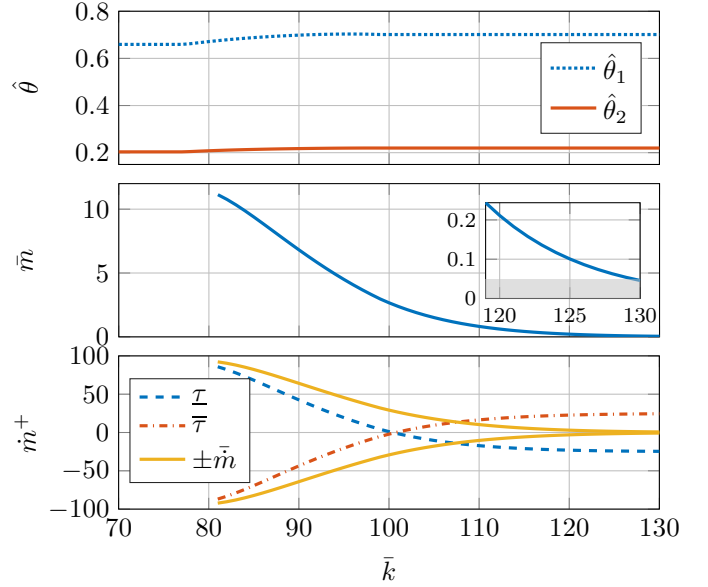

\begin{example}
Given is a dynamical system 
$x(k+1)=0.9x(k)+0.5u(k)$ 
with state constraints $\mathcal X=[-2,0.05]$ and input constraints $\mathcal U=[-0.5, 0.5]$. 
We want to model the reference depicted by the black solid line in Fig.~\ref{fig:pred_const}, which is however unknown to the GP. Only data points (depicted as crosses), which cover the transient phase  and the fact that the reference converges to a constant after the transient are known. 
We choose a constant zero prior mean function $m(t_*)=0$ and the squared exponential covariance function $\kappa(t_i,t_*)=\theta_1^2 \exp (-\frac{1}{2\theta_2^2}(t_*-t_i)^2)$. 
Algorithm~\ref{algo:hyp} is initialised  with $\bar k=70> k_\text{trans}$.  The convergence of the bounds $\bar m, \dot{\bar{m}}$ as well as the optimal hyperparameters $\hat \theta $ at each iteration are depicted in Figure~\ref{fig:conv_const}. Monotonicity of $|\dot \kappa|$ (line 5) is fulfilled after $k=81$ only.
 The convergence of $\bar m$ and $\bar{\dot{m}}$ is mainly influenced by the increase of the distance $|t_i-t_*|$. 
 After $\bar k=101$ the inner tube approximation is non empty (line 9 is feasible) and after $\bar k=108$ the bound on the derivative lies inside of it (line 9 is fulfilled). 
 The iterative algorithm is terminated at $\bar k=130$ as the bound for the predicted mean becomes small enough ($\bar m=0.0468$) to guarantee trackability for all times. The resulting GP prediction is depicted in Figure~\ref{fig:pred_const} in blue dashed line. In contrast to an GP whose hyperparameters were conventionally optimised without additional constraints (depicted in red dash-dotted line) it satisfies trackability for all times.
\label{ex:const}
\end{example}

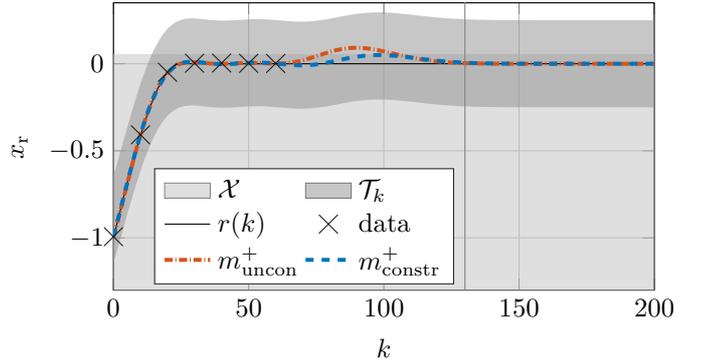
\begin{figure}
\begin{center}
\input{pretty_prediction_as_const.tex}
\caption{Comparison between an unconstrained GP (dash-dotted) and the proposed learning algorithm (dashed). The constraints for the hyperparameter optimization, i.e. the intersection of $\mathcal X$ and $\mathcal T_k$ (shaded in dark grey), were used up to $\bar k=130$ (indicated with a vertical line).}
\label{fig:pred_const}
\end{center}
\end{figure}

\subsection{Periodical references}
Often periodic references are of interest. They occur e.g.\ in the periodic operation of chemical reactors or when a robot manufactures the same item iteratively.
To be able to extrapolate periodic signals with Gaussian processes we rely on Assumption~\ref{ass:mean_const} and the following structure of covariance functions:

\begin{ass}
The covariance function $ \kappa$ is stationary and periodic, such that $\kappa (t,t')= \kappa(t,t'+n T_\text{p})$ with $n\in \N_0 $ and period $T_\text{p}$.
\label{ass:cov_per}
\end{ass}

To guarantee trackability of the reference generated by a GP using Assumptions~\ref{ass:mean_const} and \ref{ass:cov_per} the constraints of the hyperparameter optimisation~\eqref{eq:hyp_opti} should cover at least one period  $\bar k \geq T_\text{p}/T_\text{s}$. 
If the period $T_\text{p}$ is an integer multiple of the sampling time $T_\text{s}$ then \eqref{eq:hyp_opti} with $\bar k \geq T_\text{p}/T_\text{s}$ guarantees trackability for all times $t \in \mathfrak T$. In the more general case, constraint satisfaction must be guaranteed not only point wise but ultimately  for all $t\in[0,T_\text{p}]$. Therefore, we will derive bounds on the mean and its derivative to guarantee constraint satisfaction not only at the sampling instances $t \in \mathfrak T$, but for all $t\in[0,T_\text{p}]$. 
To this end, we propose the following optimization problem:
\begin{subequations}
\begin{align}
\hat \theta & := \argmin_\theta l(\theta)& &\\
\text{s.\ t.}& & & \nonumber \\
m^+_\text{min}&=\min_t m^+(t|D_\theta,\theta), & \text{s.t.}\, 0 &\leq t\leq T_\text{s}\bar k, \label{eq:hyp_opti_per_minm}\\
m^+_\text{max}&=\max_t m^+(t|D_\theta,\theta), &\text{s.t.}\, 0&\leq t\leq T_\text{s}\bar k,\label{eq:hyp_opti_per_maxm}\\
\dot m^+_{\text{min},i}&=\min_t \dot m^+(t|D_\theta,\theta), &\text{s.t.}\,  \Delta t_i &\leq t\leq  \Delta t_{i+1}, \label{eq:hyp_opti_per_minmdot} \\
\dot m^+_{\text{max},i}&=\max_t \dot m^+(t|D_\theta,\theta),& \text{s.t.}\,  \Delta t_i &\leq t\leq  \Delta t_{i+1}, \label{eq:hyp_opti_per_maxmdot}\\
m^+_\text{min}&\in \mathcal X ,\hspace{0.43cm} m^+_\text{max}\in \mathcal X, & & \label{eq:hyp_opti_per_X}\\
\dot m^+_{\text{min},i} &\geq \underline \tau_i  , \hspace{0.2cm}  \dot m^+_{\text{max},i} \leq \overline \tau_i, & & \label{eq:hyp_opti_per_tube}\\
\forall i &\in [0,1,\ldots,\eta-1],& \Delta t_i&:= \frac{i}{\eta} T_\text{s}\bar k.  
\end{align}
\label{eq:hyp_opti_per}
\end{subequations}
Here, a constant upper and lower bound on the mean via \eqref{eq:hyp_opti_per_minm} and \eqref{eq:hyp_opti_per_maxm} is obtained which need to fulfil the state constraints \eqref{eq:hyp_opti_per_X}. As a constant bound for the derivative of the mean and the inner tube approximation could be quite conservative,  piecewise constant approximations are used. The whole time span $[0, T_\text{s}  \bar k ]$ is therefor divided into $\eta$ intervals. For each interval, a lower bound $\dot m^+_{\text{min},i}$ and upper bound $\dot m^+_{\text{min},i}$ on the derivative of the mean is calculated in \eqref{eq:hyp_opti_per_minmdot} and\eqref{eq:hyp_opti_per_maxmdot}.
Furthermore, inner approximations of the reachable tube growth $[\underline \tau_i, \overline \tau_i]$ for $\tilde{\mathcal{X}}_i:=\{m^+(t)\, \forall t\in [\Delta t_i, \Delta t_{i+1}]\}$ are determined according to Assumption~\ref{ass:tube_growth}. Followability is achieved via \eqref{eq:hyp_opti_per_tube} and via the following assumption:
\begin{ass}
The optimization problem \eqref{eq:hyp_opti_per} is feasible.
\label{ass:feasible_periodic}
\end{ass}

\begin{lemm}
Given Assumptions~\ref{ass:mean_const}, \ref{ass:tube_growth}, \ref{ass:cov_per}, \ref{ass:feasible_periodic}, and $\bar k \geq T_\text{p}/T_\text{s}$ the posterior mean \eqref{eq:mean} of a GP trained with \eqref{eq:hyp_opti_per} is trackable in the sense of Definition~\ref{def:reach} for system \eqref{eq:system}.
\end{lemm}
\begin{proof}
Under Assumption~\ref{ass:feasible_periodic} optimization problem \eqref{eq:hyp_opti_per} guarantees $m^+(t)\in [m_\text{min}^+, m^+_\text{max}] \subseteq \mathcal X$ for all $t\in[0,T_\text{s} \bar k]$. With $\bar k \geq T_\text{p}/ T_\text{s}$ and Assumptions~\ref{ass:mean_const} and \ref{ass:cov_per} $m^+(t+nT_\text{p}) = m^+(t)$, where $n\in \N$ and consequently $m^+(t) \in \mathcal X$ for all $t\in \R$. 
For each time interval $\Delta t_i \geq t\geq \Delta t_{i+1}$, $\dot m^+(t) \in [m^+_{\text{min},i},m^+_{\text{max},i}] \subseteq [\underline \tau_i, \overline \tau_i]$, such that with Assumption~\ref{ass:tube_growth} $m^+(t)\in \mathcal T$ for each $t \in [0,T_\text{p}/T_\text{s}]$. With $\bar k \geq T_\text{p}/ T_\text{s}$ and Assumptions~\ref{ass:mean_const} and \ref{ass:cov_per} $m^+(t) \in \mathcal T$ for all $t\in \R$, and consequently $m^+(t) \in (\mathcal X \cap \mathcal T)$ for all $t\in \R$. \qed
\end{proof}
  
Please note that from a practical side, optimization problem \eqref{eq:hyp_opti} with very small sampling time $T_\text{s}$ might be preferred over \eqref{eq:hyp_opti_per} as it is computationally cheaper and tends to the same result for $T_\text{s} \rightarrow 0$.

\begin{figure}
\begin{center}
\input{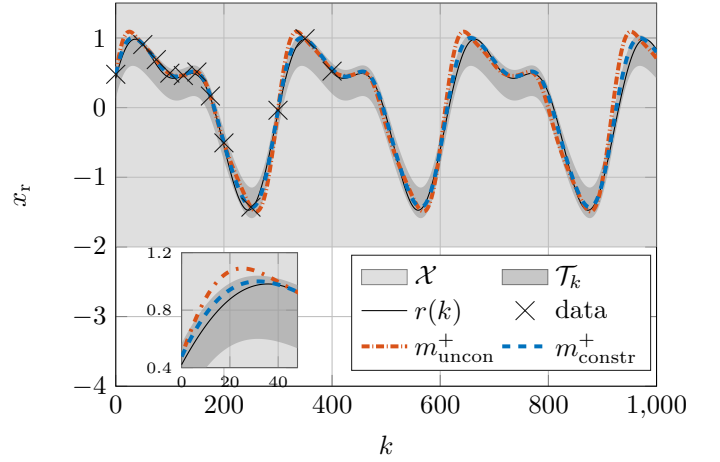}
\caption{GP prediction of a periodic reference. The underlying unknown reference $r(k)=\sin(2k)+0.5\sin(4k+1)$ is modelled based on the data points ($\times$) with a periodic covariance function.  }
\label{fig:pred_sin}
\end{center}
\end{figure}

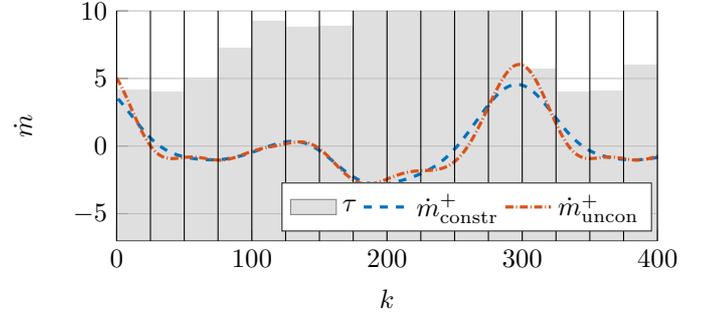
\begin{figure}
\begin{center}
\input{pretty_convergence_periodic.tex}
\caption{Inner tube growth rate approximation (shaded area) which should not be left by the derivative of the predicted mean $\dot m$. A number of $\eta=16$ intervals are used to lower the conservatism of the bounds.}
\label{fig:conv_sin}
\end{center}
\end{figure}

\begin{example}
Given is the dynamical system $x(k+1)=0.9x(k)+0.1u(k)$ with state constraints $\mathcal X=[-2,2]$ and input constraints $\mathcal U=[-3, 1.4]$. 
In accordance to Assumptions~\ref{ass:mean_const} and \ref{ass:cov_per} a zero prior mean function $m(t_*)=0$ and a periodical covariance function $\kappa=\theta_1^2 \exp (-\frac{2}{\theta_2^2}\sin(\pi\theta_3^{-1}(t_*-t_i)^2))$ are chosen. 
Figure~\ref{fig:pred_sin} shows the (unknown) reference $r(k)$ in solid black which should be modelled based on the observations depicted as black crosses. The predicted mean should be consistent with the intersection of the state constraints $\mathcal X$ and the reachability tube $\mathcal T_k$, depicted as dark grey area. 
 The Gaussian process parametrised via the constrained hyperparameter optimization provides a  reference prediction which is depicted as  dashed line. It satisfies the trackability conditions for all times. In contrast to this, an unconstrained hyperparameter optimization using the same training data leads to the red dash-dotted prediction which violates the reachable tube. This can be seen e.g. in the inlay plot in Figure~\ref{fig:pred_sin} showing a zoomed view on the first 50 steps as well as in Figure~\ref{fig:conv_sin}. 
The derivative of the unconstrained mean $\dot m_\text{uncon}^+$ violates the inner approximation of the tube growth rate (the trajectory leaves the grey area) depicted in Figure~\ref{fig:conv_sin}. In contrast, the derivative of the constrained mean $\dot m_\text{constr}^+$ is fulfilling those constraints. Even though the unconstrained GP satisfies the state constraints $\mathcal X$, it is not followable by the system dynamics due to the input constraints.

\end{example}

\section{Conclusion}
We outlined how Gaussian processes can model external reference signals of different structure (asymptotically constant or periodic). These GP predictions can be used e.g. in model predictive control to achieve an improved tracking performance as well as providing stability guarantees. To do so, we have proposed different algorithms to train GPs. 
These concepts are based on constrained hyperparameter optimisation to guarantee trackability and constrained satisfaction of the  predicted GP mean. Investigations for arbitrary references, online learning, as well as the use of truncated multinormal distributions for reference predictors are interesting future research directions.

\bibliography{ifacconf}             

\end{document}

%% file: pretty_illu_tube.tex
%
%
\definecolor{mycolor1}{rgb}{0.49400,0.18400,0.55600}%
\definecolor{mycolor2}{rgb}{0.46600,0.67400,0.18800}%
\definecolor{mycolor3}{rgb}{0.30100,0.74500,0.93300}%
\begin{tikzpicture}

\begin{axis}[%
width=2.8in,
height=1in,
at={(0.818851in,0.523448in)},
scale only axis,
xmin=1,
xmax=10,
xlabel={$k$},
ymin=0.5,
ymax=3.7,
ylabel={$x$},
xtick={1,2,3,4,5,6,7,8,9,10},
xticklabels={0,1,2,3,$\ldots$, ,  , , , , },
yticklabels={ , , , , },
axis x line*=bottom,
axis y line*=left,
axis lines = middle,
legend style={legend cell align=left,align=left,draw=white!15!black}
]

\addplot[area legend,draw=none, fill=white!75!gray,forget plot]
table[row sep=crcr] {%
x	y\\
1	1.6678\\
2	2.501\\
2	1.001\\
}--cycle;

\addplot[draw=gray,very thick]
table[row sep=crcr] {%
x	y\\
2	2.501\\
2	1.001\\
};

\addplot[area legend,draw=none, fill=white!75!gray,forget plot]
table[row sep=crcr] {%
x	y\\
2	2.0749\\
3	2.8674\\
3	1.3674\\
}--cycle;

\addplot[draw=gray,thick]
table[row sep=crcr] {%
x	y\\
3	2.8674\\
3	1.3674\\
};

\addplot[area legend,draw=none, fill=white!75!gray,forget plot]
table[row sep=crcr] {%
x	y\\
3	2.5096\\
4	3.2587\\
4	1.7587\\
}--cycle;

\addplot[draw=gray,very thick]
table[row sep=crcr] {%
x	y\\
4	3.2587\\
4	1.7587\\
};

\addplot[area legend,draw=none, fill=white!75!gray,forget plot]
table[row sep=crcr] {%
x	y\\
4	2.8388\\
5	3.5549\\
5	2.0549\\
}--cycle;

\addplot[area legend,draw=none, fill=white!75!gray,forget plot]
table[row sep=crcr] {%
x	y\\
5	2.9916\\
6	3.6925\\
6	2.1925\\
}--cycle;

\addplot[area legend,draw=none, fill=white!75!gray,forget plot]
table[row sep=crcr] {%
x	y\\
6	2.9598\\
7	3.6639\\
7	2.1639\\
}--cycle;

\addplot[area legend,draw=none, fill=white!75!gray,forget plot]
table[row sep=crcr] {%
x	y\\
7	2.7977\\
8	3.5179\\
8	2.0179\\
}--cycle;

\addplot[area legend,draw=none, fill=white!75!gray,forget plot]
table[row sep=crcr] {%
x	y\\
8	2.6219\\
9	3.3597\\
9	1.8597\\
}--cycle;

\addplot[area legend,draw=none, fill=white!75!gray,forget plot]
table[row sep=crcr] {%
x	y\\
9	2.6117\\
10	3.3505\\
10	1.8505\\
}--cycle;

\addplot[area legend,draw=none, fill=white!75!gray,forget plot]
table[row sep=crcr] {%
x	y\\
10	3.0089\\
11	3.708\\
11	2.208\\
}--cycle;

\addplot [color=black,solid,forget plot]
  table[row sep=crcr]{%
0	1.4841\\
0.01	1.4841\\
0.02	1.4841\\
0.03	1.4842\\
0.04	1.4843\\
0.05	1.4845\\
0.06	1.4847\\
0.07	1.485\\
0.08	1.4853\\
0.09	1.4856\\
0.1	1.486\\
0.11	1.4865\\
0.12	1.4869\\
0.13	1.4875\\
0.14	1.488\\
0.15	1.4887\\
0.16	1.4893\\
0.17	1.49\\
0.18	1.4908\\
0.19	1.4915\\
0.2	1.4924\\
0.21	1.4932\\
0.22	1.4942\\
0.23	1.4951\\
0.24	1.4961\\
0.25	1.4971\\
0.26	1.4982\\
0.27	1.4993\\
0.28	1.5005\\
0.29	1.5017\\
0.3	1.5029\\
0.31	1.5042\\
0.32	1.5055\\
0.33	1.5068\\
0.34	1.5082\\
0.35	1.5096\\
0.36	1.5111\\
0.37	1.5126\\
0.38	1.5141\\
0.39	1.5157\\
0.4	1.5173\\
0.41	1.519\\
0.42	1.5206\\
0.43	1.5224\\
0.44	1.5241\\
0.45	1.5259\\
0.46	1.5277\\
0.47	1.5296\\
0.48	1.5315\\
0.49	1.5334\\
0.5	1.5354\\
0.51	1.5374\\
0.52	1.5394\\
0.53	1.5414\\
0.54	1.5435\\
0.55	1.5457\\
0.56	1.5478\\
0.57	1.55\\
0.58	1.5522\\
0.59	1.5545\\
0.6	1.5568\\
0.61	1.5591\\
0.62	1.5614\\
0.63	1.5638\\
0.64	1.5662\\
0.65	1.5687\\
0.66	1.5711\\
0.67	1.5736\\
0.68	1.5762\\
0.69	1.5787\\
0.7	1.5813\\
0.71	1.5839\\
0.72	1.5866\\
0.73	1.5892\\
0.74	1.5919\\
0.75	1.5947\\
0.76	1.5974\\
0.77	1.6002\\
0.78	1.603\\
0.79	1.6058\\
0.8	1.6087\\
0.81	1.6116\\
0.82	1.6145\\
0.83	1.6175\\
0.84	1.6204\\
0.85	1.6234\\
0.86	1.6264\\
0.87	1.6295\\
0.88	1.6325\\
0.89	1.6356\\
0.9	1.6387\\
0.91	1.6419\\
0.92	1.645\\
0.93	1.6482\\
0.94	1.6514\\
0.95	1.6547\\
0.96	1.6579\\
0.97	1.6612\\
0.98	1.6645\\
0.99	1.6678\\
1	1.6712\\
1.01	1.6745\\
1.02	1.6779\\
1.03	1.6813\\
1.04	1.6847\\
1.05	1.6882\\
1.06	1.6917\\
1.07	1.6951\\
1.08	1.6987\\
1.09	1.7022\\
1.1	1.7057\\
1.11	1.7093\\
1.12	1.7129\\
1.13	1.7165\\
1.14	1.7201\\
1.15	1.7238\\
1.16	1.7274\\
1.17	1.7311\\
1.18	1.7348\\
1.19	1.7385\\
1.2	1.7422\\
1.21	1.746\\
1.22	1.7498\\
1.23	1.7535\\
1.24	1.7573\\
1.25	1.7612\\
1.26	1.765\\
1.27	1.7688\\
1.28	1.7727\\
1.29	1.7766\\
1.3	1.7805\\
1.31	1.7844\\
1.32	1.7883\\
1.33	1.7923\\
1.34	1.7962\\
1.35	1.8002\\
1.36	1.8042\\
1.37	1.8082\\
1.38	1.8122\\
1.39	1.8162\\
1.4	1.8202\\
1.41	1.8243\\
1.42	1.8283\\
1.43	1.8324\\
1.44	1.8365\\
1.45	1.8406\\
1.46	1.8447\\
1.47	1.8488\\
1.48	1.853\\
1.49	1.8571\\
1.5	1.8613\\
1.51	1.8654\\
1.52	1.8696\\
1.53	1.8738\\
1.54	1.878\\
1.55	1.8822\\
1.56	1.8864\\
1.57	1.8907\\
1.58	1.8949\\
1.59	1.8992\\
1.6	1.9034\\
1.61	1.9077\\
1.62	1.912\\
1.63	1.9163\\
1.64	1.9206\\
1.65	1.9249\\
1.66	1.9292\\
1.67	1.9335\\
1.68	1.9378\\
1.69	1.9422\\
1.7	1.9465\\
1.71	1.9508\\
1.72	1.9552\\
1.73	1.9596\\
1.74	1.9639\\
1.75	1.9683\\
1.76	1.9727\\
1.77	1.9771\\
1.78	1.9815\\
1.79	1.9859\\
1.8	1.9903\\
1.81	1.9947\\
1.82	1.9991\\
1.83	2.0035\\
1.84	2.008\\
1.85	2.0124\\
1.86	2.0168\\
1.87	2.0213\\
1.88	2.0257\\
1.89	2.0302\\
1.9	2.0346\\
1.91	2.0391\\
1.92	2.0436\\
1.93	2.048\\
1.94	2.0525\\
1.95	2.057\\
1.96	2.0614\\
1.97	2.0659\\
1.98	2.0704\\
1.99	2.0749\\
2	2.0794\\
2.01	2.0839\\
2.02	2.0883\\
2.03	2.0928\\
2.04	2.0973\\
2.05	2.1018\\
2.06	2.1063\\
2.07	2.1108\\
2.08	2.1153\\
2.09	2.1198\\
2.1	2.1243\\
2.11	2.1288\\
2.12	2.1333\\
2.13	2.1378\\
2.14	2.1423\\
2.15	2.1468\\
2.16	2.1513\\
2.17	2.1558\\
2.18	2.1603\\
2.19	2.1648\\
2.2	2.1693\\
2.21	2.1738\\
2.22	2.1783\\
2.23	2.1828\\
2.24	2.1872\\
2.25	2.1917\\
2.26	2.1962\\
2.27	2.2007\\
2.28	2.2052\\
2.29	2.2097\\
2.3	2.2141\\
2.31	2.2186\\
2.32	2.2231\\
2.33	2.2276\\
2.34	2.232\\
2.35	2.2365\\
2.36	2.241\\
2.37	2.2454\\
2.38	2.2499\\
2.39	2.2543\\
2.4	2.2588\\
2.41	2.2632\\
2.42	2.2676\\
2.43	2.2721\\
2.44	2.2765\\
2.45	2.2809\\
2.46	2.2854\\
2.47	2.2898\\
2.48	2.2942\\
2.49	2.2986\\
2.5	2.303\\
2.51	2.3074\\
2.52	2.3118\\
2.53	2.3162\\
2.54	2.3205\\
2.55	2.3249\\
2.56	2.3293\\
2.57	2.3336\\
2.58	2.338\\
2.59	2.3423\\
2.6	2.3467\\
2.61	2.351\\
2.62	2.3553\\
2.63	2.3597\\
2.64	2.364\\
2.65	2.3683\\
2.66	2.3726\\
2.67	2.3769\\
2.68	2.3812\\
2.69	2.3854\\
2.7	2.3897\\
2.71	2.394\\
2.72	2.3982\\
2.73	2.4025\\
2.74	2.4067\\
2.75	2.4109\\
2.76	2.4152\\
2.77	2.4194\\
2.78	2.4236\\
2.79	2.4278\\
2.8	2.4319\\
2.81	2.4361\\
2.82	2.4403\\
2.83	2.4445\\
2.84	2.4486\\
2.85	2.4527\\
2.86	2.4569\\
2.87	2.461\\
2.88	2.4651\\
2.89	2.4692\\
2.9	2.4733\\
2.91	2.4774\\
2.92	2.4814\\
2.93	2.4855\\
2.94	2.4895\\
2.95	2.4936\\
2.96	2.4976\\
2.97	2.5016\\
2.98	2.5056\\
2.99	2.5096\\
3	2.5136\\
3.01	2.5176\\
3.02	2.5216\\
3.03	2.5255\\
3.04	2.5294\\
3.05	2.5334\\
3.06	2.5373\\
3.07	2.5412\\
3.08	2.5451\\
3.09	2.549\\
3.1	2.5528\\
3.11	2.5567\\
3.12	2.5605\\
3.13	2.5644\\
3.14	2.5682\\
3.15	2.572\\
3.16	2.5758\\
3.17	2.5796\\
3.18	2.5833\\
3.19	2.5871\\
3.2	2.5908\\
3.21	2.5945\\
3.22	2.5983\\
3.23	2.602\\
3.24	2.6057\\
3.25	2.6093\\
3.26	2.613\\
3.27	2.6166\\
3.28	2.6203\\
3.29	2.6239\\
3.3	2.6275\\
3.31	2.6311\\
3.32	2.6347\\
3.33	2.6382\\
3.34	2.6418\\
3.35	2.6453\\
3.36	2.6489\\
3.37	2.6524\\
3.38	2.6559\\
3.39	2.6593\\
3.4	2.6628\\
3.41	2.6662\\
3.42	2.6697\\
3.43	2.6731\\
3.44	2.6765\\
3.45	2.6799\\
3.46	2.6833\\
3.47	2.6866\\
3.48	2.69\\
3.49	2.6933\\
3.5	2.6966\\
3.51	2.6999\\
3.52	2.7032\\
3.53	2.7065\\
3.54	2.7097\\
3.55	2.713\\
3.56	2.7162\\
3.57	2.7194\\
3.58	2.7226\\
3.59	2.7257\\
3.6	2.7289\\
3.61	2.732\\
3.62	2.7352\\
3.63	2.7383\\
3.64	2.7413\\
3.65	2.7444\\
3.66	2.7475\\
3.67	2.7505\\
3.68	2.7535\\
3.69	2.7566\\
3.7	2.7595\\
3.71	2.7625\\
3.72	2.7655\\
3.73	2.7684\\
3.74	2.7713\\
3.75	2.7743\\
3.76	2.7771\\
3.77	2.78\\
3.78	2.7829\\
3.79	2.7857\\
3.8	2.7885\\
3.81	2.7913\\
3.82	2.7941\\
3.83	2.7969\\
3.84	2.7996\\
3.85	2.8024\\
3.86	2.8051\\
3.87	2.8078\\
3.88	2.8105\\
3.89	2.8131\\
3.9	2.8158\\
3.91	2.8184\\
3.92	2.821\\
3.93	2.8236\\
3.94	2.8262\\
3.95	2.8288\\
3.96	2.8313\\
3.97	2.8338\\
3.98	2.8363\\
3.99	2.8388\\
4	2.8413\\
4.01	2.8437\\
4.02	2.8461\\
4.03	2.8486\\
4.04	2.851\\
4.05	2.8533\\
4.06	2.8557\\
4.07	2.858\\
4.08	2.8603\\
4.09	2.8626\\
4.1	2.8649\\
4.11	2.8672\\
4.12	2.8694\\
4.13	2.8717\\
4.14	2.8739\\
4.15	2.8761\\
4.16	2.8782\\
4.17	2.8804\\
4.18	2.8825\\
4.19	2.8846\\
4.2	2.8867\\
4.21	2.8888\\
4.22	2.8909\\
4.23	2.8929\\
4.24	2.8949\\
4.25	2.8969\\
4.26	2.8989\\
4.27	2.9009\\
4.28	2.9028\\
4.29	2.9048\\
4.3	2.9067\\
4.31	2.9085\\
4.32	2.9104\\
4.33	2.9123\\
4.34	2.9141\\
4.35	2.9159\\
4.36	2.9177\\
4.37	2.9195\\
4.38	2.9212\\
4.39	2.923\\
4.4	2.9247\\
4.41	2.9264\\
4.42	2.928\\
4.43	2.9297\\
4.44	2.9313\\
4.45	2.933\\
4.46	2.9346\\
4.47	2.9361\\
4.48	2.9377\\
4.49	2.9392\\
4.5	2.9408\\
4.51	2.9423\\
4.52	2.9437\\
4.53	2.9452\\
4.54	2.9467\\
4.55	2.9481\\
4.56	2.9495\\
4.57	2.9509\\
4.58	2.9522\\
4.59	2.9536\\
4.6	2.9549\\
4.61	2.9562\\
4.62	2.9575\\
4.63	2.9588\\
4.64	2.96\\
4.65	2.9613\\
4.66	2.9625\\
4.67	2.9637\\
4.68	2.9648\\
4.69	2.966\\
4.7	2.9671\\
4.71	2.9682\\
4.72	2.9693\\
4.73	2.9704\\
4.74	2.9715\\
4.75	2.9725\\
4.76	2.9735\\
4.77	2.9745\\
4.78	2.9755\\
4.79	2.9765\\
4.8	2.9774\\
4.81	2.9783\\
4.82	2.9792\\
4.83	2.9801\\
4.84	2.981\\
4.85	2.9818\\
4.86	2.9826\\
4.87	2.9835\\
4.88	2.9842\\
4.89	2.985\\
4.9	2.9858\\
4.91	2.9865\\
4.92	2.9872\\
4.93	2.9879\\
4.94	2.9886\\
4.95	2.9892\\
4.96	2.9898\\
4.97	2.9905\\
4.98	2.9911\\
4.99	2.9916\\
5	2.9922\\
5.01	2.9927\\
5.02	2.9932\\
5.03	2.9937\\
5.04	2.9942\\
5.05	2.9947\\
5.06	2.9951\\
5.07	2.9956\\
5.08	2.996\\
5.09	2.9964\\
5.1	2.9967\\
5.11	2.9971\\
5.12	2.9974\\
5.13	2.9977\\
5.14	2.998\\
5.15	2.9983\\
5.16	2.9985\\
5.17	2.9988\\
5.18	2.999\\
5.19	2.9992\\
5.2	2.9994\\
5.21	2.9996\\
5.22	2.9997\\
5.23	2.9998\\
5.24	2.9999\\
5.25	3\\
5.26	3.0001\\
5.27	3.0002\\
5.28	3.0002\\
5.29	3.0002\\
5.3	3.0002\\
5.31	3.0002\\
5.32	3.0002\\
5.33	3.0001\\
5.34	3\\
5.35	3\\
5.36	2.9999\\
5.37	2.9997\\
5.38	2.9996\\
5.39	2.9994\\
5.4	2.9993\\
5.41	2.9991\\
5.42	2.9989\\
5.43	2.9986\\
5.44	2.9984\\
5.45	2.9981\\
5.46	2.9978\\
5.47	2.9975\\
5.48	2.9972\\
5.49	2.9969\\
5.5	2.9965\\
5.51	2.9962\\
5.52	2.9958\\
5.53	2.9954\\
5.54	2.995\\
5.55	2.9945\\
5.56	2.9941\\
5.57	2.9936\\
5.58	2.9931\\
5.59	2.9926\\
5.6	2.9921\\
5.61	2.9916\\
5.62	2.991\\
5.63	2.9905\\
5.64	2.9899\\
5.65	2.9893\\
5.66	2.9887\\
5.67	2.9881\\
5.68	2.9874\\
5.69	2.9867\\
5.7	2.9861\\
5.71	2.9854\\
5.72	2.9847\\
5.73	2.9839\\
5.74	2.9832\\
5.75	2.9824\\
5.76	2.9817\\
5.77	2.9809\\
5.78	2.9801\\
5.79	2.9793\\
5.8	2.9784\\
5.81	2.9776\\
5.82	2.9767\\
5.83	2.9759\\
5.84	2.975\\
5.85	2.9741\\
5.86	2.9731\\
5.87	2.9722\\
5.88	2.9712\\
5.89	2.9703\\
5.9	2.9693\\
5.91	2.9683\\
5.92	2.9673\\
5.93	2.9663\\
5.94	2.9652\\
5.95	2.9642\\
5.96	2.9631\\
5.97	2.962\\
5.98	2.961\\
5.99	2.9598\\
6	2.9587\\
6.01	2.9576\\
6.02	2.9565\\
6.03	2.9553\\
6.04	2.9541\\
6.05	2.9529\\
6.06	2.9517\\
6.07	2.9505\\
6.08	2.9493\\
6.09	2.9481\\
6.1	2.9468\\
6.11	2.9456\\
6.12	2.9443\\
6.13	2.943\\
6.14	2.9417\\
6.15	2.9404\\
6.16	2.9391\\
6.17	2.9377\\
6.18	2.9364\\
6.19	2.935\\
6.2	2.9336\\
6.21	2.9323\\
6.22	2.9309\\
6.23	2.9295\\
6.24	2.928\\
6.25	2.9266\\
6.26	2.9252\\
6.27	2.9237\\
6.28	2.9223\\
6.29	2.9208\\
6.3	2.9193\\
6.31	2.9178\\
6.32	2.9163\\
6.33	2.9148\\
6.34	2.9133\\
6.35	2.9117\\
6.36	2.9102\\
6.37	2.9086\\
6.38	2.9071\\
6.39	2.9055\\
6.4	2.9039\\
6.41	2.9023\\
6.42	2.9007\\
6.43	2.8991\\
6.44	2.8975\\
6.45	2.8958\\
6.46	2.8942\\
6.47	2.8926\\
6.48	2.8909\\
6.49	2.8892\\
6.5	2.8876\\
6.51	2.8859\\
6.52	2.8842\\
6.53	2.8825\\
6.54	2.8808\\
6.55	2.8791\\
6.56	2.8773\\
6.57	2.8756\\
6.58	2.8739\\
6.59	2.8721\\
6.6	2.8704\\
6.61	2.8686\\
6.62	2.8669\\
6.63	2.8651\\
6.64	2.8633\\
6.65	2.8615\\
6.66	2.8597\\
6.67	2.8579\\
6.68	2.8561\\
6.69	2.8543\\
6.7	2.8525\\
6.71	2.8506\\
6.72	2.8488\\
6.73	2.847\\
6.74	2.8451\\
6.75	2.8433\\
6.76	2.8414\\
6.77	2.8396\\
6.78	2.8377\\
6.79	2.8359\\
6.8	2.834\\
6.81	2.8321\\
6.82	2.8302\\
6.83	2.8283\\
6.84	2.8264\\
6.85	2.8246\\
6.86	2.8227\\
6.87	2.8208\\
6.88	2.8188\\
6.89	2.8169\\
6.9	2.815\\
6.91	2.8131\\
6.92	2.8112\\
6.93	2.8093\\
6.94	2.8074\\
6.95	2.8054\\
6.96	2.8035\\
6.97	2.8016\\
6.98	2.7996\\
6.99	2.7977\\
7	2.7958\\
7.01	2.7938\\
7.02	2.7919\\
7.03	2.7899\\
7.04	2.788\\
7.05	2.786\\
7.06	2.7841\\
7.07	2.7822\\
7.08	2.7802\\
7.09	2.7783\\
7.1	2.7763\\
7.11	2.7744\\
7.12	2.7724\\
7.13	2.7705\\
7.14	2.7685\\
7.15	2.7666\\
7.16	2.7646\\
7.17	2.7627\\
7.18	2.7607\\
7.19	2.7588\\
7.2	2.7569\\
7.21	2.7549\\
7.22	2.753\\
7.23	2.751\\
7.24	2.7491\\
7.25	2.7472\\
7.26	2.7452\\
7.27	2.7433\\
7.28	2.7414\\
7.29	2.7394\\
7.3	2.7375\\
7.31	2.7356\\
7.32	2.7337\\
7.33	2.7318\\
7.34	2.7299\\
7.35	2.728\\
7.36	2.7261\\
7.37	2.7242\\
7.38	2.7223\\
7.39	2.7204\\
7.4	2.7185\\
7.41	2.7166\\
7.42	2.7147\\
7.43	2.7129\\
7.44	2.711\\
7.45	2.7091\\
7.46	2.7073\\
7.47	2.7054\\
7.48	2.7036\\
7.49	2.7018\\
7.5	2.6999\\
7.51	2.6981\\
7.52	2.6963\\
7.53	2.6945\\
7.54	2.6927\\
7.55	2.6909\\
7.56	2.6891\\
7.57	2.6873\\
7.58	2.6856\\
7.59	2.6838\\
7.6	2.682\\
7.61	2.6803\\
7.62	2.6785\\
7.63	2.6768\\
7.64	2.6751\\
7.65	2.6734\\
7.66	2.6717\\
7.67	2.67\\
7.68	2.6683\\
7.69	2.6666\\
7.7	2.665\\
7.71	2.6633\\
7.72	2.6617\\
7.73	2.66\\
7.74	2.6584\\
7.75	2.6568\\
7.76	2.6552\\
7.77	2.6536\\
7.78	2.652\\
7.79	2.6505\\
7.8	2.6489\\
7.81	2.6474\\
7.82	2.6458\\
7.83	2.6443\\
7.84	2.6428\\
7.85	2.6413\\
7.86	2.6398\\
7.87	2.6384\\
7.88	2.6369\\
7.89	2.6355\\
7.9	2.6341\\
7.91	2.6326\\
7.92	2.6313\\
7.93	2.6299\\
7.94	2.6285\\
7.95	2.6272\\
7.96	2.6258\\
7.97	2.6245\\
7.98	2.6232\\
7.99	2.6219\\
8	2.6206\\
8.01	2.6194\\
8.02	2.6181\\
8.03	2.6169\\
8.04	2.6157\\
8.05	2.6145\\
8.06	2.6133\\
8.07	2.6122\\
8.08	2.6111\\
8.09	2.6099\\
8.1	2.6088\\
8.11	2.6078\\
8.12	2.6067\\
8.13	2.6056\\
8.14	2.6046\\
8.15	2.6036\\
8.16	2.6026\\
8.17	2.6017\\
8.18	2.6007\\
8.19	2.5998\\
8.2	2.5989\\
8.21	2.598\\
8.22	2.5971\\
8.23	2.5963\\
8.24	2.5955\\
8.25	2.5947\\
8.26	2.5939\\
8.27	2.5931\\
8.28	2.5924\\
8.29	2.5917\\
8.3	2.591\\
8.31	2.5903\\
8.32	2.5897\\
8.33	2.589\\
8.34	2.5885\\
8.35	2.5879\\
8.36	2.5873\\
8.37	2.5868\\
8.38	2.5863\\
8.39	2.5858\\
8.4	2.5854\\
8.41	2.5849\\
8.42	2.5845\\
8.43	2.5842\\
8.44	2.5838\\
8.45	2.5835\\
8.46	2.5832\\
8.47	2.5829\\
8.48	2.5827\\
8.49	2.5825\\
8.5	2.5823\\
8.51	2.5821\\
8.52	2.582\\
8.53	2.5819\\
8.54	2.5818\\
8.55	2.5818\\
8.56	2.5817\\
8.57	2.5817\\
8.58	2.5818\\
8.59	2.5819\\
8.6	2.5819\\
8.61	2.5821\\
8.62	2.5822\\
8.63	2.5824\\
8.64	2.5826\\
8.65	2.5829\\
8.66	2.5832\\
8.67	2.5835\\
8.68	2.5838\\
8.69	2.5842\\
8.7	2.5846\\
8.71	2.5851\\
8.72	2.5855\\
8.73	2.586\\
8.74	2.5866\\
8.75	2.5872\\
8.76	2.5878\\
8.77	2.5884\\
8.78	2.5891\\
8.79	2.5898\\
8.8	2.5906\\
8.81	2.5913\\
8.82	2.5922\\
8.83	2.593\\
8.84	2.5939\\
8.85	2.5948\\
8.86	2.5958\\
8.87	2.5968\\
8.88	2.5978\\
8.89	2.5989\\
8.9	2.6\\
8.91	2.6011\\
8.92	2.6023\\
8.93	2.6036\\
8.94	2.6048\\
8.95	2.6061\\
8.96	2.6075\\
8.97	2.6088\\
8.98	2.6103\\
8.99	2.6117\\
9	2.6132\\
9.01	2.6148\\
9.02	2.6163\\
9.03	2.618\\
9.04	2.6196\\
9.05	2.6213\\
9.06	2.6231\\
9.07	2.6249\\
9.08	2.6267\\
9.09	2.6286\\
9.1	2.6305\\
9.11	2.6325\\
9.12	2.6345\\
9.13	2.6365\\
9.14	2.6386\\
9.15	2.6407\\
9.16	2.6429\\
9.17	2.6451\\
9.18	2.6474\\
9.19	2.6497\\
9.2	2.6521\\
9.21	2.6545\\
9.22	2.657\\
9.23	2.6595\\
9.24	2.662\\
9.25	2.6646\\
9.26	2.6672\\
9.27	2.6699\\
9.28	2.6727\\
9.29	2.6755\\
9.3	2.6783\\
9.31	2.6812\\
9.32	2.6841\\
9.33	2.6871\\
9.34	2.6901\\
9.35	2.6932\\
9.36	2.6964\\
9.37	2.6995\\
9.38	2.7028\\
9.39	2.7061\\
9.4	2.7094\\
9.41	2.7128\\
9.42	2.7162\\
9.43	2.7197\\
9.44	2.7233\\
9.45	2.7269\\
9.46	2.7305\\
9.47	2.7342\\
9.48	2.738\\
9.49	2.7418\\
9.5	2.7457\\
9.51	2.7496\\
9.52	2.7536\\
9.53	2.7576\\
9.54	2.7617\\
9.55	2.7659\\
9.56	2.7701\\
9.57	2.7744\\
9.58	2.7787\\
9.59	2.7831\\
9.6	2.7875\\
9.61	2.792\\
9.62	2.7965\\
9.63	2.8011\\
9.64	2.8058\\
9.65	2.8105\\
9.66	2.8153\\
9.67	2.8202\\
9.68	2.8251\\
9.69	2.8301\\
9.7	2.8351\\
9.71	2.8402\\
9.72	2.8453\\
9.73	2.8505\\
9.74	2.8558\\
9.75	2.8612\\
9.76	2.8666\\
9.77	2.872\\
9.78	2.8776\\
9.79	2.8831\\
9.8	2.8888\\
9.81	2.8945\\
9.82	2.9003\\
9.83	2.9062\\
9.84	2.9121\\
9.85	2.9181\\
9.86	2.9241\\
9.87	2.9302\\
9.88	2.9364\\
9.89	2.9426\\
9.9	2.949\\
9.91	2.9553\\
9.92	2.9618\\
9.93	2.9683\\
9.94	2.9749\\
9.95	2.9815\\
9.96	2.9883\\
9.97	2.9951\\
9.98	3.0019\\
9.99	3.0089\\
10	3.0159\\
};
\addplot [color=gray,dashed,forget plot]
  table[row sep=crcr]{%
1	2.3357\\
1.01	2.3357\\
1.02	2.3357\\
1.03	2.3358\\
1.04	2.3359\\
1.05	2.3361\\
1.06	2.3362\\
1.07	2.3365\\
1.08	2.3368\\
1.09	2.3371\\
1.1	2.3374\\
1.11	2.3378\\
1.12	2.3383\\
1.13	2.3387\\
1.14	2.3392\\
1.15	2.3398\\
1.16	2.3404\\
1.17	2.341\\
1.18	2.3417\\
1.19	2.3424\\
1.2	2.3431\\
1.21	2.3439\\
1.22	2.3447\\
1.23	2.3456\\
1.24	2.3465\\
1.25	2.3474\\
1.26	2.3484\\
1.27	2.3494\\
1.28	2.3504\\
1.29	2.3515\\
1.3	2.3526\\
1.31	2.3537\\
1.32	2.3549\\
1.33	2.3561\\
1.34	2.3574\\
1.35	2.3587\\
1.36	2.36\\
1.37	2.3613\\
1.38	2.3627\\
1.39	2.3641\\
1.4	2.3656\\
1.41	2.3671\\
1.42	2.3686\\
1.43	2.3701\\
1.44	2.3717\\
1.45	2.3733\\
1.46	2.3749\\
1.47	2.3766\\
1.48	2.3783\\
1.49	2.3801\\
1.5	2.3818\\
1.51	2.3836\\
1.52	2.3854\\
1.53	2.3873\\
1.54	2.3892\\
1.55	2.3911\\
1.56	2.393\\
1.57	2.395\\
1.58	2.397\\
1.59	2.399\\
1.6	2.4011\\
1.61	2.4032\\
1.62	2.4053\\
1.63	2.4074\\
1.64	2.4096\\
1.65	2.4118\\
1.66	2.414\\
1.67	2.4163\\
1.68	2.4185\\
1.69	2.4209\\
1.7	2.4232\\
1.71	2.4255\\
1.72	2.4279\\
1.73	2.4303\\
1.74	2.4327\\
1.75	2.4352\\
1.76	2.4377\\
1.77	2.4402\\
1.78	2.4427\\
1.79	2.4453\\
1.8	2.4478\\
1.81	2.4504\\
1.82	2.4531\\
1.83	2.4557\\
1.84	2.4584\\
1.85	2.4611\\
1.86	2.4638\\
1.87	2.4665\\
1.88	2.4693\\
1.89	2.4721\\
1.9	2.4749\\
1.91	2.4777\\
1.92	2.4805\\
1.93	2.4834\\
1.94	2.4863\\
1.95	2.4892\\
1.96	2.4921\\
1.97	2.4951\\
1.98	2.498\\
1.99	2.501\\
2	2.504\\
2.01	2.5071\\
2.02	2.5101\\
2.03	2.5132\\
2.04	2.5163\\
2.05	2.5194\\
2.06	2.5225\\
2.07	2.5256\\
2.08	2.5288\\
2.09	2.532\\
2.1	2.5352\\
2.11	2.5384\\
2.12	2.5416\\
2.13	2.5448\\
2.14	2.5481\\
2.15	2.5514\\
2.16	2.5547\\
2.17	2.558\\
2.18	2.5613\\
2.19	2.5647\\
2.2	2.568\\
2.21	2.5714\\
2.22	2.5748\\
2.23	2.5782\\
2.24	2.5816\\
2.25	2.585\\
2.26	2.5885\\
2.27	2.592\\
2.28	2.5954\\
2.29	2.5989\\
2.3	2.6024\\
2.31	2.606\\
2.32	2.6095\\
2.33	2.613\\
2.34	2.6166\\
2.35	2.6202\\
2.36	2.6237\\
2.37	2.6273\\
2.38	2.6309\\
2.39	2.6346\\
2.4	2.6382\\
2.41	2.6418\\
2.42	2.6455\\
2.43	2.6492\\
2.44	2.6528\\
2.45	2.6565\\
2.46	2.6602\\
2.47	2.664\\
2.48	2.6677\\
2.49	2.6714\\
2.5	2.6751\\
2.51	2.6789\\
2.52	2.6827\\
2.53	2.6864\\
2.54	2.6902\\
2.55	2.694\\
2.56	2.6978\\
2.57	2.7016\\
2.58	2.7054\\
2.59	2.7093\\
2.6	2.7131\\
2.61	2.7169\\
2.62	2.7208\\
2.63	2.7246\\
2.64	2.7285\\
2.65	2.7324\\
2.66	2.7363\\
2.67	2.7401\\
2.68	2.744\\
2.69	2.7479\\
2.7	2.7518\\
2.71	2.7558\\
2.72	2.7597\\
2.73	2.7636\\
2.74	2.7675\\
2.75	2.7715\\
2.76	2.7754\\
2.77	2.7794\\
2.78	2.7833\\
2.79	2.7873\\
2.8	2.7913\\
2.81	2.7952\\
2.82	2.7992\\
2.83	2.8032\\
2.84	2.8072\\
2.85	2.8112\\
2.86	2.8152\\
2.87	2.8192\\
2.88	2.8232\\
2.89	2.8272\\
2.9	2.8312\\
2.91	2.8352\\
2.92	2.8392\\
2.93	2.8432\\
2.94	2.8472\\
2.95	2.8513\\
2.96	2.8553\\
2.97	2.8593\\
2.98	2.8634\\
2.99	2.8674\\
3	2.8714\\
3.01	2.8755\\
3.02	2.8795\\
3.03	2.8835\\
3.04	2.8876\\
3.05	2.8916\\
3.06	2.8957\\
3.07	2.8997\\
3.08	2.9038\\
3.09	2.9078\\
3.1	2.9119\\
3.11	2.9159\\
3.12	2.92\\
3.13	2.924\\
3.14	2.9281\\
3.15	2.9321\\
3.16	2.9362\\
3.17	2.9402\\
3.18	2.9443\\
3.19	2.9483\\
3.2	2.9523\\
3.21	2.9564\\
3.22	2.9604\\
3.23	2.9645\\
3.24	2.9685\\
3.25	2.9726\\
3.26	2.9766\\
3.27	2.9806\\
3.28	2.9847\\
3.29	2.9887\\
3.3	2.9927\\
3.31	2.9968\\
3.32	3.0008\\
3.33	3.0048\\
3.34	3.0088\\
3.35	3.0128\\
3.36	3.0169\\
3.37	3.0209\\
3.38	3.0249\\
3.39	3.0289\\
3.4	3.0329\\
3.41	3.0369\\
3.42	3.0409\\
3.43	3.0449\\
3.44	3.0489\\
3.45	3.0528\\
3.46	3.0568\\
3.47	3.0608\\
3.48	3.0648\\
3.49	3.0687\\
3.5	3.0727\\
3.51	3.0766\\
3.52	3.0806\\
3.53	3.0845\\
3.54	3.0885\\
3.55	3.0924\\
3.56	3.0963\\
3.57	3.1003\\
3.58	3.1042\\
3.59	3.1081\\
3.6	3.112\\
3.61	3.1159\\
3.62	3.1198\\
3.63	3.1237\\
3.64	3.1276\\
3.65	3.1315\\
3.66	3.1353\\
3.67	3.1392\\
3.68	3.143\\
3.69	3.1469\\
3.7	3.1507\\
3.71	3.1546\\
3.72	3.1584\\
3.73	3.1622\\
3.74	3.166\\
3.75	3.1698\\
3.76	3.1736\\
3.77	3.1774\\
3.78	3.1812\\
3.79	3.185\\
3.8	3.1888\\
3.81	3.1925\\
3.82	3.1963\\
3.83	3.2\\
3.84	3.2037\\
3.85	3.2075\\
3.86	3.2112\\
3.87	3.2149\\
3.88	3.2186\\
3.89	3.2223\\
3.9	3.226\\
3.91	3.2296\\
3.92	3.2333\\
3.93	3.2369\\
3.94	3.2406\\
3.95	3.2442\\
3.96	3.2479\\
3.97	3.2515\\
3.98	3.2551\\
3.99	3.2587\\
4	3.2623\\
4.01	3.2658\\
4.02	3.2694\\
4.03	3.273\\
4.04	3.2765\\
4.05	3.28\\
4.06	3.2836\\
4.07	3.2871\\
4.08	3.2906\\
4.09	3.2941\\
4.1	3.2975\\
4.11	3.301\\
4.12	3.3045\\
4.13	3.3079\\
4.14	3.3114\\
4.15	3.3148\\
4.16	3.3182\\
4.17	3.3216\\
4.18	3.325\\
4.19	3.3284\\
4.2	3.3317\\
4.21	3.3351\\
4.22	3.3384\\
4.23	3.3418\\
4.24	3.3451\\
4.25	3.3484\\
4.26	3.3517\\
4.27	3.355\\
4.28	3.3582\\
4.29	3.3615\\
4.3	3.3648\\
4.31	3.368\\
4.32	3.3712\\
4.33	3.3744\\
4.34	3.3776\\
4.35	3.3808\\
4.36	3.384\\
4.37	3.3871\\
4.38	3.3903\\
4.39	3.3934\\
4.4	3.3965\\
4.41	3.3996\\
4.42	3.4027\\
4.43	3.4058\\
4.44	3.4089\\
4.45	3.4119\\
4.46	3.4149\\
4.47	3.418\\
4.48	3.421\\
4.49	3.424\\
4.5	3.427\\
4.51	3.4299\\
4.52	3.4329\\
4.53	3.4358\\
4.54	3.4387\\
4.55	3.4417\\
4.56	3.4446\\
4.57	3.4474\\
4.58	3.4503\\
4.59	3.4532\\
4.6	3.456\\
4.61	3.4588\\
4.62	3.4616\\
4.63	3.4644\\
4.64	3.4672\\
4.65	3.47\\
4.66	3.4727\\
4.67	3.4755\\
4.68	3.4782\\
4.69	3.4809\\
4.7	3.4836\\
4.71	3.4863\\
4.72	3.4889\\
4.73	3.4916\\
4.74	3.4942\\
4.75	3.4968\\
4.76	3.4994\\
4.77	3.502\\
4.78	3.5046\\
4.79	3.5071\\
4.8	3.5097\\
4.81	3.5122\\
4.82	3.5147\\
4.83	3.5172\\
4.84	3.5197\\
4.85	3.5221\\
4.86	3.5246\\
4.87	3.527\\
4.88	3.5294\\
4.89	3.5318\\
4.9	3.5342\\
4.91	3.5366\\
4.92	3.5389\\
4.93	3.5413\\
4.94	3.5436\\
4.95	3.5459\\
4.96	3.5482\\
4.97	3.5504\\
4.98	3.5527\\
4.99	3.5549\\
5	3.5571\\
5.01	3.5593\\
5.02	3.5615\\
5.03	3.5637\\
5.04	3.5659\\
5.05	3.568\\
5.06	3.5701\\
5.07	3.5722\\
5.08	3.5743\\
5.09	3.5764\\
5.1	3.5784\\
5.11	3.5805\\
5.12	3.5825\\
5.13	3.5845\\
5.14	3.5865\\
5.15	3.5885\\
5.16	3.5904\\
5.17	3.5924\\
5.18	3.5943\\
5.19	3.5962\\
5.2	3.5981\\
5.21	3.5999\\
5.22	3.6018\\
5.23	3.6036\\
5.24	3.6054\\
5.25	3.6072\\
5.26	3.609\\
5.27	3.6108\\
5.28	3.6125\\
5.29	3.6143\\
5.3	3.616\\
5.31	3.6177\\
5.32	3.6194\\
5.33	3.621\\
5.34	3.6227\\
5.35	3.6243\\
5.36	3.6259\\
5.37	3.6275\\
5.38	3.6291\\
5.39	3.6307\\
5.4	3.6322\\
5.41	3.6337\\
5.42	3.6352\\
5.43	3.6367\\
5.44	3.6382\\
5.45	3.6397\\
5.46	3.6411\\
5.47	3.6425\\
5.48	3.6439\\
5.49	3.6453\\
5.5	3.6467\\
5.51	3.648\\
5.52	3.6494\\
5.53	3.6507\\
5.54	3.652\\
5.55	3.6533\\
5.56	3.6545\\
5.57	3.6558\\
5.58	3.657\\
5.59	3.6582\\
5.6	3.6594\\
5.61	3.6606\\
5.62	3.6618\\
5.63	3.6629\\
5.64	3.664\\
5.65	3.6651\\
5.66	3.6662\\
5.67	3.6673\\
5.68	3.6684\\
5.69	3.6694\\
5.7	3.6704\\
5.71	3.6714\\
5.72	3.6724\\
5.73	3.6734\\
5.74	3.6743\\
5.75	3.6753\\
5.76	3.6762\\
5.77	3.6771\\
5.78	3.678\\
5.79	3.6788\\
5.8	3.6797\\
5.81	3.6805\\
5.82	3.6813\\
5.83	3.6821\\
5.84	3.6829\\
5.85	3.6836\\
5.86	3.6844\\
5.87	3.6851\\
5.88	3.6858\\
5.89	3.6865\\
5.9	3.6872\\
5.91	3.6878\\
5.92	3.6885\\
5.93	3.6891\\
5.94	3.6897\\
5.95	3.6903\\
5.96	3.6909\\
5.97	3.6914\\
5.98	3.6919\\
5.99	3.6925\\
6	3.693\\
6.01	3.6935\\
6.02	3.6939\\
6.03	3.6944\\
6.04	3.6948\\
6.05	3.6952\\
6.06	3.6956\\
6.07	3.696\\
6.08	3.6964\\
6.09	3.6967\\
6.1	3.6971\\
6.11	3.6974\\
6.12	3.6977\\
6.13	3.6979\\
6.14	3.6982\\
6.15	3.6985\\
6.16	3.6987\\
6.17	3.6989\\
6.18	3.6991\\
6.19	3.6993\\
6.2	3.6995\\
6.21	3.6996\\
6.22	3.6997\\
6.23	3.6998\\
6.24	3.6999\\
6.25	3.7\\
6.26	3.7001\\
6.27	3.7001\\
6.28	3.7002\\
6.29	3.7002\\
6.3	3.7002\\
6.31	3.7002\\
6.32	3.7002\\
6.33	3.7001\\
6.34	3.7\\
6.35	3.7\\
6.36	3.6999\\
6.37	3.6998\\
6.38	3.6996\\
6.39	3.6995\\
6.4	3.6993\\
6.41	3.6992\\
6.42	3.699\\
6.43	3.6988\\
6.44	3.6985\\
6.45	3.6983\\
6.46	3.6981\\
6.47	3.6978\\
6.48	3.6975\\
6.49	3.6972\\
6.5	3.6969\\
6.51	3.6966\\
6.52	3.6962\\
6.53	3.6959\\
6.54	3.6955\\
6.55	3.6951\\
6.56	3.6947\\
6.57	3.6943\\
6.58	3.6938\\
6.59	3.6934\\
6.6	3.6929\\
6.61	3.6924\\
6.62	3.6919\\
6.63	3.6914\\
6.64	3.6909\\
6.65	3.6904\\
6.66	3.6898\\
6.67	3.6893\\
6.68	3.6887\\
6.69	3.6881\\
6.7	3.6875\\
6.71	3.6868\\
6.72	3.6862\\
6.73	3.6855\\
6.74	3.6849\\
6.75	3.6842\\
6.76	3.6835\\
6.77	3.6828\\
6.78	3.6821\\
6.79	3.6813\\
6.8	3.6806\\
6.81	3.6798\\
6.82	3.6791\\
6.83	3.6783\\
6.84	3.6775\\
6.85	3.6766\\
6.86	3.6758\\
6.87	3.675\\
6.88	3.6741\\
6.89	3.6733\\
6.9	3.6724\\
6.91	3.6715\\
6.92	3.6706\\
6.93	3.6696\\
6.94	3.6687\\
6.95	3.6678\\
6.96	3.6668\\
6.97	3.6658\\
6.98	3.6649\\
6.99	3.6639\\
7	3.6629\\
7.01	3.6618\\
7.02	3.6608\\
7.03	3.6598\\
7.04	3.6587\\
7.05	3.6576\\
7.06	3.6566\\
7.07	3.6555\\
7.08	3.6544\\
7.09	3.6533\\
7.1	3.6521\\
7.11	3.651\\
7.12	3.6498\\
7.13	3.6487\\
7.14	3.6475\\
7.15	3.6463\\
7.16	3.6452\\
7.17	3.6439\\
7.18	3.6427\\
7.19	3.6415\\
7.2	3.6403\\
7.21	3.639\\
7.22	3.6378\\
7.23	3.6365\\
7.24	3.6352\\
7.25	3.634\\
7.26	3.6327\\
7.27	3.6314\\
7.28	3.63\\
7.29	3.6287\\
7.3	3.6274\\
7.31	3.626\\
7.32	3.6247\\
7.33	3.6233\\
7.34	3.6219\\
7.35	3.6206\\
7.36	3.6192\\
7.37	3.6178\\
7.38	3.6164\\
7.39	3.6149\\
7.4	3.6135\\
7.41	3.6121\\
7.42	3.6106\\
7.43	3.6092\\
7.44	3.6077\\
7.45	3.6063\\
7.46	3.6048\\
7.47	3.6033\\
7.48	3.6018\\
7.49	3.6003\\
7.5	3.5988\\
7.51	3.5973\\
7.52	3.5958\\
7.53	3.5942\\
7.54	3.5927\\
7.55	3.5912\\
7.56	3.5896\\
7.57	3.5881\\
7.58	3.5865\\
7.59	3.5849\\
7.6	3.5833\\
7.61	3.5818\\
7.62	3.5802\\
7.63	3.5786\\
7.64	3.577\\
7.65	3.5754\\
7.66	3.5737\\
7.67	3.5721\\
7.68	3.5705\\
7.69	3.5689\\
7.7	3.5672\\
7.71	3.5656\\
7.72	3.5639\\
7.73	3.5623\\
7.74	3.5606\\
7.75	3.559\\
7.76	3.5573\\
7.77	3.5556\\
7.78	3.5539\\
7.79	3.5523\\
7.8	3.5506\\
7.81	3.5489\\
7.82	3.5472\\
7.83	3.5455\\
7.84	3.5438\\
7.85	3.5421\\
7.86	3.5404\\
7.87	3.5387\\
7.88	3.537\\
7.89	3.5352\\
7.9	3.5335\\
7.91	3.5318\\
7.92	3.5301\\
7.93	3.5283\\
7.94	3.5266\\
7.95	3.5249\\
7.96	3.5231\\
7.97	3.5214\\
7.98	3.5197\\
7.99	3.5179\\
8	3.5162\\
8.01	3.5144\\
8.02	3.5127\\
8.03	3.5109\\
8.04	3.5092\\
8.05	3.5074\\
8.06	3.5057\\
8.07	3.5039\\
8.08	3.5022\\
8.09	3.5004\\
8.1	3.4987\\
8.11	3.4969\\
8.12	3.4952\\
8.13	3.4934\\
8.14	3.4917\\
8.15	3.4899\\
8.16	3.4882\\
8.17	3.4864\\
8.18	3.4847\\
8.19	3.4829\\
8.2	3.4812\\
8.21	3.4794\\
8.22	3.4777\\
8.23	3.4759\\
8.24	3.4742\\
8.25	3.4724\\
8.26	3.4707\\
8.27	3.469\\
8.28	3.4672\\
8.29	3.4655\\
8.3	3.4638\\
8.31	3.462\\
8.32	3.4603\\
8.33	3.4586\\
8.34	3.4569\\
8.35	3.4552\\
8.36	3.4535\\
8.37	3.4517\\
8.38	3.45\\
8.39	3.4483\\
8.4	3.4466\\
8.41	3.445\\
8.42	3.4433\\
8.43	3.4416\\
8.44	3.4399\\
8.45	3.4382\\
8.46	3.4366\\
8.47	3.4349\\
8.48	3.4332\\
8.49	3.4316\\
8.5	3.4299\\
8.51	3.4283\\
8.52	3.4267\\
8.53	3.425\\
8.54	3.4234\\
8.55	3.4218\\
8.56	3.4202\\
8.57	3.4186\\
8.58	3.417\\
8.59	3.4154\\
8.6	3.4138\\
8.61	3.4123\\
8.62	3.4107\\
8.63	3.4091\\
8.64	3.4076\\
8.65	3.406\\
8.66	3.4045\\
8.67	3.403\\
8.68	3.4015\\
8.69	3.4\\
8.7	3.3985\\
8.71	3.397\\
8.72	3.3955\\
8.73	3.394\\
8.74	3.3926\\
8.75	3.3911\\
8.76	3.3897\\
8.77	3.3882\\
8.78	3.3868\\
8.79	3.3854\\
8.8	3.384\\
8.81	3.3826\\
8.82	3.3812\\
8.83	3.3799\\
8.84	3.3785\\
8.85	3.3772\\
8.86	3.3759\\
8.87	3.3745\\
8.88	3.3732\\
8.89	3.3719\\
8.9	3.3707\\
8.91	3.3694\\
8.92	3.3681\\
8.93	3.3669\\
8.94	3.3657\\
8.95	3.3644\\
8.96	3.3632\\
8.97	3.3621\\
8.98	3.3609\\
8.99	3.3597\\
9	3.3586\\
9.01	3.3574\\
9.02	3.3563\\
9.03	3.3552\\
9.04	3.3541\\
9.05	3.3531\\
9.06	3.352\\
9.07	3.351\\
9.08	3.35\\
9.09	3.3489\\
9.1	3.348\\
9.11	3.347\\
9.12	3.346\\
9.13	3.3451\\
9.14	3.3442\\
9.15	3.3433\\
9.16	3.3424\\
9.17	3.3415\\
9.18	3.3406\\
9.19	3.3398\\
9.2	3.339\\
9.21	3.3382\\
9.22	3.3374\\
9.23	3.3367\\
9.24	3.3359\\
9.25	3.3352\\
9.26	3.3345\\
9.27	3.3338\\
9.28	3.3331\\
9.29	3.3325\\
9.3	3.3319\\
9.31	3.3313\\
9.32	3.3307\\
9.33	3.3301\\
9.34	3.3296\\
9.35	3.3291\\
9.36	3.3286\\
9.37	3.3281\\
9.38	3.3277\\
9.39	3.3272\\
9.4	3.3268\\
9.41	3.3265\\
9.42	3.3261\\
9.43	3.3258\\
9.44	3.3254\\
9.45	3.3251\\
9.46	3.3249\\
9.47	3.3246\\
9.48	3.3244\\
9.49	3.3242\\
9.5	3.3241\\
9.51	3.3239\\
9.52	3.3238\\
9.53	3.3237\\
9.54	3.3236\\
9.55	3.3236\\
9.56	3.3236\\
9.57	3.3236\\
9.58	3.3236\\
9.59	3.3237\\
9.6	3.3238\\
9.61	3.3239\\
9.62	3.324\\
9.63	3.3242\\
9.64	3.3244\\
9.65	3.3246\\
9.66	3.3249\\
9.67	3.3251\\
9.68	3.3255\\
9.69	3.3258\\
9.7	3.3262\\
9.71	3.3266\\
9.72	3.327\\
9.73	3.3274\\
9.74	3.3279\\
9.75	3.3284\\
9.76	3.329\\
9.77	3.3296\\
9.78	3.3302\\
9.79	3.3308\\
9.8	3.3315\\
9.81	3.3322\\
9.82	3.3329\\
9.83	3.3337\\
9.84	3.3345\\
9.85	3.3353\\
9.86	3.3362\\
9.87	3.3371\\
9.88	3.338\\
9.89	3.339\\
9.9	3.34\\
9.91	3.341\\
9.92	3.3421\\
9.93	3.3432\\
9.94	3.3443\\
9.95	3.3455\\
9.96	3.3467\\
9.97	3.348\\
9.98	3.3492\\
9.99	3.3505\\
10	3.3519\\
10.01	3.3533\\
10.02	3.3547\\
10.03	3.3562\\
10.04	3.3577\\
10.05	3.3592\\
10.06	3.3608\\
10.07	3.3624\\
10.08	3.364\\
10.09	3.3657\\
10.1	3.3674\\
10.11	3.3692\\
10.12	3.371\\
10.13	3.3729\\
10.14	3.3747\\
10.15	3.3767\\
10.16	3.3786\\
10.17	3.3806\\
10.18	3.3827\\
10.19	3.3848\\
10.2	3.3869\\
10.21	3.389\\
10.22	3.3913\\
10.23	3.3935\\
10.24	3.3958\\
10.25	3.3981\\
10.26	3.4005\\
10.27	3.4029\\
10.28	3.4054\\
10.29	3.4079\\
10.3	3.4105\\
10.31	3.4131\\
10.32	3.4157\\
10.33	3.4184\\
10.34	3.4211\\
10.35	3.4239\\
10.36	3.4267\\
10.37	3.4296\\
10.38	3.4325\\
10.39	3.4355\\
10.4	3.4385\\
10.41	3.4415\\
10.42	3.4446\\
10.43	3.4478\\
10.44	3.451\\
10.45	3.4542\\
10.46	3.4575\\
10.47	3.4608\\
10.48	3.4642\\
10.49	3.4676\\
10.5	3.4711\\
10.51	3.4747\\
10.52	3.4782\\
10.53	3.4819\\
10.54	3.4856\\
10.55	3.4893\\
10.56	3.4931\\
10.57	3.4969\\
10.58	3.5008\\
10.59	3.5048\\
10.6	3.5087\\
10.61	3.5128\\
10.62	3.5169\\
10.63	3.521\\
10.64	3.5252\\
10.65	3.5295\\
10.66	3.5338\\
10.67	3.5382\\
10.68	3.5426\\
10.69	3.5471\\
10.7	3.5516\\
10.71	3.5562\\
10.72	3.5608\\
10.73	3.5655\\
10.74	3.5702\\
10.75	3.575\\
10.76	3.5799\\
10.77	3.5848\\
10.78	3.5898\\
10.79	3.5948\\
10.8	3.5999\\
10.81	3.6051\\
10.82	3.6103\\
10.83	3.6155\\
10.84	3.6209\\
10.85	3.6262\\
10.86	3.6317\\
10.87	3.6372\\
10.88	3.6428\\
10.89	3.6484\\
10.9	3.6541\\
10.91	3.6598\\
10.92	3.6656\\
10.93	3.6715\\
10.94	3.6774\\
10.95	3.6834\\
10.96	3.6894\\
10.97	3.6956\\
10.98	3.7017\\
10.99	3.708\\
11	3.7143\\
};
\addplot [color=gray,dashed,forget plot]
  table[row sep=crcr]{%
1	0.83571\\
1.01	0.8357\\
1.02	0.83572\\
1.03	0.83579\\
1.04	0.8359\\
1.05	0.83605\\
1.06	0.83624\\
1.07	0.83648\\
1.08	0.83675\\
1.09	0.83707\\
1.1	0.83742\\
1.11	0.83782\\
1.12	0.83825\\
1.13	0.83873\\
1.14	0.83924\\
1.15	0.83979\\
1.16	0.84038\\
1.17	0.84102\\
1.18	0.84168\\
1.19	0.84239\\
1.2	0.84313\\
1.21	0.84392\\
1.22	0.84474\\
1.23	0.84559\\
1.24	0.84648\\
1.25	0.84741\\
1.26	0.84838\\
1.27	0.84938\\
1.28	0.85042\\
1.29	0.85149\\
1.3	0.8526\\
1.31	0.85375\\
1.32	0.85492\\
1.33	0.85614\\
1.34	0.85738\\
1.35	0.85867\\
1.36	0.85998\\
1.37	0.86133\\
1.38	0.86271\\
1.39	0.86413\\
1.4	0.86558\\
1.41	0.86706\\
1.42	0.86857\\
1.43	0.87012\\
1.44	0.8717\\
1.45	0.87331\\
1.46	0.87495\\
1.47	0.87662\\
1.48	0.87832\\
1.49	0.88006\\
1.5	0.88182\\
1.51	0.88362\\
1.52	0.88544\\
1.53	0.8873\\
1.54	0.88918\\
1.55	0.89109\\
1.56	0.89304\\
1.57	0.89501\\
1.58	0.89701\\
1.59	0.89904\\
1.6	0.9011\\
1.61	0.90318\\
1.62	0.9053\\
1.63	0.90744\\
1.64	0.90961\\
1.65	0.9118\\
1.66	0.91402\\
1.67	0.91627\\
1.68	0.91855\\
1.69	0.92085\\
1.7	0.92318\\
1.71	0.92553\\
1.72	0.92791\\
1.73	0.93032\\
1.74	0.93275\\
1.75	0.9352\\
1.76	0.93768\\
1.77	0.94018\\
1.78	0.94271\\
1.79	0.94526\\
1.8	0.94784\\
1.81	0.95044\\
1.82	0.95306\\
1.83	0.95571\\
1.84	0.95838\\
1.85	0.96107\\
1.86	0.96378\\
1.87	0.96652\\
1.88	0.96928\\
1.89	0.97206\\
1.9	0.97486\\
1.91	0.97769\\
1.92	0.98053\\
1.93	0.9834\\
1.94	0.98629\\
1.95	0.9892\\
1.96	0.99213\\
1.97	0.99507\\
1.98	0.99804\\
1.99	1.001\\
2	1.004\\
2.01	1.0071\\
2.02	1.0101\\
2.03	1.0132\\
2.04	1.0163\\
2.05	1.0194\\
2.06	1.0225\\
2.07	1.0256\\
2.08	1.0288\\
2.09	1.032\\
2.1	1.0352\\
2.11	1.0384\\
2.12	1.0416\\
2.13	1.0448\\
2.14	1.0481\\
2.15	1.0514\\
2.16	1.0547\\
2.17	1.058\\
2.18	1.0613\\
2.19	1.0647\\
2.2	1.068\\
2.21	1.0714\\
2.22	1.0748\\
2.23	1.0782\\
2.24	1.0816\\
2.25	1.085\\
2.26	1.0885\\
2.27	1.092\\
2.28	1.0954\\
2.29	1.0989\\
2.3	1.1024\\
2.31	1.106\\
2.32	1.1095\\
2.33	1.113\\
2.34	1.1166\\
2.35	1.1202\\
2.36	1.1237\\
2.37	1.1273\\
2.38	1.1309\\
2.39	1.1346\\
2.4	1.1382\\
2.41	1.1418\\
2.42	1.1455\\
2.43	1.1492\\
2.44	1.1528\\
2.45	1.1565\\
2.46	1.1602\\
2.47	1.164\\
2.48	1.1677\\
2.49	1.1714\\
2.5	1.1751\\
2.51	1.1789\\
2.52	1.1827\\
2.53	1.1864\\
2.54	1.1902\\
2.55	1.194\\
2.56	1.1978\\
2.57	1.2016\\
2.58	1.2054\\
2.59	1.2093\\
2.6	1.2131\\
2.61	1.2169\\
2.62	1.2208\\
2.63	1.2246\\
2.64	1.2285\\
2.65	1.2324\\
2.66	1.2363\\
2.67	1.2401\\
2.68	1.244\\
2.69	1.2479\\
2.7	1.2518\\
2.71	1.2558\\
2.72	1.2597\\
2.73	1.2636\\
2.74	1.2675\\
2.75	1.2715\\
2.76	1.2754\\
2.77	1.2794\\
2.78	1.2833\\
2.79	1.2873\\
2.8	1.2913\\
2.81	1.2952\\
2.82	1.2992\\
2.83	1.3032\\
2.84	1.3072\\
2.85	1.3112\\
2.86	1.3152\\
2.87	1.3192\\
2.88	1.3232\\
2.89	1.3272\\
2.9	1.3312\\
2.91	1.3352\\
2.92	1.3392\\
2.93	1.3432\\
2.94	1.3472\\
2.95	1.3513\\
2.96	1.3553\\
2.97	1.3593\\
2.98	1.3634\\
2.99	1.3674\\
3	1.3714\\
3.01	1.3755\\
3.02	1.3795\\
3.03	1.3835\\
3.04	1.3876\\
3.05	1.3916\\
3.06	1.3957\\
3.07	1.3997\\
3.08	1.4038\\
3.09	1.4078\\
3.1	1.4119\\
3.11	1.4159\\
3.12	1.42\\
3.13	1.424\\
3.14	1.4281\\
3.15	1.4321\\
3.16	1.4362\\
3.17	1.4402\\
3.18	1.4443\\
3.19	1.4483\\
3.2	1.4523\\
3.21	1.4564\\
3.22	1.4604\\
3.23	1.4645\\
3.24	1.4685\\
3.25	1.4726\\
3.26	1.4766\\
3.27	1.4806\\
3.28	1.4847\\
3.29	1.4887\\
3.3	1.4927\\
3.31	1.4968\\
3.32	1.5008\\
3.33	1.5048\\
3.34	1.5088\\
3.35	1.5128\\
3.36	1.5169\\
3.37	1.5209\\
3.38	1.5249\\
3.39	1.5289\\
3.4	1.5329\\
3.41	1.5369\\
3.42	1.5409\\
3.43	1.5449\\
3.44	1.5489\\
3.45	1.5528\\
3.46	1.5568\\
3.47	1.5608\\
3.48	1.5648\\
3.49	1.5687\\
3.5	1.5727\\
3.51	1.5766\\
3.52	1.5806\\
3.53	1.5845\\
3.54	1.5885\\
3.55	1.5924\\
3.56	1.5963\\
3.57	1.6003\\
3.58	1.6042\\
3.59	1.6081\\
3.6	1.612\\
3.61	1.6159\\
3.62	1.6198\\
3.63	1.6237\\
3.64	1.6276\\
3.65	1.6315\\
3.66	1.6353\\
3.67	1.6392\\
3.68	1.643\\
3.69	1.6469\\
3.7	1.6507\\
3.71	1.6546\\
3.72	1.6584\\
3.73	1.6622\\
3.74	1.666\\
3.75	1.6698\\
3.76	1.6736\\
3.77	1.6774\\
3.78	1.6812\\
3.79	1.685\\
3.8	1.6888\\
3.81	1.6925\\
3.82	1.6963\\
3.83	1.7\\
3.84	1.7037\\
3.85	1.7075\\
3.86	1.7112\\
3.87	1.7149\\
3.88	1.7186\\
3.89	1.7223\\
3.9	1.726\\
3.91	1.7296\\
3.92	1.7333\\
3.93	1.7369\\
3.94	1.7406\\
3.95	1.7442\\
3.96	1.7479\\
3.97	1.7515\\
3.98	1.7551\\
3.99	1.7587\\
4	1.7623\\
4.01	1.7658\\
4.02	1.7694\\
4.03	1.773\\
4.04	1.7765\\
4.05	1.78\\
4.06	1.7836\\
4.07	1.7871\\
4.08	1.7906\\
4.09	1.7941\\
4.1	1.7975\\
4.11	1.801\\
4.12	1.8045\\
4.13	1.8079\\
4.14	1.8114\\
4.15	1.8148\\
4.16	1.8182\\
4.17	1.8216\\
4.18	1.825\\
4.19	1.8284\\
4.2	1.8317\\
4.21	1.8351\\
4.22	1.8384\\
4.23	1.8418\\
4.24	1.8451\\
4.25	1.8484\\
4.26	1.8517\\
4.27	1.855\\
4.28	1.8582\\
4.29	1.8615\\
4.3	1.8648\\
4.31	1.868\\
4.32	1.8712\\
4.33	1.8744\\
4.34	1.8776\\
4.35	1.8808\\
4.36	1.884\\
4.37	1.8871\\
4.38	1.8903\\
4.39	1.8934\\
4.4	1.8965\\
4.41	1.8996\\
4.42	1.9027\\
4.43	1.9058\\
4.44	1.9089\\
4.45	1.9119\\
4.46	1.9149\\
4.47	1.918\\
4.48	1.921\\
4.49	1.924\\
4.5	1.927\\
4.51	1.9299\\
4.52	1.9329\\
4.53	1.9358\\
4.54	1.9387\\
4.55	1.9417\\
4.56	1.9446\\
4.57	1.9474\\
4.58	1.9503\\
4.59	1.9532\\
4.6	1.956\\
4.61	1.9588\\
4.62	1.9616\\
4.63	1.9644\\
4.64	1.9672\\
4.65	1.97\\
4.66	1.9727\\
4.67	1.9755\\
4.68	1.9782\\
4.69	1.9809\\
4.7	1.9836\\
4.71	1.9863\\
4.72	1.9889\\
4.73	1.9916\\
4.74	1.9942\\
4.75	1.9968\\
4.76	1.9994\\
4.77	2.002\\
4.78	2.0046\\
4.79	2.0071\\
4.8	2.0097\\
4.81	2.0122\\
4.82	2.0147\\
4.83	2.0172\\
4.84	2.0197\\
4.85	2.0221\\
4.86	2.0246\\
4.87	2.027\\
4.88	2.0294\\
4.89	2.0318\\
4.9	2.0342\\
4.91	2.0366\\
4.92	2.0389\\
4.93	2.0413\\
4.94	2.0436\\
4.95	2.0459\\
4.96	2.0482\\
4.97	2.0504\\
4.98	2.0527\\
4.99	2.0549\\
5	2.0571\\
5.01	2.0593\\
5.02	2.0615\\
5.03	2.0637\\
5.04	2.0659\\
5.05	2.068\\
5.06	2.0701\\
5.07	2.0722\\
5.08	2.0743\\
5.09	2.0764\\
5.1	2.0784\\
5.11	2.0805\\
5.12	2.0825\\
5.13	2.0845\\
5.14	2.0865\\
5.15	2.0885\\
5.16	2.0904\\
5.17	2.0924\\
5.18	2.0943\\
5.19	2.0962\\
5.2	2.0981\\
5.21	2.0999\\
5.22	2.1018\\
5.23	2.1036\\
5.24	2.1054\\
5.25	2.1072\\
5.26	2.109\\
5.27	2.1108\\
5.28	2.1125\\
5.29	2.1143\\
5.3	2.116\\
5.31	2.1177\\
5.32	2.1194\\
5.33	2.121\\
5.34	2.1227\\
5.35	2.1243\\
5.36	2.1259\\
5.37	2.1275\\
5.38	2.1291\\
5.39	2.1307\\
5.4	2.1322\\
5.41	2.1337\\
5.42	2.1352\\
5.43	2.1367\\
5.44	2.1382\\
5.45	2.1397\\
5.46	2.1411\\
5.47	2.1425\\
5.48	2.1439\\
5.49	2.1453\\
5.5	2.1467\\
5.51	2.148\\
5.52	2.1494\\
5.53	2.1507\\
5.54	2.152\\
5.55	2.1533\\
5.56	2.1545\\
5.57	2.1558\\
5.58	2.157\\
5.59	2.1582\\
5.6	2.1594\\
5.61	2.1606\\
5.62	2.1618\\
5.63	2.1629\\
5.64	2.164\\
5.65	2.1651\\
5.66	2.1662\\
5.67	2.1673\\
5.68	2.1684\\
5.69	2.1694\\
5.7	2.1704\\
5.71	2.1714\\
5.72	2.1724\\
5.73	2.1734\\
5.74	2.1743\\
5.75	2.1753\\
5.76	2.1762\\
5.77	2.1771\\
5.78	2.178\\
5.79	2.1788\\
5.8	2.1797\\
5.81	2.1805\\
5.82	2.1813\\
5.83	2.1821\\
5.84	2.1829\\
5.85	2.1836\\
5.86	2.1844\\
5.87	2.1851\\
5.88	2.1858\\
5.89	2.1865\\
5.9	2.1872\\
5.91	2.1878\\
5.92	2.1885\\
5.93	2.1891\\
5.94	2.1897\\
5.95	2.1903\\
5.96	2.1909\\
5.97	2.1914\\
5.98	2.1919\\
5.99	2.1925\\
6	2.193\\
6.01	2.1935\\
6.02	2.1939\\
6.03	2.1944\\
6.04	2.1948\\
6.05	2.1952\\
6.06	2.1956\\
6.07	2.196\\
6.08	2.1964\\
6.09	2.1967\\
6.1	2.1971\\
6.11	2.1974\\
6.12	2.1977\\
6.13	2.1979\\
6.14	2.1982\\
6.15	2.1985\\
6.16	2.1987\\
6.17	2.1989\\
6.18	2.1991\\
6.19	2.1993\\
6.2	2.1995\\
6.21	2.1996\\
6.22	2.1997\\
6.23	2.1998\\
6.24	2.1999\\
6.25	2.2\\
6.26	2.2001\\
6.27	2.2001\\
6.28	2.2002\\
6.29	2.2002\\
6.3	2.2002\\
6.31	2.2002\\
6.32	2.2002\\
6.33	2.2001\\
6.34	2.2\\
6.35	2.2\\
6.36	2.1999\\
6.37	2.1998\\
6.38	2.1996\\
6.39	2.1995\\
6.4	2.1993\\
6.41	2.1992\\
6.42	2.199\\
6.43	2.1988\\
6.44	2.1985\\
6.45	2.1983\\
6.46	2.1981\\
6.47	2.1978\\
6.48	2.1975\\
6.49	2.1972\\
6.5	2.1969\\
6.51	2.1966\\
6.52	2.1962\\
6.53	2.1959\\
6.54	2.1955\\
6.55	2.1951\\
6.56	2.1947\\
6.57	2.1943\\
6.58	2.1938\\
6.59	2.1934\\
6.6	2.1929\\
6.61	2.1924\\
6.62	2.1919\\
6.63	2.1914\\
6.64	2.1909\\
6.65	2.1904\\
6.66	2.1898\\
6.67	2.1893\\
6.68	2.1887\\
6.69	2.1881\\
6.7	2.1875\\
6.71	2.1868\\
6.72	2.1862\\
6.73	2.1855\\
6.74	2.1849\\
6.75	2.1842\\
6.76	2.1835\\
6.77	2.1828\\
6.78	2.1821\\
6.79	2.1813\\
6.8	2.1806\\
6.81	2.1798\\
6.82	2.1791\\
6.83	2.1783\\
6.84	2.1775\\
6.85	2.1766\\
6.86	2.1758\\
6.87	2.175\\
6.88	2.1741\\
6.89	2.1733\\
6.9	2.1724\\
6.91	2.1715\\
6.92	2.1706\\
6.93	2.1696\\
6.94	2.1687\\
6.95	2.1678\\
6.96	2.1668\\
6.97	2.1658\\
6.98	2.1649\\
6.99	2.1639\\
7	2.1629\\
7.01	2.1618\\
7.02	2.1608\\
7.03	2.1598\\
7.04	2.1587\\
7.05	2.1576\\
7.06	2.1566\\
7.07	2.1555\\
7.08	2.1544\\
7.09	2.1533\\
7.1	2.1521\\
7.11	2.151\\
7.12	2.1498\\
7.13	2.1487\\
7.14	2.1475\\
7.15	2.1463\\
7.16	2.1452\\
7.17	2.1439\\
7.18	2.1427\\
7.19	2.1415\\
7.2	2.1403\\
7.21	2.139\\
7.22	2.1378\\
7.23	2.1365\\
7.24	2.1352\\
7.25	2.134\\
7.26	2.1327\\
7.27	2.1314\\
7.28	2.13\\
7.29	2.1287\\
7.3	2.1274\\
7.31	2.126\\
7.32	2.1247\\
7.33	2.1233\\
7.34	2.1219\\
7.35	2.1206\\
7.36	2.1192\\
7.37	2.1178\\
7.38	2.1164\\
7.39	2.1149\\
7.4	2.1135\\
7.41	2.1121\\
7.42	2.1106\\
7.43	2.1092\\
7.44	2.1077\\
7.45	2.1063\\
7.46	2.1048\\
7.47	2.1033\\
7.48	2.1018\\
7.49	2.1003\\
7.5	2.0988\\
7.51	2.0973\\
7.52	2.0958\\
7.53	2.0942\\
7.54	2.0927\\
7.55	2.0912\\
7.56	2.0896\\
7.57	2.0881\\
7.58	2.0865\\
7.59	2.0849\\
7.6	2.0833\\
7.61	2.0818\\
7.62	2.0802\\
7.63	2.0786\\
7.64	2.077\\
7.65	2.0754\\
7.66	2.0737\\
7.67	2.0721\\
7.68	2.0705\\
7.69	2.0689\\
7.7	2.0672\\
7.71	2.0656\\
7.72	2.0639\\
7.73	2.0623\\
7.74	2.0606\\
7.75	2.059\\
7.76	2.0573\\
7.77	2.0556\\
7.78	2.0539\\
7.79	2.0523\\
7.8	2.0506\\
7.81	2.0489\\
7.82	2.0472\\
7.83	2.0455\\
7.84	2.0438\\
7.85	2.0421\\
7.86	2.0404\\
7.87	2.0387\\
7.88	2.037\\
7.89	2.0352\\
7.9	2.0335\\
7.91	2.0318\\
7.92	2.0301\\
7.93	2.0283\\
7.94	2.0266\\
7.95	2.0249\\
7.96	2.0231\\
7.97	2.0214\\
7.98	2.0197\\
7.99	2.0179\\
8	2.0162\\
8.01	2.0144\\
8.02	2.0127\\
8.03	2.0109\\
8.04	2.0092\\
8.05	2.0074\\
8.06	2.0057\\
8.07	2.0039\\
8.08	2.0022\\
8.09	2.0004\\
8.1	1.9987\\
8.11	1.9969\\
8.12	1.9952\\
8.13	1.9934\\
8.14	1.9917\\
8.15	1.9899\\
8.16	1.9882\\
8.17	1.9864\\
8.18	1.9847\\
8.19	1.9829\\
8.2	1.9812\\
8.21	1.9794\\
8.22	1.9777\\
8.23	1.9759\\
8.24	1.9742\\
8.25	1.9724\\
8.26	1.9707\\
8.27	1.969\\
8.28	1.9672\\
8.29	1.9655\\
8.3	1.9638\\
8.31	1.962\\
8.32	1.9603\\
8.33	1.9586\\
8.34	1.9569\\
8.35	1.9552\\
8.36	1.9535\\
8.37	1.9517\\
8.38	1.95\\
8.39	1.9483\\
8.4	1.9466\\
8.41	1.945\\
8.42	1.9433\\
8.43	1.9416\\
8.44	1.9399\\
8.45	1.9382\\
8.46	1.9366\\
8.47	1.9349\\
8.48	1.9332\\
8.49	1.9316\\
8.5	1.9299\\
8.51	1.9283\\
8.52	1.9267\\
8.53	1.925\\
8.54	1.9234\\
8.55	1.9218\\
8.56	1.9202\\
8.57	1.9186\\
8.58	1.917\\
8.59	1.9154\\
8.6	1.9138\\
8.61	1.9123\\
8.62	1.9107\\
8.63	1.9091\\
8.64	1.9076\\
8.65	1.906\\
8.66	1.9045\\
8.67	1.903\\
8.68	1.9015\\
8.69	1.9\\
8.7	1.8985\\
8.71	1.897\\
8.72	1.8955\\
8.73	1.894\\
8.74	1.8926\\
8.75	1.8911\\
8.76	1.8897\\
8.77	1.8882\\
8.78	1.8868\\
8.79	1.8854\\
8.8	1.884\\
8.81	1.8826\\
8.82	1.8812\\
8.83	1.8799\\
8.84	1.8785\\
8.85	1.8772\\
8.86	1.8759\\
8.87	1.8745\\
8.88	1.8732\\
8.89	1.8719\\
8.9	1.8707\\
8.91	1.8694\\
8.92	1.8681\\
8.93	1.8669\\
8.94	1.8657\\
8.95	1.8644\\
8.96	1.8632\\
8.97	1.8621\\
8.98	1.8609\\
8.99	1.8597\\
9	1.8586\\
9.01	1.8574\\
9.02	1.8563\\
9.03	1.8552\\
9.04	1.8541\\
9.05	1.8531\\
9.06	1.852\\
9.07	1.851\\
9.08	1.85\\
9.09	1.8489\\
9.1	1.848\\
9.11	1.847\\
9.12	1.846\\
9.13	1.8451\\
9.14	1.8442\\
9.15	1.8433\\
9.16	1.8424\\
9.17	1.8415\\
9.18	1.8406\\
9.19	1.8398\\
9.2	1.839\\
9.21	1.8382\\
9.22	1.8374\\
9.23	1.8367\\
9.24	1.8359\\
9.25	1.8352\\
9.26	1.8345\\
9.27	1.8338\\
9.28	1.8331\\
9.29	1.8325\\
9.3	1.8319\\
9.31	1.8313\\
9.32	1.8307\\
9.33	1.8301\\
9.34	1.8296\\
9.35	1.8291\\
9.36	1.8286\\
9.37	1.8281\\
9.38	1.8277\\
9.39	1.8272\\
9.4	1.8268\\
9.41	1.8265\\
9.42	1.8261\\
9.43	1.8258\\
9.44	1.8254\\
9.45	1.8251\\
9.46	1.8249\\
9.47	1.8246\\
9.48	1.8244\\
9.49	1.8242\\
9.5	1.8241\\
9.51	1.8239\\
9.52	1.8238\\
9.53	1.8237\\
9.54	1.8236\\
9.55	1.8236\\
9.56	1.8236\\
9.57	1.8236\\
9.58	1.8236\\
9.59	1.8237\\
9.6	1.8238\\
9.61	1.8239\\
9.62	1.824\\
9.63	1.8242\\
9.64	1.8244\\
9.65	1.8246\\
9.66	1.8249\\
9.67	1.8251\\
9.68	1.8255\\
9.69	1.8258\\
9.7	1.8262\\
9.71	1.8266\\
9.72	1.827\\
9.73	1.8274\\
9.74	1.8279\\
9.75	1.8284\\
9.76	1.829\\
9.77	1.8296\\
9.78	1.8302\\
9.79	1.8308\\
9.8	1.8315\\
9.81	1.8322\\
9.82	1.8329\\
9.83	1.8337\\
9.84	1.8345\\
9.85	1.8353\\
9.86	1.8362\\
9.87	1.8371\\
9.88	1.838\\
9.89	1.839\\
9.9	1.84\\
9.91	1.841\\
9.92	1.8421\\
9.93	1.8432\\
9.94	1.8443\\
9.95	1.8455\\
9.96	1.8467\\
9.97	1.848\\
9.98	1.8492\\
9.99	1.8505\\
10	1.8519\\
10.01	1.8533\\
10.02	1.8547\\
10.03	1.8562\\
10.04	1.8577\\
10.05	1.8592\\
10.06	1.8608\\
10.07	1.8624\\
10.08	1.864\\
10.09	1.8657\\
10.1	1.8674\\
10.11	1.8692\\
10.12	1.871\\
10.13	1.8729\\
10.14	1.8747\\
10.15	1.8767\\
10.16	1.8786\\
10.17	1.8806\\
10.18	1.8827\\
10.19	1.8848\\
10.2	1.8869\\
10.21	1.889\\
10.22	1.8913\\
10.23	1.8935\\
10.24	1.8958\\
10.25	1.8981\\
10.26	1.9005\\
10.27	1.9029\\
10.28	1.9054\\
10.29	1.9079\\
10.3	1.9105\\
10.31	1.9131\\
10.32	1.9157\\
10.33	1.9184\\
10.34	1.9211\\
10.35	1.9239\\
10.36	1.9267\\
10.37	1.9296\\
10.38	1.9325\\
10.39	1.9355\\
10.4	1.9385\\
10.41	1.9415\\
10.42	1.9446\\
10.43	1.9478\\
10.44	1.951\\
10.45	1.9542\\
10.46	1.9575\\
10.47	1.9608\\
10.48	1.9642\\
10.49	1.9676\\
10.5	1.9711\\
10.51	1.9747\\
10.52	1.9782\\
10.53	1.9819\\
10.54	1.9856\\
10.55	1.9893\\
10.56	1.9931\\
10.57	1.9969\\
10.58	2.0008\\
10.59	2.0048\\
10.6	2.0087\\
10.61	2.0128\\
10.62	2.0169\\
10.63	2.021\\
10.64	2.0252\\
10.65	2.0295\\
10.66	2.0338\\
10.67	2.0382\\
10.68	2.0426\\
10.69	2.0471\\
10.7	2.0516\\
10.71	2.0562\\
10.72	2.0608\\
10.73	2.0655\\
10.74	2.0702\\
10.75	2.075\\
10.76	2.0799\\
10.77	2.0848\\
10.78	2.0898\\
10.79	2.0948\\
10.8	2.0999\\
10.81	2.1051\\
10.82	2.1103\\
10.83	2.1155\\
10.84	2.1209\\
10.85	2.1262\\
10.86	2.1317\\
10.87	2.1372\\
10.88	2.1428\\
10.89	2.1484\\
10.9	2.1541\\
10.91	2.1598\\
10.92	2.1656\\
10.93	2.1715\\
10.94	2.1774\\
10.95	2.1834\\
10.96	2.1894\\
10.97	2.1956\\
10.98	2.2017\\
10.99	2.208\\
11	2.2143\\
};
\end{axis}

\node at (2.7,2.2){\textcolor{gray}{$\mathcal T_{1}$}};
\node at (3.4,2.5){\textcolor{gray}{$\mathcal T_{2}$}};
\node at (4.2,2.8){\textcolor{gray}{$\mathcal T_{3}$}};
\node at (9.5,3.5){$x_\text{r}$};
\node at (1.6,2.2){$x_\text{r}(0)$};
\end{tikzpicture}%

%% file: pretty_ref_illu_const.tex
%
%
\definecolor{mycolor1}{rgb}{0.85000,0.32500,0.09800}%
\begin{tikzpicture}

\begin{axis}[%
width=2.8in,
height=1in,
at={(0.818851in,0.523448in)},
scale only axis,
xmin=-0.5,
xmax=10,
xlabel={time},
xmajorgrids,
ymin=0,
ymax=1.4,
ylabel={reference},
ymajorgrids,
legend style={legend cell align=left,align=left,draw=white!15!black}
]

\addplot[area legend,solid,fill=white!75!gray,opacity=5.000000e-01,draw=none,forget plot]
table[row sep=crcr] {%
x	y\\
-1	-1\\
5	-1\\
5	2\\
-1	2\\
}--cycle;

\addplot [color=mycolor1,solid,forget plot,very thick]
  table[row sep=crcr]{%
0	0\\
0.01	0.04849\\
0.02	0.09622\\
0.03	0.14313\\
0.04	0.18916\\
0.05	0.23426\\
0.06	0.27836\\
0.07	0.32141\\
0.08	0.36335\\
0.09	0.40413\\
0.1	0.44369\\
0.11	0.482\\
0.12	0.51899\\
0.13	0.55462\\
0.14	0.58886\\
0.15	0.62165\\
0.16	0.65298\\
0.17	0.6828\\
0.18	0.71108\\
0.19	0.73781\\
0.2	0.76295\\
0.21	0.7865\\
0.22	0.80844\\
0.23	0.82877\\
0.24	0.84748\\
0.25	0.86457\\
0.26	0.88006\\
0.27	0.89394\\
0.28	0.90625\\
0.29	0.917\\
0.3	0.92622\\
0.31	0.93394\\
0.32	0.9402\\
0.33	0.94503\\
0.34	0.94849\\
0.35	0.95063\\
0.36	0.95151\\
0.37	0.95117\\
0.38	0.9497\\
0.39	0.94715\\
0.4	0.9436\\
0.41	0.93912\\
0.42	0.93379\\
0.43	0.9277\\
0.44	0.92092\\
0.45	0.91354\\
0.46	0.90566\\
0.47	0.89735\\
0.48	0.8887\\
0.49	0.87981\\
0.5	0.87077\\
0.51	0.86166\\
0.52	0.85257\\
0.53	0.84359\\
0.54	0.8348\\
0.55	0.82627\\
0.56	0.8181\\
0.57	0.81035\\
0.58	0.80309\\
0.59	0.7964\\
0.6	0.79032\\
0.61	0.78492\\
0.62	0.78025\\
0.63	0.77636\\
0.64	0.77327\\
0.65	0.77104\\
0.66	0.76968\\
0.67	0.76922\\
0.68	0.76966\\
0.69	0.77103\\
0.7	0.77332\\
0.71	0.77652\\
0.72	0.78063\\
0.73	0.78563\\
0.74	0.7915\\
0.75	0.79821\\
0.76	0.80574\\
0.77	0.81403\\
0.78	0.82306\\
0.79	0.83277\\
0.8	0.84313\\
0.81	0.85407\\
0.82	0.86554\\
0.83	0.87749\\
0.84	0.88986\\
0.85	0.90258\\
0.86	0.91561\\
0.87	0.92887\\
0.88	0.94232\\
0.89	0.95588\\
0.9	0.9695\\
0.91	0.98312\\
0.92	0.99668\\
0.93	1.0101\\
0.94	1.0234\\
0.95	1.0365\\
0.96	1.0493\\
0.97	1.0619\\
0.98	1.074\\
0.99	1.0858\\
1	1.0972\\
1.01	1.1081\\
1.02	1.1186\\
1.03	1.1285\\
1.04	1.138\\
1.05	1.1469\\
1.06	1.1552\\
1.07	1.163\\
1.08	1.1702\\
1.09	1.1768\\
1.1	1.1828\\
1.11	1.1883\\
1.12	1.1932\\
1.13	1.1975\\
1.14	1.2012\\
1.15	1.2044\\
1.16	1.2071\\
1.17	1.2092\\
1.18	1.2108\\
1.19	1.2119\\
1.2	1.2126\\
1.21	1.2128\\
1.22	1.2126\\
1.23	1.2119\\
1.24	1.2109\\
1.25	1.2095\\
1.26	1.2078\\
1.27	1.2057\\
1.28	1.2034\\
1.29	1.2008\\
1.3	1.198\\
1.31	1.1949\\
1.32	1.1917\\
1.33	1.1883\\
1.34	1.1847\\
1.35	1.181\\
1.36	1.1772\\
1.37	1.1733\\
1.38	1.1694\\
1.39	1.1654\\
1.4	1.1614\\
1.41	1.1574\\
1.42	1.1534\\
1.43	1.1494\\
1.44	1.1455\\
1.45	1.1416\\
1.46	1.1378\\
1.47	1.1341\\
1.48	1.1305\\
1.49	1.127\\
1.5	1.1236\\
1.51	1.1203\\
1.52	1.1171\\
1.53	1.1141\\
1.54	1.1112\\
1.55	1.1085\\
1.56	1.1059\\
1.57	1.1034\\
1.58	1.1011\\
1.59	1.099\\
1.6	1.097\\
1.61	1.0952\\
1.62	1.0935\\
1.63	1.092\\
1.64	1.0907\\
1.65	1.0894\\
1.66	1.0884\\
1.67	1.0874\\
1.68	1.0866\\
1.69	1.086\\
1.7	1.0854\\
1.71	1.085\\
1.72	1.0847\\
1.73	1.0845\\
1.74	1.0845\\
1.75	1.0845\\
1.76	1.0846\\
1.77	1.0848\\
1.78	1.0851\\
1.79	1.0854\\
1.8	1.0858\\
1.81	1.0863\\
1.82	1.0868\\
1.83	1.0874\\
1.84	1.088\\
1.85	1.0887\\
1.86	1.0894\\
1.87	1.0901\\
1.88	1.0908\\
1.89	1.0915\\
1.9	1.0923\\
1.91	1.093\\
1.92	1.0937\\
1.93	1.0944\\
1.94	1.0951\\
1.95	1.0958\\
1.96	1.0964\\
1.97	1.0971\\
1.98	1.0976\\
1.99	1.0982\\
2	1.0987\\
2.01	1.0992\\
2.02	1.0996\\
2.03	1.1\\
2.04	1.1003\\
2.05	1.1006\\
2.06	1.1008\\
2.07	1.1009\\
2.08	1.101\\
2.09	1.1011\\
2.1	1.1011\\
2.11	1.101\\
2.12	1.1009\\
2.13	1.1007\\
2.14	1.1005\\
2.15	1.1002\\
2.16	1.0999\\
2.17	1.0994\\
2.18	1.099\\
2.19	1.0985\\
2.2	1.0979\\
2.21	1.0973\\
2.22	1.0967\\
2.23	1.096\\
2.24	1.0952\\
2.25	1.0944\\
2.26	1.0936\\
2.27	1.0928\\
2.28	1.0919\\
2.29	1.091\\
2.3	1.09\\
2.31	1.089\\
2.32	1.088\\
2.33	1.087\\
2.34	1.0859\\
2.35	1.0849\\
2.36	1.0838\\
2.37	1.0827\\
2.38	1.0816\\
2.39	1.0805\\
2.4	1.0794\\
2.41	1.0783\\
2.42	1.0772\\
2.43	1.0761\\
2.44	1.075\\
2.45	1.0739\\
2.46	1.0728\\
2.47	1.0717\\
2.48	1.0707\\
2.49	1.0696\\
2.5	1.0686\\
2.51	1.0676\\
2.52	1.0666\\
2.53	1.0656\\
2.54	1.0647\\
2.55	1.0637\\
2.56	1.0628\\
2.57	1.062\\
2.58	1.0611\\
2.59	1.0603\\
2.6	1.0595\\
2.61	1.0587\\
2.62	1.0579\\
2.63	1.0572\\
2.64	1.0565\\
2.65	1.0558\\
2.66	1.0552\\
2.67	1.0546\\
2.68	1.054\\
2.69	1.0534\\
2.7	1.0529\\
2.71	1.0524\\
2.72	1.0519\\
2.73	1.0514\\
2.74	1.051\\
2.75	1.0506\\
2.76	1.0502\\
2.77	1.0498\\
2.78	1.0495\\
2.79	1.0492\\
2.8	1.0488\\
2.81	1.0486\\
2.82	1.0483\\
2.83	1.048\\
2.84	1.0478\\
2.85	1.0475\\
2.86	1.0473\\
2.87	1.0471\\
2.88	1.0469\\
2.89	1.0467\\
2.9	1.0465\\
2.91	1.0463\\
2.92	1.0462\\
2.93	1.046\\
2.94	1.0458\\
2.95	1.0457\\
2.96	1.0455\\
2.97	1.0453\\
2.98	1.0452\\
2.99	1.045\\
3	1.0448\\
3.01	1.0447\\
3.02	1.0445\\
3.03	1.0443\\
3.04	1.0441\\
3.05	1.0439\\
3.06	1.0437\\
3.07	1.0435\\
3.08	1.0433\\
3.09	1.0431\\
3.1	1.0429\\
3.11	1.0426\\
3.12	1.0424\\
3.13	1.0422\\
3.14	1.0419\\
3.15	1.0416\\
3.16	1.0413\\
3.17	1.0411\\
3.18	1.0408\\
3.19	1.0404\\
3.2	1.0401\\
3.21	1.0398\\
3.22	1.0395\\
3.23	1.0391\\
3.24	1.0388\\
3.25	1.0384\\
3.26	1.0381\\
3.27	1.0377\\
3.28	1.0373\\
3.29	1.0369\\
3.3	1.0365\\
3.31	1.0362\\
3.32	1.0358\\
3.33	1.0354\\
3.34	1.0349\\
3.35	1.0345\\
3.36	1.0341\\
3.37	1.0337\\
3.38	1.0333\\
3.39	1.0329\\
3.4	1.0325\\
3.41	1.032\\
3.42	1.0316\\
3.43	1.0312\\
3.44	1.0308\\
3.45	1.0304\\
3.46	1.03\\
3.47	1.0296\\
3.48	1.0292\\
3.49	1.0288\\
3.5	1.0284\\
3.51	1.028\\
3.52	1.0276\\
3.53	1.0272\\
3.54	1.0269\\
3.55	1.0265\\
3.56	1.0261\\
3.57	1.0258\\
3.58	1.0254\\
3.59	1.0251\\
3.6	1.0248\\
3.61	1.0245\\
3.62	1.0242\\
3.63	1.0239\\
3.64	1.0236\\
3.65	1.0233\\
3.66	1.023\\
3.67	1.0227\\
3.68	1.0225\\
3.69	1.0222\\
3.7	1.022\\
3.71	1.0218\\
3.72	1.0215\\
3.73	1.0213\\
3.74	1.0211\\
3.75	1.0209\\
3.76	1.0207\\
3.77	1.0205\\
3.78	1.0203\\
3.79	1.0202\\
3.8	1.02\\
3.81	1.0198\\
3.82	1.0197\\
3.83	1.0196\\
3.84	1.0194\\
3.85	1.0193\\
3.86	1.0191\\
3.87	1.019\\
3.88	1.0189\\
3.89	1.0188\\
3.9	1.0187\\
3.91	1.0186\\
3.92	1.0184\\
3.93	1.0183\\
3.94	1.0182\\
3.95	1.0181\\
3.96	1.018\\
3.97	1.0179\\
3.98	1.0178\\
3.99	1.0177\\
4	1.0176\\
4.01	1.0175\\
4.02	1.0175\\
4.03	1.0174\\
4.04	1.0173\\
4.05	1.0172\\
4.06	1.0171\\
4.07	1.017\\
4.08	1.0169\\
4.09	1.0168\\
4.1	1.0166\\
4.11	1.0165\\
4.12	1.0164\\
4.13	1.0163\\
4.14	1.0162\\
4.15	1.0161\\
4.16	1.016\\
4.17	1.0158\\
4.18	1.0157\\
4.19	1.0156\\
4.2	1.0155\\
4.21	1.0153\\
4.22	1.0152\\
4.23	1.0151\\
4.24	1.0149\\
4.25	1.0148\\
4.26	1.0146\\
4.27	1.0145\\
4.28	1.0143\\
4.29	1.0142\\
4.3	1.014\\
4.31	1.0139\\
4.32	1.0137\\
4.33	1.0136\\
4.34	1.0134\\
4.35	1.0132\\
4.36	1.0131\\
4.37	1.0129\\
4.38	1.0128\\
4.39	1.0126\\
4.4	1.0124\\
4.41	1.0123\\
4.42	1.0121\\
4.43	1.012\\
4.44	1.0118\\
4.45	1.0116\\
4.46	1.0115\\
4.47	1.0113\\
4.48	1.0112\\
4.49	1.011\\
4.5	1.0109\\
4.51	1.0107\\
4.52	1.0106\\
4.53	1.0104\\
4.54	1.0103\\
4.55	1.0101\\
4.56	1.01\\
4.57	1.0099\\
4.58	1.0097\\
4.59	1.0096\\
4.6	1.0095\\
4.61	1.0093\\
4.62	1.0092\\
4.63	1.0091\\
4.64	1.009\\
4.65	1.0089\\
4.66	1.0088\\
4.67	1.0086\\
4.68	1.0085\\
4.69	1.0084\\
4.7	1.0083\\
4.71	1.0083\\
4.72	1.0082\\
4.73	1.0081\\
4.74	1.008\\
4.75	1.0079\\
4.76	1.0078\\
4.77	1.0078\\
4.78	1.0077\\
4.79	1.0076\\
4.8	1.0076\\
4.81	1.0075\\
4.82	1.0074\\
4.83	1.0074\\
4.84	1.0073\\
4.85	1.0073\\
4.86	1.0072\\
4.87	1.0072\\
4.88	1.0071\\
4.89	1.0071\\
4.9	1.007\\
4.91	1.007\\
4.92	1.0069\\
4.93	1.0069\\
4.94	1.0069\\
4.95	1.0068\\
4.96	1.0068\\
4.97	1.0068\\
4.98	1.0067\\
4.99	1.0067\\
5	1.0066\\
5.01	1.0066\\
5.02	1.0066\\
5.03	1.0065\\
5.04	1.0065\\
5.05	1.0065\\
5.06	1.0064\\
5.07	1.0064\\
5.08	1.0064\\
5.09	1.0063\\
5.1	1.0063\\
5.11	1.0062\\
5.12	1.0062\\
5.13	1.0062\\
5.14	1.0061\\
5.15	1.0061\\
5.16	1.006\\
5.17	1.006\\
5.18	1.0059\\
5.19	1.0059\\
5.2	1.0058\\
5.21	1.0058\\
5.22	1.0057\\
5.23	1.0057\\
5.24	1.0056\\
5.25	1.0056\\
5.26	1.0055\\
5.27	1.0055\\
5.28	1.0054\\
5.29	1.0054\\
5.3	1.0053\\
5.31	1.0052\\
5.32	1.0052\\
5.33	1.0051\\
5.34	1.0051\\
5.35	1.005\\
5.36	1.0049\\
5.37	1.0049\\
5.38	1.0048\\
5.39	1.0047\\
5.4	1.0047\\
5.41	1.0046\\
5.42	1.0046\\
5.43	1.0045\\
5.44	1.0044\\
5.45	1.0044\\
5.46	1.0043\\
5.47	1.0042\\
5.48	1.0042\\
5.49	1.0041\\
5.5	1.0041\\
5.51	1.004\\
5.52	1.0039\\
5.53	1.0039\\
5.54	1.0038\\
5.55	1.0038\\
5.56	1.0037\\
5.57	1.0036\\
5.58	1.0036\\
5.59	1.0035\\
5.6	1.0035\\
5.61	1.0034\\
5.62	1.0034\\
5.63	1.0033\\
5.64	1.0033\\
5.65	1.0033\\
5.66	1.0032\\
5.67	1.0032\\
5.68	1.0031\\
5.69	1.0031\\
5.7	1.0031\\
5.71	1.003\\
5.72	1.003\\
5.73	1.003\\
5.74	1.0029\\
5.75	1.0029\\
5.76	1.0029\\
5.77	1.0028\\
5.78	1.0028\\
5.79	1.0028\\
5.8	1.0028\\
5.81	1.0027\\
5.82	1.0027\\
5.83	1.0027\\
5.84	1.0027\\
5.85	1.0027\\
5.86	1.0026\\
5.87	1.0026\\
5.88	1.0026\\
5.89	1.0026\\
5.9	1.0026\\
5.91	1.0026\\
5.92	1.0026\\
5.93	1.0025\\
5.94	1.0025\\
5.95	1.0025\\
5.96	1.0025\\
5.97	1.0025\\
5.98	1.0025\\
5.99	1.0025\\
6	1.0025\\
6.01	1.0025\\
6.02	1.0024\\
6.03	1.0024\\
6.04	1.0024\\
6.05	1.0024\\
6.06	1.0024\\
6.07	1.0024\\
6.08	1.0024\\
6.09	1.0024\\
6.1	1.0024\\
6.11	1.0023\\
6.12	1.0023\\
6.13	1.0023\\
6.14	1.0023\\
6.15	1.0023\\
6.16	1.0023\\
6.17	1.0023\\
6.18	1.0022\\
6.19	1.0022\\
6.2	1.0022\\
6.21	1.0022\\
6.22	1.0022\\
6.23	1.0022\\
6.24	1.0021\\
6.25	1.0021\\
6.26	1.0021\\
6.27	1.0021\\
6.28	1.0021\\
6.29	1.002\\
6.3	1.002\\
6.31	1.002\\
6.32	1.002\\
6.33	1.0019\\
6.34	1.0019\\
6.35	1.0019\\
6.36	1.0019\\
6.37	1.0018\\
6.38	1.0018\\
6.39	1.0018\\
6.4	1.0018\\
6.41	1.0017\\
6.42	1.0017\\
6.43	1.0017\\
6.44	1.0017\\
6.45	1.0016\\
6.46	1.0016\\
6.47	1.0016\\
6.48	1.0016\\
6.49	1.0015\\
6.5	1.0015\\
6.51	1.0015\\
6.52	1.0014\\
6.53	1.0014\\
6.54	1.0014\\
6.55	1.0014\\
6.56	1.0014\\
6.57	1.0013\\
6.58	1.0013\\
6.59	1.0013\\
6.6	1.0013\\
6.61	1.0012\\
6.62	1.0012\\
6.63	1.0012\\
6.64	1.0012\\
6.65	1.0012\\
6.66	1.0011\\
6.67	1.0011\\
6.68	1.0011\\
6.69	1.0011\\
6.7	1.0011\\
6.71	1.0011\\
6.72	1.0011\\
6.73	1.001\\
6.74	1.001\\
6.75	1.001\\
6.76	1.001\\
6.77	1.001\\
6.78	1.001\\
6.79	1.001\\
6.8	1.001\\
6.81	1.001\\
6.82	1.001\\
6.83	1.001\\
6.84	1.001\\
6.85	1.0009\\
6.86	1.0009\\
6.87	1.0009\\
6.88	1.0009\\
6.89	1.0009\\
6.9	1.0009\\
6.91	1.0009\\
6.92	1.0009\\
6.93	1.0009\\
6.94	1.0009\\
6.95	1.0009\\
6.96	1.0009\\
6.97	1.0009\\
6.98	1.0009\\
6.99	1.0009\\
7	1.0009\\
7.01	1.0009\\
7.02	1.0009\\
7.03	1.0009\\
7.04	1.0009\\
7.05	1.0009\\
7.06	1.0009\\
7.07	1.0009\\
7.08	1.0009\\
7.09	1.0009\\
7.1	1.0009\\
7.11	1.0009\\
7.12	1.0009\\
7.13	1.0009\\
7.14	1.0009\\
7.15	1.0009\\
7.16	1.0009\\
7.17	1.0009\\
7.18	1.0009\\
7.19	1.0009\\
7.2	1.0009\\
7.21	1.0008\\
7.22	1.0008\\
7.23	1.0008\\
7.24	1.0008\\
7.25	1.0008\\
7.26	1.0008\\
7.27	1.0008\\
7.28	1.0008\\
7.29	1.0008\\
7.3	1.0008\\
7.31	1.0008\\
7.32	1.0008\\
7.33	1.0007\\
7.34	1.0007\\
7.35	1.0007\\
7.36	1.0007\\
7.37	1.0007\\
7.38	1.0007\\
7.39	1.0007\\
7.4	1.0007\\
7.41	1.0007\\
7.42	1.0006\\
7.43	1.0006\\
7.44	1.0006\\
7.45	1.0006\\
7.46	1.0006\\
7.47	1.0006\\
7.48	1.0006\\
7.49	1.0006\\
7.5	1.0006\\
7.51	1.0005\\
7.52	1.0005\\
7.53	1.0005\\
7.54	1.0005\\
7.55	1.0005\\
7.56	1.0005\\
7.57	1.0005\\
7.58	1.0005\\
7.59	1.0005\\
7.6	1.0004\\
7.61	1.0004\\
7.62	1.0004\\
7.63	1.0004\\
7.64	1.0004\\
7.65	1.0004\\
7.66	1.0004\\
7.67	1.0004\\
7.68	1.0004\\
7.69	1.0004\\
7.7	1.0004\\
7.71	1.0004\\
7.72	1.0004\\
7.73	1.0004\\
7.74	1.0004\\
7.75	1.0003\\
7.76	1.0003\\
7.77	1.0003\\
7.78	1.0003\\
7.79	1.0003\\
7.8	1.0003\\
7.81	1.0003\\
7.82	1.0003\\
7.83	1.0003\\
7.84	1.0003\\
7.85	1.0003\\
7.86	1.0003\\
7.87	1.0003\\
7.88	1.0003\\
7.89	1.0003\\
7.9	1.0003\\
7.91	1.0003\\
7.92	1.0003\\
7.93	1.0003\\
7.94	1.0003\\
7.95	1.0003\\
7.96	1.0003\\
7.97	1.0003\\
7.98	1.0003\\
7.99	1.0003\\
8	1.0003\\
8.01	1.0003\\
8.02	1.0003\\
8.03	1.0003\\
8.04	1.0003\\
8.05	1.0003\\
8.06	1.0003\\
8.07	1.0003\\
8.08	1.0003\\
8.09	1.0003\\
8.1	1.0003\\
8.11	1.0003\\
8.12	1.0003\\
8.13	1.0003\\
8.14	1.0003\\
8.15	1.0003\\
8.16	1.0003\\
8.17	1.0003\\
8.18	1.0003\\
8.19	1.0003\\
8.2	1.0003\\
8.21	1.0003\\
8.22	1.0003\\
8.23	1.0003\\
8.24	1.0003\\
8.25	1.0003\\
8.26	1.0003\\
8.27	1.0003\\
8.28	1.0003\\
8.29	1.0003\\
8.3	1.0003\\
8.31	1.0003\\
8.32	1.0003\\
8.33	1.0003\\
8.34	1.0003\\
8.35	1.0003\\
8.36	1.0003\\
8.37	1.0003\\
8.38	1.0003\\
8.39	1.0003\\
8.4	1.0003\\
8.41	1.0003\\
8.42	1.0002\\
8.43	1.0002\\
8.44	1.0002\\
8.45	1.0002\\
8.46	1.0002\\
8.47	1.0002\\
8.48	1.0002\\
8.49	1.0002\\
8.5	1.0002\\
8.51	1.0002\\
8.52	1.0002\\
8.53	1.0002\\
8.54	1.0002\\
8.55	1.0002\\
8.56	1.0002\\
8.57	1.0002\\
8.58	1.0002\\
8.59	1.0002\\
8.6	1.0002\\
8.61	1.0001\\
8.62	1.0001\\
8.63	1.0001\\
8.64	1.0001\\
8.65	1.0001\\
8.66	1.0001\\
8.67	1.0001\\
8.68	1.0001\\
8.69	1.0001\\
8.7	1.0001\\
8.71	1.0001\\
8.72	1.0001\\
8.73	1.0001\\
8.74	1.0001\\
8.75	1.0001\\
8.76	1.0001\\
8.77	1.0001\\
8.78	1.0001\\
8.79	1.0001\\
8.8	1.0001\\
8.81	1.0001\\
8.82	1.0001\\
8.83	1.0001\\
8.84	1.0001\\
8.85	1.0001\\
8.86	1.0001\\
8.87	1.0001\\
8.88	1.0001\\
8.89	1.0001\\
8.9	1.0001\\
8.91	1.0001\\
8.92	1.0001\\
8.93	1.0001\\
8.94	1.0001\\
8.95	1.0001\\
8.96	1.0001\\
8.97	1.0001\\
8.98	1.0001\\
8.99	1.0001\\
9	1.0001\\
9.01	1.0001\\
9.02	1.0001\\
9.03	1.0001\\
9.04	1.0001\\
9.05	1.0001\\
9.06	1.0001\\
9.07	1.0001\\
9.08	1.0001\\
9.09	1.0001\\
9.1	1.0001\\
9.11	1.0001\\
9.12	1.0001\\
9.13	1.0001\\
9.14	1.0001\\
9.15	1.0001\\
9.16	1.0001\\
9.17	1.0001\\
9.18	1.0001\\
9.19	1.0001\\
9.2	1.0001\\
9.21	1.0001\\
9.22	1.0001\\
9.23	1.0001\\
9.24	1.0001\\
9.25	1.0001\\
9.26	1.0001\\
9.27	1.0001\\
9.28	1.0001\\
9.29	1.0001\\
9.3	1.0001\\
9.31	1.0001\\
9.32	1.0001\\
9.33	1.0001\\
9.34	1.0001\\
9.35	1.0001\\
9.36	1.0001\\
9.37	1.0001\\
9.38	1.0001\\
9.39	1.0001\\
9.4	1.0001\\
9.41	1.0001\\
9.42	1.0001\\
9.43	1.0001\\
9.44	1.0001\\
9.45	1.0001\\
9.46	1.0001\\
9.47	1.0001\\
9.48	1.0001\\
9.49	1.0001\\
9.5	1.0001\\
9.51	1.0001\\
9.52	1.0001\\
9.53	1.0001\\
9.54	1.0001\\
9.55	1.0001\\
9.56	1.0001\\
9.57	1.0001\\
9.58	1.0001\\
9.59	1\\
9.6	1\\
9.61	1\\
9.62	1\\
9.63	1\\
9.64	1\\
9.65	1\\
9.66	1\\
9.67	1\\
9.68	1\\
9.69	1\\
9.7	1\\
9.71	1\\
9.72	1\\
9.73	1\\
9.74	1\\
9.75	1\\
9.76	1\\
9.77	1\\
9.78	1\\
9.79	1\\
9.8	1\\
9.81	1\\
9.82	1\\
9.83	1\\
9.84	1\\
9.85	1\\
9.86	1\\
9.87	1\\
9.88	1\\
9.89	1\\
9.9	1\\
9.91	1\\
9.92	1\\
9.93	1\\
9.94	1\\
9.95	1\\
9.96	1\\
9.97	1\\
9.98	1\\
9.99	1\\
10	1\\
};
\end{axis}

\node at (4,2){transient phase};
\node at (7.5,2){asymptotic phase};
\node at (6,1.07){$T_\text{trans}$};

\draw [] plot coordinates{(5.8,1.3) (5.8,1.45) };
\end{tikzpicture}%

%% file: pretty_convergence_as_const.tex
%
%
\definecolor{mycolor1}{rgb}{0.00000,0.44700,0.74100}%
\definecolor{mycolor2}{rgb}{0.85000,0.32500,0.09800}%
\definecolor{mycolor3}{rgb}{0.92900,0.69400,0.12500}%
\definecolor{mycolor4}{rgb}{0.49400,0.18400,0.55600}%
\begin{tikzpicture}

\begin{axis}[%
width=2.8in,
height=0.8in,
at={(0.8in,1.4in)},
scale only axis,
unbounded coords=jump,
xmin=70,
xmax=130,
xmajorgrids,
ymin=0,
ymax=12,
ylabel={$\bar m$},
xticklabels={},
ymajorgrids,
legend style={legend cell align=left,align=left,draw=white!15!black}
]
\addplot [color=mycolor1,solid,forget plot,very thick]
  table[row sep=crcr]{%
70	nan\\
71	nan\\
72	nan\\
73	nan\\
74	nan\\
75	nan\\
76	nan\\
77	nan\\
78	nan\\
79	nan\\
80	nan\\
81	11.134\\
82	10.764\\
83	10.339\\
84	9.8732\\
85	9.3807\\
86	8.8712\\
87	8.3532\\
88	7.8332\\
89	7.3172\\
90	6.8093\\
91	6.3131\\
92	5.8313\\
93	5.3658\\
94	4.9183\\
95	4.4897\\
96	4.0805\\
97	3.691\\
98	3.3208\\
99	2.9699\\
100	2.6635\\
101	2.3848\\
102	2.1316\\
103	1.9022\\
104	1.6945\\
105	1.5069\\
106	1.338\\
107	1.1858\\
108	1.0492\\
109	0.92681\\
110	0.8173\\
111	0.71943\\
112	0.63225\\
113	0.55463\\
114	0.48572\\
115	0.42465\\
116	0.37061\\
117	0.32286\\
118	0.2808\\
119	0.24378\\
120	0.21128\\
121	0.18279\\
122	0.15788\\
123	0.13609\\
124	0.11712\\
125	0.10062\\
126	0.086284\\
127	0.073861\\
128	0.063115\\
129	0.053837\\
130	0.045841\\
};
\end{axis}

\begin{axis}[%
width=2.8in,
height=0.8in,
at={(0.8in,0.5in)},
scale only axis,
scale only axis,
unbounded coords=jump,
xmin=70,
xmax=130,
xlabel={$\bar k$},
xmajorgrids,
ymin=-100,
ymax=100,
ylabel={$\dot m^+$},
ymajorgrids,
legend style={at={(0.02,0.17)},anchor=south west,legend cell align=left,align=left,draw=white!15!black}
]
\addplot [color=mycolor1,dashed,very thick]
  table[row sep=crcr]{%
70	nan\\
71	nan\\
72	nan\\
73	nan\\
74	nan\\
75	nan\\
76	nan\\
77	nan\\
78	nan\\
79	nan\\
80	nan\\
81	85.798\\
82	82.102\\
83	77.849\\
84	73.196\\
85	68.273\\
86	63.18\\
87	58.002\\
88	52.805\\
89	47.647\\
90	42.571\\
91	37.611\\
92	32.795\\
93	28.143\\
94	23.67\\
95	19.387\\
96	15.298\\
97	11.406\\
98	7.7065\\
99	4.2003\\
100	1.1405\\
101	-1.6406\\
102	-4.1632\\
103	-6.4468\\
104	-8.5107\\
105	-10.372\\
106	-12.046\\
107	-13.551\\
108	-14.899\\
109	-16.105\\
110	-17.181\\
111	-18.14\\
112	-18.993\\
113	-19.749\\
114	-20.419\\
115	-21.011\\
116	-21.533\\
117	-21.993\\
118	-22.397\\
119	-22.75\\
120	-23.06\\
121	-23.33\\
122	-23.565\\
123	-23.77\\
124	-23.947\\
125	-24.1\\
126	-24.233\\
127	-24.348\\
128	-24.446\\
129	-24.53\\
130	-24.603\\
};
\addlegendentry{$\underline \tau$};

\addplot [color=mycolor2,dashdotted,very thick]
  table[row sep=crcr]{%
70	nan\\
71	nan\\
72	nan\\
73	nan\\
74	nan\\
75	nan\\
76	nan\\
77	nan\\
78	nan\\
79	nan\\
80	nan\\
81	-86.884\\
82	-83.184\\
83	-78.926\\
84	-74.269\\
85	-69.342\\
86	-64.244\\
87	-59.062\\
88	-53.859\\
89	-48.697\\
90	-43.616\\
91	-38.651\\
92	-33.83\\
93	-29.173\\
94	-24.695\\
95	-20.406\\
96	-16.312\\
97	-12.414\\
98	-8.7094\\
99	-5.1976\\
100	-2.1297\\
101	0.66431\\
102	3.2045\\
103	5.5098\\
104	7.599\\
105	9.4894\\
106	11.194\\
107	12.733\\
108	14.116\\
109	15.359\\
110	16.473\\
111	17.471\\
112	18.362\\
113	19.158\\
114	19.866\\
115	20.496\\
116	21.054\\
117	21.55\\
118	21.987\\
119	22.374\\
120	22.715\\
121	23.014\\
122	23.278\\
123	23.509\\
124	23.711\\
125	23.887\\
126	24.041\\
127	24.175\\
128	24.292\\
129	24.393\\
130	24.48\\
};
\addlegendentry{$\overline \tau$};

\addplot [color=mycolor3,solid,very thick]
  table[row sep=crcr]{%
70	nan\\
71	nan\\
72	nan\\
73	nan\\
74	nan\\
75	nan\\
76	nan\\
77	nan\\
78	nan\\
79	nan\\
80	nan\\
81	92.193\\
82	90.378\\
83	88.042\\
84	85.29\\
85	82.213\\
86	78.884\\
87	75.367\\
88	71.715\\
89	67.979\\
90	64.195\\
91	60.397\\
92	56.615\\
93	52.872\\
94	49.188\\
95	45.579\\
96	42.056\\
97	38.628\\
98	35.3\\
99	32.077\\
100	29.214\\
101	26.557\\
102	24.097\\
103	21.824\\
104	19.728\\
105	17.799\\
106	16.031\\
107	14.41\\
108	12.929\\
109	11.579\\
110	10.35\\
111	9.2346\\
112	8.2242\\
113	7.3101\\
114	6.4858\\
115	5.7439\\
116	5.0772\\
117	4.4793\\
118	3.9446\\
119	3.4672\\
120	3.0419\\
121	2.6638\\
122	2.3286\\
123	2.0312\\
124	1.7687\\
125	1.5373\\
126	1.3336\\
127	1.1547\\
128	0.99798\\
129	0.86088\\
130	0.74122\\
};

\addplot [color=mycolor3,solid,very thick]
  table[row sep=crcr]{%
70	nan\\
71	nan\\
72	nan\\
73	nan\\
74	nan\\
75	nan\\
76	nan\\
77	nan\\
78	nan\\
79	nan\\
80	nan\\
81	-92.193\\
82	-90.378\\
83	-88.042\\
84	-85.29\\
85	-82.213\\
86	-78.884\\
87	-75.367\\
88	-71.715\\
89	-67.979\\
90	-64.195\\
91	-60.397\\
92	-56.615\\
93	-52.872\\
94	-49.188\\
95	-45.579\\
96	-42.056\\
97	-38.628\\
98	-35.3\\
99	-32.077\\
100	-29.214\\
101	-26.557\\
102	-24.097\\
103	-21.824\\
104	-19.728\\
105	-17.799\\
106	-16.031\\
107	-14.41\\
108	-12.929\\
109	-11.579\\
110	-10.35\\
111	-9.2346\\
112	-8.2242\\
113	-7.3101\\
114	-6.4858\\
115	-5.7439\\
116	-5.0772\\
117	-4.4793\\
118	-3.9446\\
119	-3.4672\\
120	-3.0419\\
121	-2.6638\\
122	-2.3286\\
123	-2.0312\\
124	-1.7687\\
125	-1.5373\\
126	-1.3336\\
127	-1.1547\\
128	-0.99798\\
129	-0.86088\\
130	-0.74122\\
};
\addlegendentry{$\pm \bar{\dot{m}}$};

\end{axis}

\begin{axis}[%
width=2.8in,
height=0.8in,
at={(0.8in,2.3in)},
scale only axis,
xmin=70,
xmax=130,
xticklabels={},
xmajorgrids,
ymin=0.15,
ymax=0.8,
ylabel={$\hat \theta$},
ymajorgrids,
legend style={at={(0.78,0.2)},anchor=south west,legend cell align=left,align=left,draw=white!15!black}
]
\addplot [color=mycolor1,densely dotted,very thick]
  table[row sep=crcr]{%
70	0.65931\\
71	0.65931\\
72	0.65931\\
73	0.65931\\
74	0.65931\\
75	0.65931\\
76	0.65934\\
77	0.65932\\
78	0.66248\\
79	0.66704\\
80	0.67132\\
81	0.67528\\
82	0.67892\\
83	0.68228\\
84	0.68535\\
85	0.68816\\
86	0.69072\\
87	0.69305\\
88	0.69514\\
89	0.69701\\
90	0.69865\\
91	0.70007\\
92	0.70125\\
93	0.70218\\
94	0.70286\\
95	0.70325\\
96	0.70333\\
97	0.70305\\
98	0.70236\\
99	0.7012\\
100	0.70119\\
101	0.70118\\
102	0.70118\\
103	0.70118\\
104	0.70118\\
105	0.70117\\
106	0.70118\\
107	0.70117\\
108	0.70117\\
109	0.70117\\
110	0.70118\\
111	0.70117\\
112	0.70117\\
113	0.70117\\
114	0.70117\\
115	0.70117\\
116	0.70117\\
117	0.70117\\
118	0.70117\\
119	0.70117\\
120	0.70117\\
121	0.70117\\
122	0.70118\\
123	0.70117\\
124	0.70117\\
125	0.70117\\
126	0.70117\\
127	0.70117\\
128	0.70117\\
129	0.70117\\
130	0.70117\\
};
\addlegendentry{$\hat \theta_1$};

\addplot [color=mycolor2,solid,very thick]
  table[row sep=crcr]{%
70	0.20347\\
71	0.20347\\
72	0.20347\\
73	0.20347\\
74	0.20347\\
75	0.20347\\
76	0.20348\\
77	0.20347\\
78	0.2048\\
79	0.20657\\
80	0.20814\\
81	0.20953\\
82	0.21078\\
83	0.2119\\
84	0.21292\\
85	0.21385\\
86	0.21469\\
87	0.21546\\
88	0.21616\\
89	0.21679\\
90	0.21737\\
91	0.21788\\
92	0.21834\\
93	0.21873\\
94	0.21907\\
95	0.21934\\
96	0.21954\\
97	0.21967\\
98	0.21972\\
99	0.21967\\
100	0.21967\\
101	0.21967\\
102	0.21967\\
103	0.21967\\
104	0.21967\\
105	0.21967\\
106	0.21967\\
107	0.21967\\
108	0.21967\\
109	0.21967\\
110	0.21967\\
111	0.21967\\
112	0.21967\\
113	0.21967\\
114	0.21967\\
115	0.21967\\
116	0.21967\\
117	0.21967\\
118	0.21967\\
119	0.21967\\
120	0.21967\\
121	0.21967\\
122	0.21967\\
123	0.21967\\
124	0.21967\\
125	0.21967\\
126	0.21967\\
127	0.21967\\
128	0.21967\\
129	0.21967\\
130	0.21967\\
};
\addlegendentry{$\hat \theta_2$};

\end{axis}

\begin{axis}[%
width=0.8in,
height=0.5in,
at={(2.7in,1.6in)},
scale only axis,
xmin=119,
xmax=130,
yticklabel style = {font=\small},
xticklabel style = {font=\small},
xmajorgrids,
ymin=0,
ymax=0.24378,
ymajorgrids,
legend style={legend cell align=left,align=left,draw=white!15!black},
axis background/.style={fill=white},
]
\addplot [color=mycolor1,solid,forget plot,very thick]
  table[row sep=crcr]{%
119	0.24378\\
120	0.21128\\
121	0.18279\\
122	0.15788\\
123	0.13609\\
124	0.11712\\
125	0.10062\\
126	0.086284\\
127	0.073861\\
128	0.063115\\
129	0.053837\\
130	0.045841\\
};

\addplot[area legend,solid,fill=lightgray,opacity=5.000000e-01,draw=none,forget plot]
table[row sep=crcr] {%
x	y\\
119	0\\
120	0\\
121	0\\
122	0\\
123	0\\
124	0\\
125	0\\
126	0\\
127	0\\
128	0\\
129	0\\
130	0\\
130	0.05\\
129	0.05\\
128	0.05\\
127	0.05\\
126	0.05\\
125	0.05\\
124	0.05\\
123	0.05\\
122	0.05\\
121	0.05\\
120	0.05\\
119	0.05\\
}--cycle;

\end{axis}

\end{tikzpicture}%

%% file: pretty_prediction_as_const.tex
%
%
\definecolor{mycolor1}{rgb}{0.85000,0.32500,0.09800}%
\definecolor{mycolor2}{rgb}{0.00000,0.44700,0.74100}%
\begin{tikzpicture}

\begin{axis}[%
width=2.8in,
height=1.5in,
at={(0.836782in,0.531034in)},
scale only axis,
xmin=0,
xmax=200,
xlabel={$k$},
xmajorgrids,
ymin=-1.3,
ymax=0.35,
ylabel={$x_\text{r}$},
ymajorgrids,
title style={font=\bfseries},
legend style={legend cell align=left,align=left,draw=white!15!black,at={(0.63,0.02)},anchor=south east, legend columns=2}
]

\addplot[area legend,solid,fill=lightgray,opacity=5.000000e-01,draw=none]
table[row sep=crcr] {%
x	y\\
0	-10\\
1	-10\\
2	-10\\
3	-10\\
4	-10\\
5	-10\\
6	-10\\
7	-10\\
8	-10\\
9	-10\\
10	-10\\
11	-10\\
12	-10\\
13	-10\\
14	-10\\
15	-10\\
16	-10\\
17	-10\\
18	-10\\
19	-10\\
20	-10\\
21	-10\\
22	-10\\
23	-10\\
24	-10\\
25	-10\\
26	-10\\
27	-10\\
28	-10\\
29	-10\\
30	-10\\
31	-10\\
32	-10\\
33	-10\\
34	-10\\
35	-10\\
36	-10\\
37	-10\\
38	-10\\
39	-10\\
40	-10\\
41	-10\\
42	-10\\
43	-10\\
44	-10\\
45	-10\\
46	-10\\
47	-10\\
48	-10\\
49	-10\\
50	-10\\
51	-10\\
52	-10\\
53	-10\\
54	-10\\
55	-10\\
56	-10\\
57	-10\\
58	-10\\
59	-10\\
60	-10\\
61	-10\\
62	-10\\
63	-10\\
64	-10\\
65	-10\\
66	-10\\
67	-10\\
68	-10\\
69	-10\\
70	-10\\
71	-10\\
72	-10\\
73	-10\\
74	-10\\
75	-10\\
76	-10\\
77	-10\\
78	-10\\
79	-10\\
80	-10\\
81	-10\\
82	-10\\
83	-10\\
84	-10\\
85	-10\\
86	-10\\
87	-10\\
88	-10\\
89	-10\\
90	-10\\
91	-10\\
92	-10\\
93	-10\\
94	-10\\
95	-10\\
96	-10\\
97	-10\\
98	-10\\
99	-10\\
100	-10\\
101	-10\\
102	-10\\
103	-10\\
104	-10\\
105	-10\\
106	-10\\
107	-10\\
108	-10\\
109	-10\\
110	-10\\
111	-10\\
112	-10\\
113	-10\\
114	-10\\
115	-10\\
116	-10\\
117	-10\\
118	-10\\
119	-10\\
120	-10\\
121	-10\\
122	-10\\
123	-10\\
124	-10\\
125	-10\\
126	-10\\
127	-10\\
128	-10\\
129	-10\\
130	-10\\
131	-10\\
132	-10\\
133	-10\\
134	-10\\
135	-10\\
136	-10\\
137	-10\\
138	-10\\
139	-10\\
140	-10\\
141	-10\\
142	-10\\
143	-10\\
144	-10\\
145	-10\\
146	-10\\
147	-10\\
148	-10\\
149	-10\\
150	-10\\
151	-10\\
152	-10\\
153	-10\\
154	-10\\
155	-10\\
156	-10\\
157	-10\\
158	-10\\
159	-10\\
160	-10\\
161	-10\\
162	-10\\
163	-10\\
164	-10\\
165	-10\\
166	-10\\
167	-10\\
168	-10\\
169	-10\\
170	-10\\
171	-10\\
172	-10\\
173	-10\\
174	-10\\
175	-10\\
176	-10\\
177	-10\\
178	-10\\
179	-10\\
180	-10\\
181	-10\\
182	-10\\
183	-10\\
184	-10\\
185	-10\\
186	-10\\
187	-10\\
188	-10\\
189	-10\\
190	-10\\
191	-10\\
192	-10\\
193	-10\\
194	-10\\
195	-10\\
196	-10\\
197	-10\\
198	-10\\
199	-10\\
200	-10\\
200	0.05\\
199	0.05\\
198	0.05\\
197	0.05\\
196	0.05\\
195	0.05\\
194	0.05\\
193	0.05\\
192	0.05\\
191	0.05\\
190	0.05\\
189	0.05\\
188	0.05\\
187	0.05\\
186	0.05\\
185	0.05\\
184	0.05\\
183	0.05\\
182	0.05\\
181	0.05\\
180	0.05\\
179	0.05\\
178	0.05\\
177	0.05\\
176	0.05\\
175	0.05\\
174	0.05\\
173	0.05\\
172	0.05\\
171	0.05\\
170	0.05\\
169	0.05\\
168	0.05\\
167	0.05\\
166	0.05\\
165	0.05\\
164	0.05\\
163	0.05\\
162	0.05\\
161	0.05\\
160	0.05\\
159	0.05\\
158	0.05\\
157	0.05\\
156	0.05\\
155	0.05\\
154	0.05\\
153	0.05\\
152	0.05\\
151	0.05\\
150	0.05\\
149	0.05\\
148	0.05\\
147	0.05\\
146	0.05\\
145	0.05\\
144	0.05\\
143	0.05\\
142	0.05\\
141	0.05\\
140	0.05\\
139	0.05\\
138	0.05\\
137	0.05\\
136	0.05\\
135	0.05\\
134	0.05\\
133	0.05\\
132	0.05\\
131	0.05\\
130	0.05\\
129	0.05\\
128	0.05\\
127	0.05\\
126	0.05\\
125	0.05\\
124	0.05\\
123	0.05\\
122	0.05\\
121	0.05\\
120	0.05\\
119	0.05\\
118	0.05\\
117	0.05\\
116	0.05\\
115	0.05\\
114	0.05\\
113	0.05\\
112	0.05\\
111	0.05\\
110	0.05\\
109	0.05\\
108	0.05\\
107	0.05\\
106	0.05\\
105	0.05\\
104	0.05\\
103	0.05\\
102	0.05\\
101	0.05\\
100	0.05\\
99	0.05\\
98	0.05\\
97	0.05\\
96	0.05\\
95	0.05\\
94	0.05\\
93	0.05\\
92	0.05\\
91	0.05\\
90	0.05\\
89	0.05\\
88	0.05\\
87	0.05\\
86	0.05\\
85	0.05\\
84	0.05\\
83	0.05\\
82	0.05\\
81	0.05\\
80	0.05\\
79	0.05\\
78	0.05\\
77	0.05\\
76	0.05\\
75	0.05\\
74	0.05\\
73	0.05\\
72	0.05\\
71	0.05\\
70	0.05\\
69	0.05\\
68	0.05\\
67	0.05\\
66	0.05\\
65	0.05\\
64	0.05\\
63	0.05\\
62	0.05\\
61	0.05\\
60	0.05\\
59	0.05\\
58	0.05\\
57	0.05\\
56	0.05\\
55	0.05\\
54	0.05\\
53	0.05\\
52	0.05\\
51	0.05\\
50	0.05\\
49	0.05\\
48	0.05\\
47	0.05\\
46	0.05\\
45	0.05\\
44	0.05\\
43	0.05\\
42	0.05\\
41	0.05\\
40	0.05\\
39	0.05\\
38	0.05\\
37	0.05\\
36	0.05\\
35	0.05\\
34	0.05\\
33	0.05\\
32	0.05\\
31	0.05\\
30	0.05\\
29	0.05\\
28	0.05\\
27	0.05\\
26	0.05\\
25	0.05\\
24	0.05\\
23	0.05\\
22	0.05\\
21	0.05\\
20	0.05\\
19	0.05\\
18	0.05\\
17	0.05\\
16	0.05\\
15	0.05\\
14	0.05\\
13	0.05\\
12	0.05\\
11	0.05\\
10	0.05\\
9	0.05\\
8	0.05\\
7	0.05\\
6	0.05\\
5	0.05\\
4	0.05\\
3	0.05\\
2	0.05\\
1	0.05\\
0	0.05\\
}--cycle;

\addlegendentry{$\mathcal X$};

\addplot[area legend,solid,fill=black!25!lightgray,opacity=5.000000e-01,draw=none]
table[row sep=crcr] {%
x	y\\
0	-1.1434\\
1	-1.0902\\
2	-1.036\\
3	-0.98111\\
4	-0.92606\\
5	-0.87128\\
6	-0.8172\\
7	-0.76422\\
8	-0.71274\\
9	-0.66311\\
10	-0.61563\\
11	-0.5706\\
12	-0.52825\\
13	-0.48878\\
14	-0.45231\\
15	-0.41897\\
16	-0.38879\\
17	-0.36178\\
18	-0.33791\\
19	-0.3171\\
20	-0.29923\\
21	-0.28416\\
22	-0.27171\\
23	-0.26168\\
24	-0.25386\\
25	-0.248\\
26	-0.24387\\
27	-0.24124\\
28	-0.23984\\
29	-0.23946\\
30	-0.23987\\
31	-0.24087\\
32	-0.24225\\
33	-0.24386\\
34	-0.24554\\
35	-0.24717\\
36	-0.24866\\
37	-0.24992\\
38	-0.2509\\
39	-0.25158\\
40	-0.25193\\
41	-0.25198\\
42	-0.25173\\
43	-0.25123\\
44	-0.25052\\
45	-0.24966\\
46	-0.24871\\
47	-0.24771\\
48	-0.24674\\
49	-0.24585\\
50	-0.24507\\
51	-0.24447\\
52	-0.24407\\
53	-0.24389\\
54	-0.24395\\
55	-0.24425\\
56	-0.24479\\
57	-0.24555\\
58	-0.2465\\
59	-0.24762\\
60	-0.24887\\
61	-0.2502\\
62	-0.25156\\
63	-0.25291\\
64	-0.25419\\
65	-0.25536\\
66	-0.25637\\
67	-0.25717\\
68	-0.25774\\
69	-0.25804\\
70	-0.25804\\
71	-0.25772\\
72	-0.25708\\
73	-0.25612\\
74	-0.25483\\
75	-0.25323\\
76	-0.25134\\
77	-0.24918\\
78	-0.24678\\
79	-0.24418\\
80	-0.24141\\
81	-0.23851\\
82	-0.23553\\
83	-0.23251\\
84	-0.2295\\
85	-0.22652\\
86	-0.22363\\
87	-0.22085\\
88	-0.21823\\
89	-0.21579\\
90	-0.21357\\
91	-0.21157\\
92	-0.20983\\
93	-0.20834\\
94	-0.20713\\
95	-0.20619\\
96	-0.20553\\
97	-0.20513\\
98	-0.205\\
99	-0.20512\\
100	-0.20548\\
101	-0.20607\\
102	-0.20686\\
103	-0.20783\\
104	-0.20898\\
105	-0.21026\\
106	-0.21168\\
107	-0.21319\\
108	-0.21479\\
109	-0.21646\\
110	-0.21817\\
111	-0.2199\\
112	-0.22164\\
113	-0.22339\\
114	-0.22511\\
115	-0.2268\\
116	-0.22845\\
117	-0.23005\\
118	-0.23159\\
119	-0.23307\\
120	-0.23448\\
121	-0.23581\\
122	-0.23707\\
123	-0.23826\\
124	-0.23937\\
125	-0.2404\\
126	-0.24136\\
127	-0.24225\\
128	-0.24306\\
129	-0.24381\\
130	-0.24449\\
131	-0.24511\\
132	-0.24567\\
133	-0.24618\\
134	-0.24664\\
135	-0.24705\\
136	-0.24742\\
137	-0.24775\\
138	-0.24804\\
139	-0.2483\\
140	-0.24853\\
141	-0.24873\\
142	-0.2489\\
143	-0.24906\\
144	-0.24919\\
145	-0.24931\\
146	-0.24941\\
147	-0.2495\\
148	-0.24957\\
149	-0.24964\\
150	-0.2497\\
151	-0.24974\\
152	-0.24979\\
153	-0.24982\\
154	-0.24985\\
155	-0.24987\\
156	-0.2499\\
157	-0.24991\\
158	-0.24993\\
159	-0.24994\\
160	-0.24995\\
161	-0.24996\\
162	-0.24997\\
163	-0.24997\\
164	-0.24998\\
165	-0.24998\\
166	-0.24999\\
167	-0.24999\\
168	-0.24999\\
169	-0.24999\\
170	-0.24999\\
171	-0.25\\
172	-0.25\\
173	-0.25\\
174	-0.25\\
175	-0.25\\
176	-0.25\\
177	-0.25\\
178	-0.25\\
179	-0.25\\
180	-0.25\\
181	-0.25\\
182	-0.25\\
183	-0.25\\
184	-0.25\\
185	-0.25\\
186	-0.25\\
187	-0.25\\
188	-0.25\\
189	-0.25\\
190	-0.25\\
191	-0.25\\
192	-0.25\\
193	-0.25\\
194	-0.25\\
195	-0.25\\
196	-0.25\\
197	-0.25\\
198	-0.25\\
199	-0.25\\
200	-0.25\\
200	0.25\\
199	0.25\\
198	0.25\\
197	0.25\\
196	0.25\\
195	0.25\\
194	0.25\\
193	0.25\\
192	0.25\\
191	0.25\\
190	0.25\\
189	0.25\\
188	0.25\\
187	0.25\\
186	0.25\\
185	0.25\\
184	0.25\\
183	0.25\\
182	0.25\\
181	0.25\\
180	0.25\\
179	0.25\\
178	0.25\\
177	0.25\\
176	0.25\\
175	0.25\\
174	0.25\\
173	0.25\\
172	0.25\\
171	0.25\\
170	0.25001\\
169	0.25001\\
168	0.25001\\
167	0.25001\\
166	0.25001\\
165	0.25002\\
164	0.25002\\
163	0.25003\\
162	0.25003\\
161	0.25004\\
160	0.25005\\
159	0.25006\\
158	0.25007\\
157	0.25009\\
156	0.2501\\
155	0.25013\\
154	0.25015\\
153	0.25018\\
152	0.25021\\
151	0.25026\\
150	0.2503\\
149	0.25036\\
148	0.25043\\
147	0.2505\\
146	0.25059\\
145	0.25069\\
144	0.25081\\
143	0.25094\\
142	0.2511\\
141	0.25127\\
140	0.25147\\
139	0.2517\\
138	0.25196\\
137	0.25225\\
136	0.25258\\
135	0.25295\\
134	0.25336\\
133	0.25382\\
132	0.25433\\
131	0.25489\\
130	0.25551\\
129	0.25619\\
128	0.25694\\
127	0.25775\\
126	0.25864\\
125	0.2596\\
124	0.26063\\
123	0.26174\\
122	0.26293\\
121	0.26419\\
120	0.26552\\
119	0.26693\\
118	0.26841\\
117	0.26995\\
116	0.27155\\
115	0.2732\\
114	0.27489\\
113	0.27661\\
112	0.27836\\
111	0.2801\\
110	0.28183\\
109	0.28354\\
108	0.28521\\
107	0.28681\\
106	0.28832\\
105	0.28974\\
104	0.29102\\
103	0.29217\\
102	0.29314\\
101	0.29393\\
100	0.29452\\
99	0.29488\\
98	0.295\\
97	0.29487\\
96	0.29447\\
95	0.29381\\
94	0.29287\\
93	0.29166\\
92	0.29017\\
91	0.28843\\
90	0.28643\\
89	0.28421\\
88	0.28177\\
87	0.27915\\
86	0.27637\\
85	0.27348\\
84	0.2705\\
83	0.26749\\
82	0.26447\\
81	0.26149\\
80	0.25859\\
79	0.25582\\
78	0.25322\\
77	0.25082\\
76	0.24866\\
75	0.24677\\
74	0.24517\\
73	0.24388\\
72	0.24292\\
71	0.24228\\
70	0.24196\\
69	0.24196\\
68	0.24226\\
67	0.24283\\
66	0.24363\\
65	0.24464\\
64	0.24581\\
63	0.24709\\
62	0.24844\\
61	0.2498\\
60	0.25113\\
59	0.25238\\
58	0.2535\\
57	0.25445\\
56	0.25521\\
55	0.25575\\
54	0.25605\\
53	0.25611\\
52	0.25593\\
51	0.25553\\
50	0.25493\\
49	0.25415\\
48	0.25326\\
47	0.25229\\
46	0.25129\\
45	0.25034\\
44	0.24948\\
43	0.24877\\
42	0.24827\\
41	0.24802\\
40	0.24807\\
39	0.24842\\
38	0.2491\\
37	0.25008\\
36	0.25134\\
35	0.25283\\
34	0.25446\\
33	0.25614\\
32	0.25775\\
31	0.25913\\
30	0.26013\\
29	0.26054\\
28	0.26016\\
27	0.25876\\
26	0.25613\\
25	0.252\\
24	0.24614\\
23	0.23832\\
22	0.22829\\
21	0.21584\\
20	0.20077\\
19	0.1829\\
18	0.16209\\
17	0.13822\\
16	0.11121\\
15	0.081033\\
14	0.047686\\
13	0.011224\\
12	-0.028255\\
11	-0.070604\\
10	-0.11563\\
9	-0.16311\\
8	-0.21274\\
7	-0.26422\\
6	-0.3172\\
5	-0.37128\\
4	-0.42606\\
3	-0.48111\\
2	-0.53598\\
1	-0.59021\\
0	-0.64337\\
}--cycle;

\addlegendentry{$\mathcal T_k$};

\addplot [color=black,solid]
  table[row sep=crcr]{%
0	-1\\
1	-0.93721\\
2	-0.87467\\
3	-0.81262\\
4	-0.75131\\
5	-0.69098\\
6	-0.63188\\
7	-0.57422\\
8	-0.51825\\
9	-0.46417\\
10	-0.41221\\
11	-0.36258\\
12	-0.31545\\
13	-0.27103\\
14	-0.22949\\
15	-0.19098\\
16	-0.15567\\
17	-0.12369\\
18	-0.095173\\
19	-0.070224\\
20	-0.048943\\
21	-0.031417\\
22	-0.017713\\
23	-0.0078853\\
24	-0.0019733\\
25	0\\
26	0\\
27	0\\
28	0\\
29	0\\
30	0\\
31	0\\
32	0\\
33	0\\
34	0\\
35	0\\
36	0\\
37	0\\
38	0\\
39	0\\
40	0\\
41	0\\
42	0\\
43	0\\
44	0\\
45	0\\
46	0\\
47	0\\
48	0\\
49	0\\
50	0\\
51	0\\
52	0\\
53	0\\
54	0\\
55	0\\
56	0\\
57	0\\
58	0\\
59	0\\
60	0\\
61	0\\
62	0\\
63	0\\
64	0\\
65	0\\
66	0\\
67	0\\
68	0\\
69	0\\
70	0\\
71	0\\
72	0\\
73	0\\
74	0\\
75	0\\
76	0\\
77	0\\
78	0\\
79	0\\
80	0\\
81	0\\
82	0\\
83	0\\
84	0\\
85	0\\
86	0\\
87	0\\
88	0\\
89	0\\
90	0\\
91	0\\
92	0\\
93	0\\
94	0\\
95	0\\
96	0\\
97	0\\
98	0\\
99	0\\
100	0\\
101	0\\
102	0\\
103	0\\
104	0\\
105	0\\
106	0\\
107	0\\
108	0\\
109	0\\
110	0\\
111	0\\
112	0\\
113	0\\
114	0\\
115	0\\
116	0\\
117	0\\
118	0\\
119	0\\
120	0\\
121	0\\
122	0\\
123	0\\
124	0\\
125	0\\
126	0\\
127	0\\
128	0\\
129	0\\
130	0\\
131	0\\
132	0\\
133	0\\
134	0\\
135	0\\
136	0\\
137	0\\
138	0\\
139	0\\
140	0\\
141	0\\
142	0\\
143	0\\
144	0\\
145	0\\
146	0\\
147	0\\
148	0\\
149	0\\
150	0\\
151	0\\
152	0\\
153	0\\
154	0\\
155	0\\
156	0\\
157	0\\
158	0\\
159	0\\
160	0\\
161	0\\
162	0\\
163	0\\
164	0\\
165	0\\
166	0\\
167	0\\
168	0\\
169	0\\
170	0\\
171	0\\
172	0\\
173	0\\
174	0\\
175	0\\
176	0\\
177	0\\
178	0\\
179	0\\
180	0\\
181	0\\
182	0\\
183	0\\
184	0\\
185	0\\
186	0\\
187	0\\
188	0\\
189	0\\
190	0\\
191	0\\
192	0\\
193	0\\
194	0\\
195	0\\
196	0\\
197	0\\
198	0\\
199	0\\
200	0\\
};
\addlegendentry{$r(k)$};

\addplot [color=white!10!black,mark size=5.0pt,only marks,mark=x,mark options={solid}]
  table[row sep=crcr]{%
0	-0.99306\\
10	-0.40789\\
20	-0.048848\\
30	0.0032086\\
40	0.0041431\\
50	0.0026027\\
60	0.0018705\\
};
\addlegendentry{ data};

\addplot [color=mycolor1,densely dashdotted ,line width=1.5pt]
  table[row sep=crcr]{%
0	-0.99249\\
1	-0.9354\\
2	-0.87651\\
3	-0.81638\\
4	-0.75559\\
5	-0.69469\\
6	-0.63425\\
7	-0.57479\\
8	-0.51681\\
9	-0.46078\\
10	-0.40711\\
11	-0.35618\\
12	-0.30829\\
13	-0.26371\\
14	-0.22263\\
15	-0.18517\\
16	-0.15141\\
17	-0.12135\\
18	-0.094948\\
19	-0.0721\\
20	-0.052656\\
21	-0.036424\\
22	-0.023174\\
23	-0.012652\\
24	-0.0045782\\
25	0.0013379\\
26	0.0053939\\
27	0.0078862\\
28	0.0091036\\
29	0.009321\\
30	0.0087944\\
31	0.0077563\\
32	0.0064125\\
33	0.0049392\\
34	0.0034824\\
35	0.0021564\\
36	0.0010455\\
37	0.00020451\\
38	-0.00033855\\
39	-0.00057973\\
40	-0.0005357\\
41	-0.00024001\\
42	0.00026098\\
43	0.00091216\\
44	0.0016538\\
45	0.0024253\\
46	0.0031688\\
47	0.0038323\\
48	0.0043724\\
49	0.0047561\\
50	0.0049625\\
51	0.0049835\\
52	0.0048242\\
53	0.0045024\\
54	0.0040479\\
55	0.0035009\\
56	0.0029106\\
57	0.0023329\\
58	0.0018281\\
59	0.0014588\\
60	0.0012873\\
61	0.0013732\\
62	0.0017713\\
63	0.0025296\\
64	0.0036876\\
65	0.0052747\\
66	0.0073094\\
67	0.0097987\\
68	0.012738\\
69	0.016109\\
70	0.019885\\
71	0.024028\\
72	0.028488\\
73	0.03321\\
74	0.038131\\
75	0.043182\\
76	0.048294\\
77	0.053393\\
78	0.058406\\
79	0.063264\\
80	0.067898\\
81	0.072247\\
82	0.076253\\
83	0.079867\\
84	0.083047\\
85	0.085758\\
86	0.087976\\
87	0.089684\\
88	0.090873\\
89	0.091543\\
90	0.091703\\
91	0.091366\\
92	0.090555\\
93	0.089296\\
94	0.087622\\
95	0.085567\\
96	0.083172\\
97	0.080477\\
98	0.077524\\
99	0.074355\\
100	0.071014\\
101	0.067541\\
102	0.063977\\
103	0.060359\\
104	0.056721\\
105	0.053098\\
106	0.049517\\
107	0.046006\\
108	0.042586\\
109	0.039278\\
110	0.036097\\
111	0.033056\\
112	0.030166\\
113	0.027435\\
114	0.024865\\
115	0.022461\\
116	0.020221\\
117	0.018145\\
118	0.016229\\
119	0.014468\\
120	0.012857\\
121	0.011389\\
122	0.010057\\
123	0.0088533\\
124	0.0077694\\
125	0.0067972\\
126	0.0059286\\
127	0.0051553\\
128	0.0044694\\
129	0.0038631\\
130	0.0033292\\
131	0.0028606\\
132	0.0024508\\
133	0.0020936\\
134	0.0017832\\
135	0.0015145\\
136	0.0012825\\
137	0.001083\\
138	0.00091197\\
139	0.00076575\\
140	0.00064116\\
141	0.00053533\\
142	0.00044573\\
143	0.00037009\\
144	0.00030644\\
145	0.00025303\\
146	0.00020836\\
147	0.00017111\\
148	0.00014013\\
149	0.00011445\\
150	9.3219e-05\\
151	7.5723e-05\\
152	6.1345e-05\\
153	4.9564e-05\\
154	3.9938e-05\\
155	3.2096e-05\\
156	2.5725e-05\\
157	2.0563e-05\\
158	1.6394e-05\\
159	1.3036e-05\\
160	1.0338e-05\\
161	8.1769e-06\\
162	6.4506e-06\\
163	5.0755e-06\\
164	3.983e-06\\
165	3.1176e-06\\
166	2.4339e-06\\
167	1.8951e-06\\
168	1.4718e-06\\
169	1.1401e-06\\
170	8.8089e-07\\
171	6.7884e-07\\
172	5.2179e-07\\
173	4.0005e-07\\
174	3.0592e-07\\
175	2.3334e-07\\
176	1.7753e-07\\
177	1.3472e-07\\
178	1.0197e-07\\
179	7.6989e-08\\
180	5.7979e-08\\
181	4.3552e-08\\
182	3.2632e-08\\
183	2.4388e-08\\
184	1.8181e-08\\
185	1.3519e-08\\
186	1.0027e-08\\
187	7.4186e-09\\
188	5.4748e-09\\
189	4.0302e-09\\
190	2.9593e-09\\
191	2.1675e-09\\
192	1.5836e-09\\
193	1.1541e-09\\
194	8.3894e-10\\
195	6.0834e-10\\
196	4.4002e-10\\
197	3.1748e-10\\
198	2.285e-10\\
199	1.6404e-10\\
200	1.1748e-10\\
};
\addlegendentry{$m^+_\text{uncon}$};

\addplot [color=mycolor2,dashed,line width=1.5pt]
  table[row sep=crcr]{%
0	-0.99263\\
1	-0.93357\\
2	-0.87331\\
3	-0.81234\\
4	-0.75118\\
5	-0.69031\\
6	-0.63022\\
7	-0.57136\\
8	-0.51416\\
9	-0.45901\\
10	-0.40626\\
11	-0.35623\\
12	-0.30917\\
13	-0.26531\\
14	-0.22479\\
15	-0.18774\\
16	-0.15421\\
17	-0.1242\\
18	-0.097676\\
19	-0.074553\\
20	-0.054701\\
21	-0.037957\\
22	-0.024125\\
23	-0.012982\\
24	-0.0042863\\
25	0.002221\\
26	0.0068061\\
27	0.0097388\\
28	0.011286\\
29	0.011708\\
30	0.011252\\
31	0.010149\\
32	0.0086104\\
33	0.0068242\\
34	0.0049545\\
35	0.003139\\
36	0.0014887\\
37	8.8216e-05\\
38	-0.0010037\\
39	-0.0017528\\
40	-0.002147\\
41	-0.0021946\\
42	-0.0019211\\
43	-0.0013665\\
44	-0.00058194\\
45	0.00037348\\
46	0.0014358\\
47	0.0025393\\
48	0.0036196\\
49	0.0046158\\
50	0.0054732\\
51	0.0061451\\
52	0.006594\\
53	0.0067932\\
54	0.0067269\\
55	0.0063909\\
56	0.0057918\\
57	0.0049474\\
58	0.0038847\\
59	0.0026396\\
60	0.0012551\\
61	-0.00022008\\
62	-0.0017334\\
63	-0.00323\\
64	-0.0046546\\
65	-0.0059531\\
66	-0.0070742\\
67	-0.0079707\\
68	-0.0086011\\
69	-0.0089302\\
70	-0.0089303\\
71	-0.0085814\\
72	-0.0078721\\
73	-0.0067991\\
74	-0.0053675\\
75	-0.0035904\\
76	-0.0014882\\
77	0.00091199\\
78	0.0035771\\
79	0.006469\\
80	0.0095458\\
81	0.012763\\
82	0.016073\\
83	0.019428\\
84	0.022783\\
85	0.026089\\
86	0.029304\\
87	0.032387\\
88	0.035298\\
89	0.038006\\
90	0.040481\\
91	0.042698\\
92	0.044638\\
93	0.046287\\
94	0.047635\\
95	0.048678\\
96	0.049416\\
97	0.049854\\
98	0.049999\\
99	0.049864\\
100	0.049463\\
101	0.048814\\
102	0.047937\\
103	0.046852\\
104	0.045583\\
105	0.044151\\
106	0.042581\\
107	0.040896\\
108	0.039118\\
109	0.03727\\
110	0.035372\\
111	0.033445\\
112	0.031506\\
113	0.029572\\
114	0.027659\\
115	0.02578\\
116	0.023947\\
117	0.022169\\
118	0.020456\\
119	0.018814\\
120	0.017249\\
121	0.015764\\
122	0.014362\\
123	0.013045\\
124	0.011812\\
125	0.010664\\
126	0.0095996\\
127	0.0086157\\
128	0.0077104\\
129	0.0068803\\
130	0.0061222\\
131	0.0054324\\
132	0.0048068\\
133	0.0042415\\
134	0.0037325\\
135	0.0032756\\
136	0.0028669\\
137	0.0025025\\
138	0.0021786\\
139	0.0018916\\
140	0.0016381\\
141	0.0014149\\
142	0.001219\\
143	0.0010475\\
144	0.0008978\\
145	0.00076755\\
146	0.00065454\\
147	0.00055677\\
148	0.00047241\\
149	0.00039984\\
150	0.00033758\\
151	0.00028432\\
152	0.00023887\\
153	0.00020019\\
154	0.00016737\\
155	0.00013959\\
156	0.00011614\\
157	9.6399e-05\\
158	7.9821e-05\\
159	6.5937e-05\\
160	5.4338e-05\\
161	4.4673e-05\\
162	3.6641e-05\\
163	2.9982e-05\\
164	2.4476e-05\\
165	1.9934e-05\\
166	1.6197e-05\\
167	1.313e-05\\
168	1.0619e-05\\
169	8.5689e-06\\
170	6.8983e-06\\
171	5.5407e-06\\
172	4.44e-06\\
173	3.5499e-06\\
174	2.8317e-06\\
175	2.2537e-06\\
176	1.7896e-06\\
177	1.4178e-06\\
178	1.1207e-06\\
179	8.8391e-07\\
180	6.9556e-07\\
181	5.4611e-07\\
182	4.2781e-07\\
183	3.3439e-07\\
184	2.6078e-07\\
185	2.0292e-07\\
186	1.5755e-07\\
187	1.2205e-07\\
188	9.4339e-08\\
189	7.2758e-08\\
190	5.5989e-08\\
191	4.299e-08\\
192	3.2937e-08\\
193	2.5178e-08\\
194	1.9205e-08\\
195	1.4617e-08\\
196	1.11e-08\\
197	8.4113e-09\\
198	6.3598e-09\\
199	4.7981e-09\\
200	3.612e-09\\
};
\addlegendentry{$m^+_\text{constr}$};

\addplot [color=gray,solid,forget plot]
  table[row sep=crcr]{%
130	-2\\
130	2\\
};
\end{axis}
\end{tikzpicture}%

%% file: pretty_convergence_periodic.tex
%
%

\definecolor{mycolor1}{rgb}{0.85000,0.32500,0.09800}%
\definecolor{mycolor2}{rgb}{0.00000,0.44700,0.74100}%
\definecolor{mycolor3}{rgb}{0.92900,0.69400,0.12500}%
\definecolor{mycolor4}{rgb}{1.00000,1.00000,0.00000}%
\definecolor{mycolor5}{rgb}{0.92900,0.69400,0.12500}%

\begin{tikzpicture}

\begin{axis}[%
width=2.8in,
height=1.2in,
at={(0.818851in,0.523448in)},
scale only axis,
xmin=0,
xmax=4,
xlabel={$k$},
xtick={0,1,2,3,4},
xticklabels={0,100,200,300,400},
xmajorgrids,
ymin=-7,
ymax=10,
ylabel={$\dot{m}$},
ymajorgrids,
axis x line*=bottom,
axis y line*=left,
legend style={at={(0.99,0.04)},anchor=south east,legend cell align=left,align=left,draw=white!15!black,legend columns=3}
]

\addplot[area legend,solid,fill=lightgray,opacity=5.000000e-01,draw=none]
table[row sep=crcr] {%
x	y\\
0	-34.765\\
0.25	-34.765\\
0.25	4.1752\\
0	4.1752\\
}--cycle;

\addlegendentry{$\tau$};

%

\addplot[area legend,solid,fill=lightgray,opacity=5.000000e-01,draw=none,forget plot]
table[row sep=crcr] {%
x	y\\
0.25	-39.114\\
0.5	-39.114\\
0.5	4\\
0.25	4\\
}--cycle;

%
%
%

\addplot[area legend,solid,fill=lightgray,opacity=5.000000e-01,draw=none,forget plot]
table[row sep=crcr] {%
x	y\\
0.5	-36.711\\
0.75	-36.711\\
0.75	4.8861\\
0.5	4.8861\\
}--cycle;


\addplot[area legend,solid,fill=lightgray,opacity=5.000000e-01,draw=none,forget plot]
table[row sep=crcr] {%
x	y\\
0.75	-34.742\\
1	-34.742\\
1	7.2888\\
0.75	7.2888\\
}--cycle;


\addplot[area legend,solid,fill=lightgray,opacity=5.000000e-01,draw=none,forget plot]
table[row sep=crcr] {%
x	y\\
1	-34.458\\
1.25	-34.458\\
1.25	9.258\\
1	9.258\\
}--cycle;


\addplot[area legend,solid,fill=lightgray,opacity=5.000000e-01,draw=none,forget plot]
table[row sep=crcr] {%
x	y\\
1.25	-34.681\\
1.5	-34.681\\
1.5	8.8161\\
1.25	8.8161\\
}--cycle;


\addplot[area legend,solid,fill=lightgray,opacity=5.000000e-01,draw=none,forget plot]
table[row sep=crcr] {%
x	y\\
1.5	-31.683\\
1.75	-31.683\\
1.75	8.9205\\
1.5	8.9205\\
}--cycle;


\addplot[area legend,solid,fill=lightgray,opacity=5.000000e-01,draw=none,forget plot]
table[row sep=crcr] {%
x	y\\
1.75	-24.934\\
2	-24.934\\
2	12.317\\
1.75	12.317\\
}--cycle;


\addplot[area legend,solid,fill=lightgray,opacity=5.000000e-01,draw=none,forget plot]
table[row sep=crcr] {%
x	y\\
2	-18.801\\
2.25	-18.801\\
2.25	19.066\\
2	19.066\\
}--cycle;


\addplot[area legend,solid,fill=lightgray,opacity=5.000000e-01,draw=none,forget plot]
table[row sep=crcr] {%
x	y\\
2.25	-15.71\\
2.5	-15.71\\
2.5	25.199\\
2.25	25.199\\
}--cycle;


\addplot[area legend,solid,fill=lightgray,opacity=5.000000e-01,draw=none,forget plot]
table[row sep=crcr] {%
x	y\\
2.5	-15.698\\
2.75	-15.698\\
2.75	24.698\\
2.5	24.698\\
}--cycle;


\addplot[area legend,solid,fill=lightgray,opacity=5.000000e-01,draw=none,forget plot]
table[row sep=crcr] {%
x	y\\
2.75	-19.302\\
3	-19.302\\
3	14.359\\
2.75	14.359\\
}--cycle;


\addplot[area legend,solid,fill=lightgray,opacity=5.000000e-01,draw=none,forget plot]
table[row sep=crcr] {%
x	y\\
3	-29.641\\
3.25	-29.641\\
3.25	5.7387\\
3	5.7387\\
}--cycle;


\addplot[area legend,solid,fill=lightgray,opacity=5.000000e-01,draw=none,forget plot]
table[row sep=crcr] {%
x	y\\
3.25	-38.261\\
3.5	-38.261\\
3.5	4\\
3.25	4\\
}--cycle;


\addplot[area legend,solid,fill=lightgray,opacity=5.000000e-01,draw=none,forget plot]
table[row sep=crcr] {%
x	y\\
3.5	-37.974\\
3.75	-37.974\\
3.75	4.1138\\
3.5	4.1138\\
}--cycle;


\addplot[area legend,solid,fill=lightgray,opacity=5.000000e-01,draw=none,forget plot,very thick]
table[row sep=crcr] {%
x	y\\
3.75	-35.552\\
4	-35.552\\
4	6.026\\
3.75	6.026\\
}--cycle;

\addplot [color=mycolor2,dashed,very thick]
  table[row sep=crcr]{%
0.01	3.496\\
0.02	3.3837\\
0.03	3.2678\\
0.04	3.1489\\
0.05	3.0273\\
0.06	2.9034\\
0.07	2.7776\\
0.08	2.6504\\
0.09	2.5221\\
0.1	2.3931\\
0.11	2.2638\\
0.12	2.1346\\
0.13	2.0058\\
0.14	1.8779\\
0.15	1.7511\\
0.16	1.6258\\
0.17	1.5023\\
0.18	1.3809\\
0.19	1.2618\\
0.2	1.1452\\
0.21	1.0315\\
0.22	0.92077\\
0.23	0.81323\\
0.24	0.70902\\
0.25	0.60828\\
0.26	0.51112\\
0.27	0.41763\\
0.28	0.32785\\
0.29	0.24185\\
0.3	0.15962\\
0.31	0.081182\\
0.32	0.0065045\\
0.33	-0.064449\\
0.34	-0.13173\\
0.35	-0.19542\\
0.36	-0.25559\\
0.37	-0.31235\\
0.38	-0.36579\\
0.39	-0.41605\\
0.4	-0.46324\\
0.41	-0.5075\\
0.42	-0.54895\\
0.43	-0.58775\\
0.44	-0.62402\\
0.45	-0.6579\\
0.46	-0.68954\\
0.47	-0.71907\\
0.48	-0.74662\\
0.49	-0.7723\\
0.5	-0.79625\\
0.51	-0.81858\\
0.52	-0.83938\\
0.53	-0.85875\\
0.54	-0.87677\\
0.55	-0.89353\\
0.56	-0.90909\\
0.57	-0.92351\\
0.58	-0.93683\\
0.59	-0.9491\\
0.6	-0.96034\\
0.61	-0.97058\\
0.62	-0.97982\\
0.63	-0.98807\\
0.64	-0.99534\\
0.65	-1.0016\\
0.66	-1.0069\\
0.67	-1.0111\\
0.68	-1.0143\\
0.69	-1.0163\\
0.7	-1.0173\\
0.71	-1.0172\\
0.72	-1.0159\\
0.73	-1.0133\\
0.74	-1.0096\\
0.75	-1.0046\\
0.76	-0.99836\\
0.77	-0.9908\\
0.78	-0.98191\\
0.79	-0.9717\\
0.8	-0.96013\\
0.81	-0.94721\\
0.82	-0.93294\\
0.83	-0.91731\\
0.84	-0.90032\\
0.85	-0.882\\
0.86	-0.86236\\
0.87	-0.84142\\
0.88	-0.8192\\
0.89	-0.79573\\
0.9	-0.77105\\
0.91	-0.74519\\
0.92	-0.71821\\
0.93	-0.69013\\
0.94	-0.66103\\
0.95	-0.63095\\
0.96	-0.59994\\
0.97	-0.56807\\
0.98	-0.53541\\
0.99	-0.50202\\
1	-0.46798\\
1.01	-0.43335\\
1.02	-0.39822\\
1.03	-0.36267\\
1.04	-0.32679\\
1.05	-0.29065\\
1.06	-0.25436\\
1.07	-0.21802\\
1.08	-0.18171\\
1.09	-0.14555\\
1.1	-0.10965\\
1.11	-0.074115\\
1.12	-0.039064\\
1.13	-0.0046195\\
1.14	0.029093\\
1.15	0.061944\\
1.16	0.0938\\
1.17	0.12452\\
1.18	0.15397\\
1.19	0.182\\
1.2	0.20847\\
1.21	0.23322\\
1.22	0.25611\\
1.23	0.27698\\
1.24	0.29569\\
1.25	0.31209\\
1.26	0.32601\\
1.27	0.33733\\
1.28	0.3459\\
1.29	0.35158\\
1.3	0.35423\\
1.31	0.35374\\
1.32	0.34999\\
1.33	0.34287\\
1.34	0.33229\\
1.35	0.31816\\
1.36	0.30041\\
1.37	0.27897\\
1.38	0.25382\\
1.39	0.22491\\
1.4	0.19223\\
1.41	0.1558\\
1.42	0.11563\\
1.43	0.071759\\
1.44	0.024247\\
1.45	-0.026831\\
1.46	-0.08138\\
1.47	-0.13929\\
1.48	-0.20043\\
1.49	-0.26465\\
1.5	-0.33179\\
1.51	-0.40167\\
1.52	-0.47409\\
1.53	-0.54884\\
1.54	-0.62571\\
1.55	-0.70445\\
1.56	-0.78483\\
1.57	-0.86658\\
1.58	-0.94945\\
1.59	-1.0332\\
1.6	-1.1175\\
1.61	-1.2021\\
1.62	-1.2867\\
1.63	-1.3711\\
1.64	-1.455\\
1.65	-1.5381\\
1.66	-1.6201\\
1.67	-1.7008\\
1.68	-1.78\\
1.69	-1.8574\\
1.7	-1.9328\\
1.71	-2.006\\
1.72	-2.0767\\
1.73	-2.1449\\
1.74	-2.2103\\
1.75	-2.2728\\
1.76	-2.3322\\
1.77	-2.3884\\
1.78	-2.4414\\
1.79	-2.491\\
1.8	-2.5372\\
1.81	-2.58\\
1.82	-2.6192\\
1.83	-2.655\\
1.84	-2.6872\\
1.85	-2.716\\
1.86	-2.7414\\
1.87	-2.7633\\
1.88	-2.782\\
1.89	-2.7973\\
1.9	-2.8096\\
1.91	-2.8187\\
1.92	-2.8249\\
1.93	-2.8282\\
1.94	-2.8288\\
1.95	-2.8268\\
1.96	-2.8223\\
1.97	-2.8155\\
1.98	-2.8064\\
1.99	-2.7952\\
2	-2.782\\
2.01	-2.767\\
2.02	-2.7503\\
2.03	-2.7319\\
2.04	-2.712\\
2.05	-2.6907\\
2.06	-2.6681\\
2.07	-2.6443\\
2.08	-2.6194\\
2.09	-2.5933\\
2.1	-2.5662\\
2.11	-2.5382\\
2.12	-2.5092\\
2.13	-2.4792\\
2.14	-2.4484\\
2.15	-2.4167\\
2.16	-2.384\\
2.17	-2.3505\\
2.18	-2.3159\\
2.19	-2.2804\\
2.2	-2.2438\\
2.21	-2.2061\\
2.22	-2.1673\\
2.23	-2.1271\\
2.24	-2.0857\\
2.25	-2.0428\\
2.26	-1.9984\\
2.27	-1.9524\\
2.28	-1.9047\\
2.29	-1.8551\\
2.3	-1.8036\\
2.31	-1.75\\
2.32	-1.6943\\
2.33	-1.6364\\
2.34	-1.576\\
2.35	-1.5132\\
2.36	-1.4478\\
2.37	-1.3797\\
2.38	-1.3088\\
2.39	-1.2351\\
2.4	-1.1584\\
2.41	-1.0788\\
2.42	-0.99604\\
2.43	-0.91021\\
2.44	-0.82124\\
2.45	-0.7291\\
2.46	-0.63377\\
2.47	-0.53525\\
2.48	-0.43356\\
2.49	-0.32872\\
2.5	-0.22078\\
2.51	-0.1098\\
2.52	0.0041584\\
2.53	0.121\\
2.54	0.24062\\
2.55	0.3629\\
2.56	0.48771\\
2.57	0.61488\\
2.58	0.74425\\
2.59	0.87563\\
2.6	1.0088\\
2.61	1.1436\\
2.62	1.2798\\
2.63	1.417\\
2.64	1.5552\\
2.65	1.6939\\
2.66	1.8329\\
2.67	1.9719\\
2.68	2.1107\\
2.69	2.2488\\
2.7	2.3859\\
2.71	2.5218\\
2.72	2.6561\\
2.73	2.7884\\
2.74	2.9185\\
2.75	3.0458\\
2.76	3.1702\\
2.77	3.2913\\
2.78	3.4086\\
2.79	3.5219\\
2.8	3.6309\\
2.81	3.7352\\
2.82	3.8345\\
2.83	3.9284\\
2.84	4.0168\\
2.85	4.0993\\
2.86	4.1756\\
2.87	4.2456\\
2.88	4.3089\\
2.89	4.3653\\
2.9	4.4148\\
2.91	4.457\\
2.92	4.492\\
2.93	4.5194\\
2.94	4.5393\\
2.95	4.5516\\
2.96	4.5562\\
2.97	4.5532\\
2.98	4.5425\\
2.99	4.5241\\
3	4.4982\\
3.01	4.4649\\
3.02	4.4242\\
3.03	4.3763\\
3.04	4.3213\\
3.05	4.2596\\
3.06	4.1913\\
3.07	4.1166\\
3.08	4.0359\\
3.09	3.9493\\
3.1	3.8573\\
3.11	3.7602\\
3.12	3.6583\\
3.13	3.5519\\
3.14	3.4415\\
3.15	3.3273\\
3.16	3.2099\\
3.17	3.0896\\
3.18	2.9669\\
3.19	2.842\\
3.2	2.7154\\
3.21	2.5876\\
3.22	2.4589\\
3.23	2.3297\\
3.24	2.2004\\
3.25	2.0714\\
3.26	1.943\\
3.27	1.8156\\
3.28	1.6895\\
3.29	1.5651\\
3.3	1.4425\\
3.31	1.3222\\
3.32	1.2043\\
3.33	1.0891\\
3.34	0.97685\\
3.35	0.86767\\
3.36	0.76174\\
3.37	0.65922\\
3.38	0.56022\\
3.39	0.46485\\
3.4	0.37317\\
3.41	0.28524\\
3.42	0.20108\\
3.43	0.12071\\
3.44	0.044122\\
3.45	-0.028725\\
3.46	-0.097872\\
3.47	-0.16338\\
3.48	-0.22534\\
3.49	-0.28382\\
3.5	-0.33894\\
3.51	-0.39081\\
3.52	-0.43955\\
3.53	-0.48529\\
3.54	-0.52815\\
3.55	-0.56828\\
3.56	-0.60582\\
3.57	-0.64091\\
3.58	-0.67368\\
3.59	-0.70427\\
3.6	-0.73281\\
3.61	-0.75943\\
3.62	-0.78425\\
3.63	-0.80739\\
3.64	-0.82895\\
3.65	-0.84904\\
3.66	-0.86774\\
3.67	-0.88514\\
3.68	-0.9013\\
3.69	-0.9163\\
3.7	-0.93017\\
3.71	-0.94297\\
3.72	-0.95473\\
3.73	-0.96548\\
3.74	-0.97523\\
3.75	-0.98399\\
3.76	-0.99176\\
3.77	-0.99853\\
3.78	-1.0043\\
3.79	-1.0091\\
3.8	-1.0128\\
3.81	-1.0154\\
3.82	-1.017\\
3.83	-1.0174\\
3.84	-1.0167\\
3.85	-1.0148\\
3.86	-1.0117\\
3.87	-1.0073\\
3.88	-1.0017\\
3.89	-0.99482\\
3.9	-0.98661\\
3.91	-0.97707\\
3.92	-0.9662\\
3.93	-0.95397\\
3.94	-0.94039\\
3.95	-0.92545\\
3.96	-0.90915\\
3.97	-0.89151\\
3.98	-0.87254\\
3.99	-0.85226\\
4	-0.83069\\
4.01	-0.80785\\
4.02	-0.78378\\
4.03	-0.75852\\
4.04	-0.73211\\
4.05	-0.70458\\
4.06	-0.676\\
4.07	-0.64641\\
4.08	-0.61586\\
4.09	-0.58443\\
4.1	-0.55216\\
4.11	-0.51914\\
4.12	-0.48542\\
4.13	-0.45108\\
4.14	-0.4162\\
4.15	-0.38085\\
4.16	-0.34512\\
4.17	-0.30911\\
4.18	-0.27288\\
4.19	-0.23656\\
4.2	-0.20022\\
4.21	-0.16397\\
4.22	-0.12792\\
4.23	-0.092185\\
4.24	-0.056872\\
4.25	-0.022103\\
4.26	0.011999\\
4.27	0.045306\\
4.28	0.077686\\
4.29	0.109\\
4.3	0.13912\\
4.31	0.16789\\
4.32	0.19517\\
4.33	0.22082\\
4.34	0.24468\\
4.35	0.2666\\
4.36	0.28643\\
4.37	0.30402\\
4.38	0.31923\\
4.39	0.3319\\
4.4	0.34188\\
4.41	0.34905\\
4.42	0.35326\\
4.43	0.35439\\
4.44	0.35232\\
4.45	0.34693\\
4.46	0.33813\\
4.47	0.32581\\
4.48	0.30992\\
4.49	0.29037\\
4.5	0.26711\\
4.51	0.24012\\
4.52	0.20937\\
4.53	0.17485\\
4.54	0.13658\\
4.55	0.094592\\
4.56	0.048929\\
4.57	-0.00034111\\
4.58	-0.053133\\
4.59	-0.10934\\
4.6	-0.16885\\
4.61	-0.23152\\
4.62	-0.29719\\
4.63	-0.3657\\
4.64	-0.43685\\
4.65	-0.51044\\
4.66	-0.58626\\
4.67	-0.66407\\
4.68	-0.74365\\
4.69	-0.82473\\
4.7	-0.90706\\
4.71	-0.99038\\
4.72	-1.0744\\
4.73	-1.1589\\
4.74	-1.2436\\
4.75	-1.3281\\
4.76	-1.4123\\
4.77	-1.4958\\
4.78	-1.5784\\
4.79	-1.6598\\
4.8	-1.7398\\
4.81	-1.8182\\
4.82	-1.8946\\
4.83	-1.969\\
4.84	-2.041\\
4.85	-2.1105\\
4.86	-2.1773\\
4.87	-2.2413\\
4.88	-2.3023\\
4.89	-2.3601\\
4.9	-2.4148\\
4.91	-2.4661\\
4.92	-2.5141\\
4.93	-2.5586\\
4.94	-2.5996\\
4.95	-2.6372\\
4.96	-2.6712\\
4.97	-2.7018\\
4.98	-2.7289\\
4.99	-2.7525\\
5	-2.7729\\
5.01	-2.7899\\
5.02	-2.8037\\
5.03	-2.8144\\
5.04	-2.8221\\
5.05	-2.8269\\
5.06	-2.8289\\
5.07	-2.8282\\
5.08	-2.8249\\
5.09	-2.8193\\
5.1	-2.8113\\
5.11	-2.8012\\
5.12	-2.789\\
5.13	-2.7749\\
5.14	-2.759\\
5.15	-2.7415\\
5.16	-2.7223\\
5.17	-2.7018\\
5.18	-2.6798\\
5.19	-2.6566\\
5.2	-2.6322\\
5.21	-2.6067\\
5.22	-2.5802\\
5.23	-2.5526\\
5.24	-2.5241\\
5.25	-2.4946\\
5.26	-2.4643\\
5.27	-2.433\\
5.28	-2.4008\\
5.29	-2.3677\\
5.3	-2.3337\\
5.31	-2.2987\\
5.32	-2.2626\\
5.33	-2.2255\\
5.34	-2.1872\\
5.35	-2.1478\\
5.36	-2.107\\
5.37	-2.0649\\
5.38	-2.0213\\
5.39	-1.9761\\
5.4	-1.9292\\
5.41	-1.8806\\
5.42	-1.8301\\
5.43	-1.7776\\
5.44	-1.723\\
5.45	-1.6662\\
5.46	-1.6071\\
5.47	-1.5455\\
5.48	-1.4815\\
5.49	-1.4147\\
5.5	-1.3453\\
5.51	-1.273\\
5.52	-1.1979\\
5.53	-1.1198\\
5.54	-1.0386\\
5.55	-0.95437\\
5.56	-0.86701\\
5.57	-0.77649\\
5.58	-0.68278\\
5.59	-0.58589\\
5.6	-0.48582\\
5.61	-0.38258\\
5.62	-0.27622\\
5.63	-0.16678\\
5.64	-0.054327\\
5.65	0.061057\\
5.66	0.17928\\
5.67	0.30022\\
5.68	0.42375\\
5.69	0.54974\\
5.7	0.67801\\
5.71	0.80839\\
5.72	0.94068\\
5.73	1.0747\\
5.74	1.2102\\
5.75	1.3469\\
5.76	1.4846\\
5.77	1.6231\\
5.78	1.762\\
5.79	1.9011\\
5.8	2.04\\
5.81	2.1784\\
5.82	2.3161\\
5.83	2.4527\\
5.84	2.5878\\
5.85	2.7212\\
5.86	2.8524\\
5.87	2.9812\\
5.88	3.1072\\
5.89	3.23\\
5.9	3.3492\\
5.91	3.4647\\
5.92	3.5759\\
5.93	3.6826\\
5.94	3.7845\\
5.95	3.8812\\
5.96	3.9725\\
5.97	4.058\\
5.98	4.1375\\
5.99	4.2107\\
6	4.2774\\
6.01	4.3374\\
6.02	4.3905\\
6.03	4.4364\\
6.04	4.4751\\
6.05	4.5064\\
6.06	4.5301\\
6.07	4.5463\\
6.08	4.5548\\
6.09	4.5557\\
6.1	4.5489\\
6.11	4.5344\\
6.12	4.5124\\
6.13	4.4828\\
6.14	4.4458\\
6.15	4.4016\\
6.16	4.3502\\
6.17	4.2919\\
6.18	4.2269\\
6.19	4.1555\\
6.2	4.0778\\
6.21	3.9942\\
6.22	3.9049\\
6.23	3.8104\\
6.24	3.7108\\
6.25	3.6067\\
6.26	3.4983\\
6.27	3.386\\
6.28	3.2702\\
6.29	3.1513\\
6.3	3.0298\\
6.31	2.9059\\
6.32	2.7802\\
6.33	2.6529\\
6.34	2.5246\\
6.35	2.3956\\
6.36	2.2664\\
6.37	2.1372\\
6.38	2.0084\\
6.39	1.8805\\
6.4	1.7537\\
6.41	1.6283\\
6.42	1.5048\\
6.43	1.3833\\
6.44	1.2641\\
6.45	1.1475\\
6.46	1.0337\\
6.47	0.92295\\
6.48	0.81535\\
6.49	0.71107\\
6.5	0.61026\\
6.51	0.51303\\
6.52	0.41946\\
6.53	0.32961\\
6.54	0.24353\\
6.55	0.16123\\
6.56	0.082714\\
6.57	0.0079614\\
6.58	-0.063066\\
6.59	-0.13042\\
6.6	-0.19418\\
6.61	-0.25442\\
6.62	-0.31124\\
6.63	-0.36476\\
6.64	-0.41508\\
6.65	-0.46233\\
6.66	-0.50664\\
6.67	-0.54815\\
6.68	-0.58699\\
6.69	-0.62331\\
6.7	-0.65725\\
6.71	-0.68893\\
6.72	-0.7185\\
6.73	-0.74608\\
6.74	-0.77181\\
6.75	-0.79579\\
6.76	-0.81815\\
6.77	-0.83897\\
6.78	-0.85837\\
6.79	-0.87642\\
6.8	-0.89321\\
6.81	-0.90879\\
6.82	-0.92323\\
6.83	-0.93658\\
6.84	-0.94886\\
6.85	-0.96013\\
6.86	-0.97038\\
6.87	-0.97964\\
6.88	-0.98792\\
6.89	-0.9952\\
6.9	-1.0015\\
6.91	-1.0068\\
6.92	-1.011\\
6.93	-1.0142\\
6.94	-1.0163\\
6.95	-1.0173\\
6.96	-1.0172\\
6.97	-1.0159\\
6.98	-1.0134\\
6.99	-1.0097\\
7	-1.0047\\
7.01	-0.9985\\
7.02	-0.99096\\
7.03	-0.9821\\
7.04	-0.97191\\
7.05	-0.96038\\
7.06	-0.94748\\
7.07	-0.93324\\
7.08	-0.91763\\
7.09	-0.90068\\
7.1	-0.88238\\
7.11	-0.86277\\
7.12	-0.84185\\
7.13	-0.81965\\
7.14	-0.79621\\
7.15	-0.77155\\
7.16	-0.74572\\
7.17	-0.71876\\
7.18	-0.69071\\
7.19	-0.66162\\
7.2	-0.63156\\
7.21	-0.60057\\
7.22	-0.56872\\
7.23	-0.53607\\
7.24	-0.5027\\
7.25	-0.46866\\
7.26	-0.43405\\
7.27	-0.39893\\
7.28	-0.36339\\
7.29	-0.32751\\
7.3	-0.29138\\
7.31	-0.25509\\
7.32	-0.21874\\
7.33	-0.18244\\
7.34	-0.14628\\
7.35	-0.11037\\
7.36	-0.074821\\
7.37	-0.03976\\
7.38	-0.0053016\\
7.39	0.028427\\
7.4	0.061296\\
7.41	0.093173\\
7.42	0.12392\\
7.43	0.1534\\
7.44	0.18146\\
7.45	0.20795\\
7.46	0.23274\\
7.47	0.25567\\
7.48	0.27659\\
7.49	0.29534\\
7.5	0.31178\\
7.51	0.32576\\
7.52	0.33713\\
7.53	0.34576\\
7.54	0.35149\\
7.55	0.35421\\
7.56	0.35378\\
7.57	0.3501\\
7.58	0.34305\\
7.59	0.33254\\
7.6	0.31848\\
7.61	0.3008\\
7.62	0.27944\\
7.63	0.25436\\
7.64	0.22552\\
7.65	0.19293\\
7.66	0.15657\\
7.67	0.11647\\
7.68	0.072672\\
7.69	0.025232\\
7.7	-0.025775\\
7.71	-0.080256\\
7.72	-0.1381\\
7.73	-0.19917\\
7.74	-0.26334\\
7.75	-0.33042\\
7.76	-0.40024\\
7.77	-0.47261\\
7.78	-0.54732\\
7.79	-0.62415\\
7.8	-0.70286\\
7.81	-0.7832\\
7.82	-0.86493\\
7.83	-0.94778\\
7.84	-1.0315\\
7.85	-1.1158\\
7.86	-1.2004\\
7.87	-1.285\\
7.88	-1.3694\\
7.89	-1.4533\\
7.9	-1.5364\\
7.91	-1.6185\\
7.92	-1.6992\\
7.93	-1.7785\\
7.94	-1.8559\\
7.95	-1.9313\\
7.96	-2.0046\\
7.97	-2.0754\\
7.98	-2.1436\\
7.99	-2.209\\
8	-2.2715\\
8.01	-2.331\\
8.02	-2.3873\\
8.03	-2.4404\\
8.04	-2.49\\
8.05	-2.5363\\
8.06	-2.5791\\
8.07	-2.6185\\
8.08	-2.6543\\
8.09	-2.6866\\
8.1	-2.7155\\
8.11	-2.7409\\
8.12	-2.7629\\
8.13	-2.7816\\
8.14	-2.7971\\
8.15	-2.8094\\
8.16	-2.8186\\
8.17	-2.8248\\
8.18	-2.8282\\
8.19	-2.8289\\
8.2	-2.8269\\
8.21	-2.8224\\
8.22	-2.8156\\
8.23	-2.8066\\
8.24	-2.7954\\
8.25	-2.7823\\
8.26	-2.7673\\
8.27	-2.7506\\
8.28	-2.7323\\
8.29	-2.7124\\
8.3	-2.6912\\
8.31	-2.6686\\
8.32	-2.6448\\
8.33	-2.6199\\
8.34	-2.5938\\
8.35	-2.5668\\
8.36	-2.5387\\
8.37	-2.5098\\
8.38	-2.4799\\
8.39	-2.449\\
8.4	-2.4173\\
8.41	-2.3847\\
8.42	-2.3512\\
8.43	-2.3166\\
8.44	-2.2811\\
8.45	-2.2446\\
8.46	-2.2069\\
8.47	-2.1681\\
8.48	-2.128\\
8.49	-2.0865\\
8.5	-2.0437\\
8.51	-1.9993\\
8.52	-1.9533\\
8.53	-1.9056\\
8.54	-1.8561\\
8.55	-1.8046\\
8.56	-1.7511\\
8.57	-1.6955\\
8.58	-1.6375\\
8.59	-1.5772\\
8.6	-1.5145\\
8.61	-1.4491\\
8.62	-1.381\\
8.63	-1.3102\\
8.64	-1.2366\\
8.65	-1.16\\
8.66	-1.0804\\
8.67	-0.99772\\
8.68	-0.91196\\
8.69	-0.82305\\
8.7	-0.73097\\
8.71	-0.6357\\
8.72	-0.53725\\
8.73	-0.43562\\
8.74	-0.33085\\
8.75	-0.22297\\
8.76	-0.11205\\
8.77	0.0018507\\
8.78	0.11864\\
8.79	0.2382\\
8.8	0.36043\\
8.81	0.48519\\
8.82	0.61231\\
8.83	0.74164\\
8.84	0.87298\\
8.85	1.0061\\
8.86	1.1409\\
8.87	1.277\\
8.88	1.4143\\
8.89	1.5524\\
8.9	1.6911\\
8.91	1.8301\\
8.92	1.9692\\
8.93	2.1079\\
8.94	2.246\\
8.95	2.3832\\
8.96	2.5191\\
8.97	2.6534\\
8.98	2.7858\\
8.99	2.9159\\
9	3.0433\\
9.01	3.1678\\
9.02	3.2889\\
9.03	3.4063\\
9.04	3.5197\\
9.05	3.6288\\
9.06	3.7332\\
9.07	3.8325\\
9.08	3.9266\\
9.09	4.0151\\
9.1	4.0977\\
9.11	4.1742\\
9.12	4.2442\\
9.13	4.3077\\
9.14	4.3643\\
9.15	4.4139\\
9.16	4.4563\\
9.17	4.4913\\
9.18	4.5189\\
9.19	4.539\\
9.2	4.5515\\
9.21	4.5562\\
9.22	4.5533\\
9.23	4.5428\\
9.24	4.5246\\
9.25	4.4988\\
9.26	4.4656\\
9.27	4.4251\\
9.28	4.3773\\
9.29	4.3225\\
9.3	4.2609\\
9.31	4.1927\\
9.32	4.1182\\
9.33	4.0375\\
9.34	3.9511\\
9.35	3.8592\\
9.36	3.7622\\
9.37	3.6603\\
9.38	3.5541\\
9.39	3.4437\\
9.4	3.3297\\
9.41	3.2123\\
9.42	3.0921\\
9.43	2.9693\\
9.44	2.8445\\
9.45	2.718\\
9.46	2.5902\\
9.47	2.4615\\
9.48	2.3323\\
9.49	2.203\\
9.5	2.074\\
9.51	1.9456\\
9.52	1.8182\\
9.53	1.6921\\
9.54	1.5676\\
9.55	1.445\\
9.56	1.3246\\
9.57	1.2067\\
9.58	1.0914\\
9.59	0.97907\\
9.6	0.86982\\
9.61	0.76383\\
9.62	0.66124\\
9.63	0.56217\\
9.64	0.46672\\
9.65	0.37497\\
9.66	0.28696\\
9.67	0.20273\\
9.68	0.12229\\
9.69	0.045617\\
9.7	-0.027304\\
9.71	-0.096525\\
9.72	-0.16211\\
9.73	-0.22413\\
9.74	-0.28269\\
9.75	-0.33787\\
9.76	-0.38981\\
9.77	-0.4386\\
9.78	-0.4844\\
9.79	-0.52732\\
9.8	-0.56751\\
9.81	-0.6051\\
9.82	-0.64023\\
9.83	-0.67304\\
9.84	-0.70368\\
9.85	-0.73226\\
9.86	-0.75891\\
9.87	-0.78377\\
9.88	-0.80694\\
9.89	-0.82854\\
9.9	-0.84865\\
9.91	-0.86738\\
9.92	-0.8848\\
9.93	-0.90099\\
9.94	-0.91601\\
9.95	-0.9299\\
9.96	-0.94273\\
9.97	-0.95451\\
9.98	-0.96528\\
9.99	-0.97504\\
10	-0.98382\\
};
\addlegendentry{$\dot m^+_\text{constr}$};

\addplot [color=mycolor1,densely dashdotted ,very thick]
  table[row sep=crcr]{%
0	5.0393\\
0.01	4.8684\\
0.02	4.687\\
0.03	4.4962\\
0.04	4.2968\\
0.05	4.0899\\
0.06	3.8766\\
0.07	3.6579\\
0.08	3.435\\
0.09	3.2091\\
0.1	2.9812\\
0.11	2.7524\\
0.12	2.524\\
0.13	2.297\\
0.14	2.0725\\
0.15	1.8515\\
0.16	1.6351\\
0.17	1.424\\
0.18	1.2193\\
0.19	1.0218\\
0.2	0.83208\\
0.21	0.65092\\
0.22	0.47888\\
0.23	0.31644\\
0.24	0.16402\\
0.25	0.021914\\
0.26	-0.10964\\
0.27	-0.23053\\
0.28	-0.34069\\
0.29	-0.4402\\
0.3	-0.52917\\
0.31	-0.60783\\
0.32	-0.67649\\
0.33	-0.73551\\
0.34	-0.78533\\
0.35	-0.82646\\
0.36	-0.85945\\
0.37	-0.88492\\
0.38	-0.90349\\
0.39	-0.91585\\
0.4	-0.92269\\
0.41	-0.92473\\
0.42	-0.92268\\
0.43	-0.91724\\
0.44	-0.90912\\
0.45	-0.89901\\
0.46	-0.88754\\
0.47	-0.87534\\
0.48	-0.86298\\
0.49	-0.85098\\
0.5	-0.83983\\
0.51	-0.82994\\
0.52	-0.82165\\
0.53	-0.81528\\
0.54	-0.81103\\
0.55	-0.80908\\
0.56	-0.80952\\
0.57	-0.81238\\
0.58	-0.81764\\
0.59	-0.82519\\
0.6	-0.83489\\
0.61	-0.84656\\
0.62	-0.85993\\
0.63	-0.87474\\
0.64	-0.89066\\
0.65	-0.90735\\
0.66	-0.92445\\
0.67	-0.94158\\
0.68	-0.95835\\
0.69	-0.97439\\
0.7	-0.98932\\
0.71	-1.0028\\
0.72	-1.0145\\
0.73	-1.024\\
0.74	-1.0312\\
0.75	-1.0357\\
0.76	-1.0374\\
0.77	-1.0361\\
0.78	-1.0316\\
0.79	-1.0241\\
0.8	-1.0133\\
0.81	-0.99929\\
0.82	-0.98221\\
0.83	-0.96214\\
0.84	-0.93923\\
0.85	-0.91367\\
0.86	-0.88567\\
0.87	-0.85548\\
0.88	-0.82337\\
0.89	-0.78963\\
0.9	-0.75456\\
0.91	-0.71846\\
0.92	-0.68164\\
0.93	-0.64438\\
0.94	-0.60699\\
0.95	-0.56974\\
0.96	-0.53287\\
0.97	-0.49662\\
0.98	-0.46118\\
0.99	-0.42673\\
1	-0.39342\\
1.01	-0.36134\\
1.02	-0.33056\\
1.03	-0.30113\\
1.04	-0.27306\\
1.05	-0.24631\\
1.06	-0.22084\\
1.07	-0.19657\\
1.08	-0.17338\\
1.09	-0.15117\\
1.1	-0.1298\\
1.11	-0.10912\\
1.12	-0.088993\\
1.13	-0.069263\\
1.14	-0.049792\\
1.15	-0.030447\\
1.16	-0.011111\\
1.17	0.0083136\\
1.18	0.027901\\
1.19	0.047699\\
1.2	0.067725\\
1.21	0.087964\\
1.22	0.10837\\
1.23	0.12885\\
1.24	0.1493\\
1.25	0.16956\\
1.26	0.18943\\
1.27	0.2087\\
1.28	0.22712\\
1.29	0.2444\\
1.3	0.26025\\
1.31	0.27433\\
1.32	0.28632\\
1.33	0.29586\\
1.34	0.3026\\
1.35	0.30619\\
1.36	0.30628\\
1.37	0.30252\\
1.38	0.2946\\
1.39	0.28222\\
1.4	0.2651\\
1.41	0.243\\
1.42	0.21571\\
1.43	0.18307\\
1.44	0.14494\\
1.45	0.10126\\
1.46	0.051976\\
1.47	-0.0028812\\
1.48	-0.063249\\
1.49	-0.12901\\
1.5	-0.2\\
1.51	-0.27599\\
1.52	-0.35674\\
1.53	-0.44191\\
1.54	-0.53117\\
1.55	-0.62412\\
1.56	-0.72034\\
1.57	-0.81936\\
1.58	-0.9207\\
1.59	-1.0239\\
1.6	-1.1283\\
1.61	-1.2335\\
1.62	-1.3389\\
1.63	-1.4439\\
1.64	-1.548\\
1.65	-1.6507\\
1.66	-1.7515\\
1.67	-1.8498\\
1.68	-1.9452\\
1.69	-2.0371\\
1.7	-2.1252\\
1.71	-2.2091\\
1.72	-2.2885\\
1.73	-2.3629\\
1.74	-2.4321\\
1.75	-2.4959\\
1.76	-2.5541\\
1.77	-2.6066\\
1.78	-2.6531\\
1.79	-2.6937\\
1.8	-2.7284\\
1.81	-2.7572\\
1.82	-2.78\\
1.83	-2.7971\\
1.84	-2.8086\\
1.85	-2.8147\\
1.86	-2.8155\\
1.87	-2.8112\\
1.88	-2.8023\\
1.89	-2.7889\\
1.9	-2.7713\\
1.91	-2.75\\
1.92	-2.7251\\
1.93	-2.6971\\
1.94	-2.6663\\
1.95	-2.6331\\
1.96	-2.5979\\
1.97	-2.5609\\
1.98	-2.5226\\
1.99	-2.4833\\
2	-2.4433\\
2.01	-2.403\\
2.02	-2.3626\\
2.03	-2.3225\\
2.04	-2.283\\
2.05	-2.2443\\
2.06	-2.2066\\
2.07	-2.1701\\
2.08	-2.1351\\
2.09	-2.1017\\
2.1	-2.07\\
2.11	-2.0402\\
2.12	-2.0123\\
2.13	-1.9865\\
2.14	-1.9627\\
2.15	-1.941\\
2.16	-1.9213\\
2.17	-1.9036\\
2.18	-1.8879\\
2.19	-1.8741\\
2.2	-1.8621\\
2.21	-1.8517\\
2.22	-1.8428\\
2.23	-1.8352\\
2.24	-1.8287\\
2.25	-1.8232\\
2.26	-1.8184\\
2.27	-1.814\\
2.28	-1.8099\\
2.29	-1.8057\\
2.3	-1.8011\\
2.31	-1.7959\\
2.32	-1.7897\\
2.33	-1.7823\\
2.34	-1.7733\\
2.35	-1.7623\\
2.36	-1.7492\\
2.37	-1.7334\\
2.38	-1.7148\\
2.39	-1.6929\\
2.4	-1.6674\\
2.41	-1.638\\
2.42	-1.6043\\
2.43	-1.5661\\
2.44	-1.523\\
2.45	-1.4747\\
2.46	-1.4209\\
2.47	-1.3614\\
2.48	-1.2959\\
2.49	-1.2241\\
2.5	-1.1459\\
2.51	-1.061\\
2.52	-0.96926\\
2.53	-0.87055\\
2.54	-0.76474\\
2.55	-0.65174\\
2.56	-0.5315\\
2.57	-0.40401\\
2.58	-0.26928\\
2.59	-0.12736\\
2.6	0.021632\\
2.61	0.17757\\
2.62	0.34027\\
2.63	0.50951\\
2.64	0.68501\\
2.65	0.86647\\
2.66	1.0535\\
2.67	1.2457\\
2.68	1.4426\\
2.69	1.6438\\
2.7	1.8486\\
2.71	2.0566\\
2.72	2.267\\
2.73	2.4791\\
2.74	2.6924\\
2.75	2.906\\
2.76	3.1192\\
2.77	3.3311\\
2.78	3.541\\
2.79	3.748\\
2.8	3.9512\\
2.81	4.1498\\
2.82	4.3429\\
2.83	4.5296\\
2.84	4.709\\
2.85	4.8804\\
2.86	5.0427\\
2.87	5.1953\\
2.88	5.3372\\
2.89	5.4679\\
2.9	5.5865\\
2.91	5.6923\\
2.92	5.7849\\
2.93	5.8635\\
2.94	5.9277\\
2.95	5.9771\\
2.96	6.0113\\
2.97	6.0301\\
2.98	6.0332\\
2.99	6.0205\\
3	5.992\\
3.01	5.9477\\
3.02	5.8878\\
3.03	5.8125\\
3.04	5.7222\\
3.05	5.6171\\
3.06	5.4979\\
3.07	5.3651\\
3.08	5.2194\\
3.09	5.0614\\
3.1	4.892\\
3.11	4.712\\
3.12	4.5224\\
3.13	4.3241\\
3.14	4.1181\\
3.15	3.9056\\
3.16	3.6876\\
3.17	3.4652\\
3.18	3.2396\\
3.19	3.0119\\
3.2	2.7832\\
3.21	2.5547\\
3.22	2.3274\\
3.23	2.1025\\
3.24	1.881\\
3.25	1.6639\\
3.26	1.452\\
3.27	1.2464\\
3.28	1.0479\\
3.29	0.85709\\
3.3	0.67475\\
3.31	0.50145\\
3.32	0.3377\\
3.33	0.18391\\
3.34	0.040403\\
3.35	-0.092582\\
3.36	-0.2149\\
3.37	-0.32651\\
3.38	-0.42744\\
3.39	-0.51782\\
3.4	-0.59785\\
3.41	-0.66783\\
3.42	-0.72812\\
3.43	-0.77915\\
3.44	-0.82142\\
3.45	-0.85547\\
3.46	-0.88191\\
3.47	-0.90137\\
3.48	-0.91453\\
3.49	-0.92207\\
3.5	-0.92471\\
3.51	-0.92316\\
3.52	-0.91814\\
3.53	-0.91035\\
3.54	-0.90046\\
3.55	-0.88914\\
3.56	-0.877\\
3.57	-0.86463\\
3.58	-0.85256\\
3.59	-0.84126\\
3.6	-0.83118\\
3.61	-0.82266\\
3.62	-0.81601\\
3.63	-0.81147\\
3.64	-0.80921\\
3.65	-0.80932\\
3.66	-0.81186\\
3.67	-0.81679\\
3.68	-0.82404\\
3.69	-0.83347\\
3.7	-0.84488\\
3.71	-0.85805\\
3.72	-0.87268\\
3.73	-0.88847\\
3.74	-0.90508\\
3.75	-0.92214\\
3.76	-0.93929\\
3.77	-0.95613\\
3.78	-0.97229\\
3.79	-0.98739\\
3.8	-1.0011\\
3.81	-1.013\\
3.82	-1.0229\\
3.83	-1.0303\\
3.84	-1.0352\\
3.85	-1.0373\\
3.86	-1.0364\\
3.87	-1.0324\\
3.88	-1.0253\\
3.89	-1.0149\\
3.9	-1.0013\\
3.91	-0.98469\\
3.92	-0.96501\\
3.93	-0.94247\\
3.94	-0.91725\\
3.95	-0.88956\\
3.96	-0.85965\\
3.97	-0.82779\\
3.98	-0.79425\\
3.99	-0.75934\\
4	-0.72336\\
4.01	-0.68662\\
4.02	-0.6494\\
4.03	-0.61202\\
4.04	-0.57473\\
4.05	-0.53779\\
4.06	-0.50144\\
4.07	-0.46589\\
4.08	-0.4313\\
4.09	-0.39782\\
4.1	-0.36557\\
4.11	-0.33462\\
4.12	-0.30501\\
4.13	-0.27676\\
4.14	-0.24983\\
4.15	-0.22419\\
4.16	-0.19976\\
4.17	-0.17644\\
4.18	-0.1541\\
4.19	-0.13263\\
4.2	-0.11187\\
4.21	-0.091671\\
4.22	-0.071897\\
4.23	-0.052399\\
4.24	-0.033044\\
4.25	-0.013713\\
4.26	0.0056949\\
4.27	0.025257\\
4.28	0.045025\\
4.29	0.06502\\
4.3	0.085233\\
4.31	0.10562\\
4.32	0.1261\\
4.33	0.14656\\
4.34	0.16685\\
4.35	0.18679\\
4.36	0.20616\\
4.37	0.2247\\
4.38	0.24215\\
4.39	0.25821\\
4.4	0.27255\\
4.41	0.28484\\
4.42	0.29474\\
4.43	0.30187\\
4.44	0.30591\\
4.45	0.30648\\
4.46	0.30326\\
4.47	0.29592\\
4.48	0.28415\\
4.49	0.26768\\
4.5	0.24627\\
4.51	0.21969\\
4.52	0.18777\\
4.53	0.15039\\
4.54	0.10745\\
4.55	0.058925\\
4.56	0.004814\\
4.57	-0.054819\\
4.58	-0.11986\\
4.59	-0.19016\\
4.6	-0.2655\\
4.61	-0.34562\\
4.62	-0.43022\\
4.63	-0.51895\\
4.64	-0.61143\\
4.65	-0.70723\\
4.66	-0.80591\\
4.67	-0.90696\\
4.68	-1.0099\\
4.69	-1.1142\\
4.7	-1.2193\\
4.71	-1.3247\\
4.72	-1.4298\\
4.73	-1.5341\\
4.74	-1.637\\
4.75	-1.7381\\
4.76	-1.8368\\
4.77	-1.9325\\
4.78	-2.025\\
4.79	-2.1136\\
4.8	-2.1981\\
4.81	-2.2781\\
4.82	-2.3532\\
4.83	-2.4231\\
4.84	-2.4877\\
4.85	-2.5466\\
4.86	-2.5999\\
4.87	-2.6472\\
4.88	-2.6886\\
4.89	-2.7241\\
4.9	-2.7536\\
4.91	-2.7773\\
4.92	-2.7952\\
4.93	-2.8074\\
4.94	-2.8142\\
4.95	-2.8156\\
4.96	-2.8121\\
4.97	-2.8038\\
4.98	-2.7909\\
4.99	-2.7739\\
5	-2.753\\
5.01	-2.7286\\
5.02	-2.701\\
5.03	-2.6706\\
5.04	-2.6377\\
5.05	-2.6027\\
5.06	-2.5659\\
5.07	-2.5278\\
5.08	-2.4886\\
5.09	-2.4487\\
5.1	-2.4084\\
5.11	-2.368\\
5.12	-2.3279\\
5.13	-2.2883\\
5.14	-2.2494\\
5.15	-2.2116\\
5.16	-2.1749\\
5.17	-2.1397\\
5.18	-2.1061\\
5.19	-2.0742\\
5.2	-2.0441\\
5.21	-2.016\\
5.22	-1.9898\\
5.23	-1.9658\\
5.24	-1.9438\\
5.25	-1.9238\\
5.26	-1.9059\\
5.27	-1.8899\\
5.28	-1.8759\\
5.29	-1.8636\\
5.3	-1.853\\
5.31	-1.8439\\
5.32	-1.8361\\
5.33	-1.8296\\
5.34	-1.8239\\
5.35	-1.819\\
5.36	-1.8146\\
5.37	-1.8104\\
5.38	-1.8062\\
5.39	-1.8017\\
5.4	-1.7966\\
5.41	-1.7906\\
5.42	-1.7834\\
5.43	-1.7746\\
5.44	-1.7639\\
5.45	-1.7511\\
5.46	-1.7357\\
5.47	-1.7175\\
5.48	-1.696\\
5.49	-1.671\\
5.5	-1.6422\\
5.51	-1.6091\\
5.52	-1.5715\\
5.53	-1.5291\\
5.54	-1.4815\\
5.55	-1.4285\\
5.56	-1.3697\\
5.57	-1.305\\
5.58	-1.2341\\
5.59	-1.1568\\
5.6	-1.0728\\
5.61	-0.982\\
5.62	-0.88423\\
5.63	-0.77938\\
5.64	-0.66735\\
5.65	-0.54808\\
5.66	-0.42157\\
5.67	-0.2878\\
5.68	-0.14685\\
5.69	0.0011999\\
5.7	0.15622\\
5.71	0.31802\\
5.72	0.4864\\
5.73	0.66107\\
5.74	0.84175\\
5.75	1.0281\\
5.76	1.2196\\
5.77	1.4159\\
5.78	1.6165\\
5.79	1.8209\\
5.8	2.0285\\
5.81	2.2386\\
5.82	2.4505\\
5.83	2.6637\\
5.84	2.8773\\
5.85	3.0906\\
5.86	3.3027\\
5.87	3.513\\
5.88	3.7204\\
5.89	3.9242\\
5.9	4.1234\\
5.91	4.3173\\
5.92	4.5049\\
5.93	4.6854\\
5.94	4.8578\\
5.95	5.0214\\
5.96	5.1754\\
5.97	5.3188\\
5.98	5.451\\
5.99	5.5713\\
6	5.6789\\
6.01	5.7732\\
6.02	5.8537\\
6.03	5.9199\\
6.04	5.9713\\
6.05	6.0076\\
6.06	6.0285\\
6.07	6.0337\\
6.08	6.0231\\
6.09	5.9967\\
6.1	5.9545\\
6.11	5.8967\\
6.12	5.8235\\
6.13	5.7352\\
6.14	5.6321\\
6.15	5.5148\\
6.16	5.3838\\
6.17	5.2397\\
6.18	5.0833\\
6.19	4.9154\\
6.2	4.7368\\
6.21	4.5484\\
6.22	4.3512\\
6.23	4.1462\\
6.24	3.9345\\
6.25	3.7172\\
6.26	3.4953\\
6.27	3.2701\\
6.28	3.0426\\
6.29	2.8139\\
6.3	2.5853\\
6.31	2.3578\\
6.32	2.1326\\
6.33	1.9105\\
6.34	1.6927\\
6.35	1.4802\\
6.36	1.2737\\
6.37	1.0741\\
6.38	0.88225\\
6.39	0.69874\\
6.4	0.5242\\
6.41	0.35914\\
6.42	0.20399\\
6.43	0.059082\\
6.44	-0.075327\\
6.45	-0.19909\\
6.46	-0.31213\\
6.47	-0.41449\\
6.48	-0.50628\\
6.49	-0.58769\\
6.5	-0.659\\
6.51	-0.72057\\
6.52	-0.77282\\
6.53	-0.81623\\
6.54	-0.85135\\
6.55	-0.87878\\
6.56	-0.89914\\
6.57	-0.9131\\
6.58	-0.92136\\
6.59	-0.92461\\
6.6	-0.92359\\
6.61	-0.91899\\
6.62	-0.91153\\
6.63	-0.90189\\
6.64	-0.89072\\
6.65	-0.87866\\
6.66	-0.86629\\
6.67	-0.85414\\
6.68	-0.84272\\
6.69	-0.83245\\
6.7	-0.8237\\
6.71	-0.81679\\
6.72	-0.81195\\
6.73	-0.80937\\
6.74	-0.80917\\
6.75	-0.81138\\
6.76	-0.81599\\
6.77	-0.82294\\
6.78	-0.83208\\
6.79	-0.84324\\
6.8	-0.85618\\
6.81	-0.87064\\
6.82	-0.88629\\
6.83	-0.90281\\
6.84	-0.91984\\
6.85	-0.93699\\
6.86	-0.9539\\
6.87	-0.97017\\
6.88	-0.98544\\
6.89	-0.99933\\
6.9	-1.0115\\
6.91	-1.0217\\
6.92	-1.0295\\
6.93	-1.0347\\
6.94	-1.0372\\
6.95	-1.0367\\
6.96	-1.0331\\
6.97	-1.0264\\
6.98	-1.0165\\
6.99	-1.0034\\
7	-0.9871\\
7.01	-0.96783\\
7.02	-0.94566\\
7.03	-0.92079\\
7.04	-0.89342\\
7.05	-0.86379\\
7.06	-0.83217\\
7.07	-0.79885\\
7.08	-0.7641\\
7.09	-0.72825\\
7.1	-0.69159\\
7.11	-0.65442\\
7.12	-0.61704\\
7.13	-0.57972\\
7.14	-0.54273\\
7.15	-0.50629\\
7.16	-0.47061\\
7.17	-0.43589\\
7.18	-0.40225\\
7.19	-0.36983\\
7.2	-0.3387\\
7.21	-0.30891\\
7.22	-0.28047\\
7.23	-0.25338\\
7.24	-0.22757\\
7.25	-0.20298\\
7.26	-0.17951\\
7.27	-0.15705\\
7.28	-0.13547\\
7.29	-0.11462\\
7.3	-0.094358\\
7.31	-0.074535\\
7.32	-0.055008\\
7.33	-0.035642\\
7.34	-0.016313\\
7.35	0.003079\\
7.36	0.022617\\
7.37	0.042355\\
7.38	0.06232\\
7.39	0.082505\\
7.4	0.10287\\
7.41	0.12335\\
7.42	0.14382\\
7.43	0.16414\\
7.44	0.18414\\
7.45	0.2036\\
7.46	0.22227\\
7.47	0.23988\\
7.48	0.25615\\
7.49	0.27074\\
7.5	0.28332\\
7.51	0.29356\\
7.52	0.30109\\
7.53	0.30556\\
7.54	0.30662\\
7.55	0.30393\\
7.56	0.29716\\
7.57	0.286\\
7.58	0.27018\\
7.59	0.24944\\
7.6	0.22357\\
7.61	0.19238\\
7.62	0.15573\\
7.63	0.11355\\
7.64	0.065772\\
7.65	0.012409\\
7.66	-0.046487\\
7.67	-0.11081\\
7.68	-0.18041\\
7.69	-0.25509\\
7.7	-0.33458\\
7.71	-0.4186\\
7.72	-0.5068\\
7.73	-0.5988\\
7.74	-0.69418\\
7.75	-0.7925\\
7.76	-0.89326\\
7.77	-0.99598\\
7.78	-1.1001\\
7.79	-1.2052\\
7.8	-1.3105\\
7.81	-1.4157\\
7.82	-1.5202\\
7.83	-1.6233\\
7.84	-1.7246\\
7.85	-1.8236\\
7.86	-1.9198\\
7.87	-2.0128\\
7.88	-2.102\\
7.89	-2.187\\
7.9	-2.2676\\
7.91	-2.3434\\
7.92	-2.414\\
7.93	-2.4793\\
7.94	-2.539\\
7.95	-2.593\\
7.96	-2.6412\\
7.97	-2.6834\\
7.98	-2.7197\\
7.99	-2.75\\
8	-2.7745\\
8.01	-2.7931\\
8.02	-2.8061\\
8.03	-2.8136\\
8.04	-2.8157\\
8.05	-2.8129\\
8.06	-2.8052\\
8.07	-2.7929\\
8.08	-2.7764\\
8.09	-2.7561\\
8.1	-2.7321\\
8.11	-2.7049\\
8.12	-2.6748\\
8.13	-2.6423\\
8.14	-2.6075\\
8.15	-2.571\\
8.16	-2.533\\
8.17	-2.4939\\
8.18	-2.4541\\
8.19	-2.4138\\
8.2	-2.3734\\
8.21	-2.3333\\
8.22	-2.2936\\
8.23	-2.2546\\
8.24	-2.2166\\
8.25	-2.1798\\
8.26	-2.1444\\
8.27	-2.1105\\
8.28	-2.0783\\
8.29	-2.048\\
8.3	-2.0196\\
8.31	-1.9932\\
8.32	-1.9689\\
8.33	-1.9466\\
8.34	-1.9264\\
8.35	-1.9082\\
8.36	-1.892\\
8.37	-1.8777\\
8.38	-1.8652\\
8.39	-1.8543\\
8.4	-1.845\\
8.41	-1.8371\\
8.42	-1.8304\\
8.43	-1.8246\\
8.44	-1.8196\\
8.45	-1.8152\\
8.46	-1.811\\
8.47	-1.8068\\
8.48	-1.8024\\
8.49	-1.7974\\
8.5	-1.7915\\
8.51	-1.7844\\
8.52	-1.7759\\
8.53	-1.7655\\
8.54	-1.753\\
8.55	-1.7379\\
8.56	-1.7201\\
8.57	-1.6991\\
8.58	-1.6746\\
8.59	-1.6463\\
8.6	-1.6138\\
8.61	-1.5768\\
8.62	-1.535\\
8.63	-1.4882\\
8.64	-1.4359\\
8.65	-1.378\\
8.66	-1.3141\\
8.67	-1.244\\
8.68	-1.1676\\
8.69	-1.0845\\
8.7	-0.9946\\
8.71	-0.89778\\
8.72	-0.79388\\
8.73	-0.68282\\
8.74	-0.56453\\
8.75	-0.43899\\
8.76	-0.3062\\
8.77	-0.16621\\
8.78	-0.019106\\
8.79	0.13498\\
8.8	0.29589\\
8.81	0.46339\\
8.82	0.63724\\
8.83	0.81713\\
8.84	1.0027\\
8.85	1.1936\\
8.86	1.3893\\
8.87	1.5894\\
8.88	1.7933\\
8.89	2.0004\\
8.9	2.2102\\
8.91	2.422\\
8.92	2.635\\
8.93	2.8486\\
8.94	3.062\\
8.95	3.2743\\
8.96	3.4848\\
8.97	3.6927\\
8.98	3.897\\
8.99	4.0969\\
9	4.2916\\
9.01	4.4801\\
9.02	4.6616\\
9.03	4.8352\\
9.04	5\\
9.05	5.1553\\
9.06	5.3002\\
9.07	5.4339\\
9.08	5.5558\\
9.09	5.6652\\
9.1	5.7613\\
9.11	5.8437\\
9.12	5.9119\\
9.13	5.9653\\
9.14	6.0036\\
9.15	6.0266\\
9.16	6.0339\\
9.17	6.0254\\
9.18	6.0012\\
9.19	5.9611\\
9.2	5.9054\\
9.21	5.8342\\
9.22	5.7479\\
9.23	5.6468\\
9.24	5.5313\\
9.25	5.4021\\
9.26	5.2598\\
9.27	5.105\\
9.28	4.9386\\
9.29	4.7614\\
9.3	4.5742\\
9.31	4.3782\\
9.32	4.1742\\
9.33	3.9633\\
9.34	3.7467\\
9.35	3.5253\\
9.36	3.3005\\
9.37	3.0732\\
9.38	2.8447\\
9.39	2.616\\
9.4	2.3883\\
9.41	2.1627\\
9.42	1.9401\\
9.43	1.7217\\
9.44	1.5084\\
9.45	1.301\\
9.46	1.1005\\
9.47	0.90757\\
9.48	0.7229\\
9.49	0.54712\\
9.5	0.38075\\
9.51	0.22425\\
9.52	0.077951\\
9.53	-0.05788\\
9.54	-0.18308\\
9.55	-0.29756\\
9.56	-0.40135\\
9.57	-0.49455\\
9.58	-0.57734\\
9.59	-0.64999\\
9.6	-0.71285\\
9.61	-0.76632\\
9.62	-0.81089\\
9.63	-0.84709\\
9.64	-0.87552\\
9.65	-0.89679\\
9.66	-0.91157\\
9.67	-0.92055\\
9.68	-0.92444\\
9.69	-0.92395\\
9.7	-0.91979\\
9.71	-0.91268\\
9.72	-0.90329\\
9.73	-0.89229\\
9.74	-0.88031\\
9.75	-0.86795\\
9.76	-0.85574\\
9.77	-0.8442\\
9.78	-0.83375\\
9.79	-0.82478\\
9.8	-0.8176\\
9.81	-0.81248\\
9.82	-0.80959\\
9.83	-0.80905\\
9.84	-0.81094\\
9.85	-0.81524\\
9.86	-0.82188\\
9.87	-0.83073\\
9.88	-0.84163\\
9.89	-0.85435\\
9.9	-0.86862\\
9.91	-0.88413\\
9.92	-0.90055\\
9.93	-0.91753\\
9.94	-0.93469\\
9.95	-0.95165\\
9.96	-0.96804\\
9.97	-0.98346\\
9.98	-0.99756\\
9.99	-1.01\\
10	-1.0204\\
};
\addlegendentry{$\dot m^+_\text{uncon}$};

\addplot [color=black,solid,forget plot]
  table[row sep=crcr]{%
0	-40\\
0	40\\
};
\addplot [color=black,solid,forget plot]
  table[row sep=crcr]{%
0.25	-40\\
0.25	40\\
};
\addplot [color=black,solid,forget plot]
  table[row sep=crcr]{%
0.5	-40\\
0.5	40\\
};
\addplot [color=black,solid,forget plot]
  table[row sep=crcr]{%
0.75	-40\\
0.75	40\\
};
\addplot [color=black,solid,forget plot]
  table[row sep=crcr]{%
1	-40\\
1	40\\
};
\addplot [color=black,solid,forget plot]
  table[row sep=crcr]{%
1.25	-40\\
1.25	40\\
};
\addplot [color=black,solid,forget plot]
  table[row sep=crcr]{%
1.5	-40\\
1.5	40\\
};
\addplot [color=black,solid,forget plot]
  table[row sep=crcr]{%
1.75	-40\\
1.75	40\\
};
\addplot [color=black,solid,forget plot]
  table[row sep=crcr]{%
2	-40\\
2	40\\
};
\addplot [color=black,solid,forget plot]
  table[row sep=crcr]{%
2.25	-40\\
2.25	40\\
};
\addplot [color=black,solid,forget plot]
  table[row sep=crcr]{%
2.5	-40\\
2.5	40\\
};
\addplot [color=black,solid,forget plot]
  table[row sep=crcr]{%
2.75	-40\\
2.75	40\\
};
\addplot [color=black,solid,forget plot]
  table[row sep=crcr]{%
3	-40\\
3	40\\
};
\addplot [color=black,solid,forget plot]
  table[row sep=crcr]{%
3.25	-40\\
3.25	40\\
};
\addplot [color=black,solid,forget plot]
  table[row sep=crcr]{%
3.5	-40\\
3.5	40\\
};
\addplot [color=black,solid,forget plot]
  table[row sep=crcr]{%
3.75	-40\\
3.75	40\\
};
\addplot [color=black,solid,forget plot]
  table[row sep=crcr]{%
4	-40\\
4	40\\
};
\end{axis}
\end{tikzpicture}%